\documentclass[11pt,reqno]{amsart}
\usepackage{caption,amsmath,amssymb,amsfonts,amsthm,latexsym,graphicx,multirow,color,enumerate}
\usepackage{hyperref}
\usepackage{caption,stmaryrd}

\newtheorem{theorem}{Theorem}[section]
\newtheorem{lemma}[theorem]{Lemma}
\newtheorem{proposition}[theorem]{Proposition}
\newtheorem{corollary}[theorem]{Corollary}

\theoremstyle{definition}
\newtheorem{definition}[theorem]{Definition}

\newtheorem{remark}[theorem]{Remark}
\newtheorem{hypothesis}[theorem]{Hypothesis}

\newcommand{\A}{\mathrm{A}}  \newcommand{\Aut}{\mathrm{Aut}}
\newcommand{\bbF}{\mathbb{F}} \newcommand{\bfa}{\alpha} \newcommand{\bfO}{\mathbf{O}}
\newcommand{\C}{\mathrm{C}}   \newcommand{\calE}{\mathcal{E}} \newcommand{\calO}{\mathcal{O}}   \newcommand\calN{\mathcal{N}}
\newcommand{\D}{\mathrm{D}}
\newcommand{\GaL}{\mathrm{\Gamma L}} \newcommand{\GaSp}{\mathrm{\Gamma Sp}} \newcommand{\GaO}{\mathrm{\Gamma O}} \newcommand{\GaU}{\mathrm{\Gamma U}} \newcommand{\GL}{\mathrm{GL}} \newcommand{\GO}{\mathrm{O}} \newcommand{\GU}{\mathrm{GU}}
\newcommand{\la}{\langle}
\newcommand{\M}{\mathrm{M}} \newcommand{\mc}{\mathcal}
\newcommand{\N}{\mathrm{N}} \newcommand{\Nor}{\mathbf{N}}
\newcommand{\Pa}{\mathrm{P}}    \newcommand{\POm}{\mathrm{P\Omega}} \newcommand{\PSL}{\mathrm{PSL}}   \newcommand{\PSU}{\mathrm{PSU}}

\newcommand{\ra}{\rangle} \newcommand\rmO{\mathrm{O}}
\newcommand{\SL}{\mathrm{SL}} \newcommand{\Soc}{\mathrm{Soc}} \newcommand{\Sp}{\mathrm{Sp}} \newcommand{\SU}{\mathrm{SU}} \newcommand{\Sy}{\mathrm{S}}
\newcommand\Tr{\mathrm{Tr}}
\newcommand{\Z}{\mathbf{Z}}

\begin{document}
\title[Solvable factors of almost simple groups]{A complete classification of solvable factors of almost simple groups}

\author[Feng]{Tao Feng}
\address{(Tao Feng) School of Mathematical Sciences\\Zhejiang University\\Hangzhou 310058\\Zhejiang\\P.~R.~China}
\email{tfeng@zju.edu.cn}

\author[Li]{Cai Heng Li}
\address{(Cai Heng Li) Department of Mathematics\\Southern University of Science and Technology\\Shenzhen 518055\\Guangdong\\P.~R.~China}
\email{lich@sustech.edu.cn}

\author[Li]{Conghui Li}
\address{(Conghui Li) School of Mathematics\\Southwest Jiaotong University\\Chengdu 611756\\Sichuan\\P.~R.~China}
\email{liconghui@swjtu.edu.cn}

\author[Wang]{Lei Wang}
\address{(Lei Wang) School of Mathematics and Statistics\\Yunnan University\\Kunming 650091\\Yunnan\\P.~R.~China}
\email{wanglei@ynu.edu.cn}

\author[Xia]{Binzhou Xia}
\address{(Binzhou Xia) School of Mathematics and Statistics\\The University of Melbourne\\Parkville, VIC 3010\\Australia}
\email{binzhoux@unimelb.edu.au}

\author[Zou]{Hanlin Zou}
\address{(Hanlin Zou) School of Mathematics and Statistics\\Yunnan University\\Kunming 650091\\Yunnan\\P.~R.~China}
\email{zouhanlin@ynu.edu.cn}

\begin{abstract}
We give an explicit characterization of solvable factors in factorizations of finite classical groups of Lie type. This completes the classification of solvable factors in factorizations of almost simple groups, finishing the program initiated in [Memoirs of the AMS, 279 (2022), no.~1375] and [Advances in Mathematics, 377 (2021), 107499]. In particular, it resolves the final remaining case in the long-standing problem of determining exact factorizations of almost simple groups. As a byproduct, we obtain a new characterization of one-dimensional transitive groups, offering further insights into their group structures. We also apply our classification to describe quasiprimitive permutation groups with a solvable transitive subgroup, leading to an interesting result that these subgroups are ``small''.


\end{abstract}

\maketitle

\section{Introduction}

Given a group $G$, an expression $G=HK$ with subgroups $H$ and $K$ is called a \emph{factorization} of $G$, where $H$ and $K$ are called \emph{factors}.
The factorization of classical groups of Lie type is an important topic in group theory and has many strong applications, for example,~\cite{GGP2023,GS1995,LX2022,LP,LPS2010}.

A substantial amount of work has been dedicated to the classification of factorizations of classical groups. Notably, a seminal result by Liebeck, Praeger and Saxl in~\cite{LPS1990} classified factorizations with maximal factors, forming the cornerstone for subsequent research~\cite{LWX2023,LX2019,LX2022}. Among various cases where significant advancements were made, a crucial one is the factorizations of classical groups $G$ with a solvable factor $H$.
The description in~\cite{LX2022} provides an upper bound for $H$, while a sharp lower bound for the order $|H|$ is obtained in~\cite{BL2021}. In this paper, we give an explicit classification of the solvable factor $H$.

Note that if $G=HK$ is a factorization of a group $G$, then $G=H^xK^y$ is also a factorization for any elements $x$ and $y$ in $G$. Thus we only describe factorizations up to conjugacy classes of subgroups. If $G$ is an almost simple group, that is, $L\leqslant G\leqslant\Aut(L)$ for some finite nonabelian simple group $L$, then we are only interested in the nontrivial factorizations in the sense that both factors are core-free. The main theorem of this paper is summarized as follows and will be explained in Section~\ref{sec:Xia-11}.

\begin{theorem}\label{thm:Xia-6}
Let $G$ be an almost simple group. Then the solvable subgroups $H$ such that $G=HK$ for some core-free subgroups $K$ of $G$ are explicitly known.
\end{theorem}

A factorization $G=HK$ is called \emph{exact} if $H\cap K=1$. The effort to classify exact factorizations of almost simple groups dates back to 1980 when Wiegold and Williamson~\cite{WW1980} determined the exact factorizations of alternating groups and symmetric groups.
The case where both factors are nonsolvable culminates in~\cite{LWX2023}, with an explicit list of such exact factorizations obtained in~\cite[Table~1]{LWX2023}. For exact factorizations of almost simple groups with a solvable factor, a classification is given in~\cite[Theorem~3]{BL2021}. However, when $G$ is an almost simple group with socle $\Sp_{2m}(q)$ for even $m$, it is not known in~\cite{BL2021} whether $G$ indeed has exact factorizations with a solvable factor.
As a consequence of Theorem~\ref{thm:Xia-6}, this final uncertain case in the classification can be now resolved; see Theorem~\ref{thm:Xia-5}, which shows that no exact factorization arises in this case.
The classification of exact factorizations of almost simple groups now reads as follows.

\begin{theorem}\label{thm:Xia-7}
Let $G$ be an almost simple group with socle $L$, and let $H$ and $K$ be core-free subgroups of $G$ such that $G=HK$ and $H\cap K=1$. Then one of the following holds:
\begin{enumerate}[\rm(a)]
\item $L=\A_n$, $H$ is transitive on $\{1,\ldots,n\}$, and $K=\A_{n-1}$ or $\Sy_{n-1}$;
\item $L=\A_n$ with $n=q$ for some prime power $q$, $H\leqslant\mathrm{A\Gamma L}_1(q)$ is $2$-homogeneous on $\{1,\ldots,n\}$, and $\A_{n-2}\trianglelefteq K\leqslant\Sy_{n-2}\times\Sy_2$;
\item $L=\A_n$ with $n=q+1$ for some prime power $q$, $H$ is $3$-transitive on $\{1,\ldots,n\}$ with socle $\PSL_2(q)$, and $K=\A_{n-3}$ or $\A_{n-3}.2$;
\item $L=\PSL_n(q)$, $H\cap L\leqslant\frac{q^n-1}{(q-1)\gcd(n,q-1)}{:}n$, and $K\cap L\trianglerighteq q^{n-1}{:}\SL_{n-1}(q)$;
\item $L=\Sp_{2m}(q)$ with $q$ even and $m\geqslant3$ odd, $H\cap L\leqslant q^m{:}(q^m-1).m$, and $K\cap L=\Omega_{2m}^-(q)$;
\item $(G,H,K)$ is one of the finitely many triples in~\cite[Table~4]{BL2021} or~\cite[Rows~4--23~of~Table~1]{LWX2023}.
\end{enumerate}
\end{theorem}

A transitive permutation group $G\leqslant\mathrm{Sym}(\Omega)$ is said to be \emph{primitive} if $\Omega$ admits no nontrivial $G$-invariant partition. The study of primitive groups $G$ containing a certain transitive subgroup $H$ dates back to Burnside's 1900 paper~\cite{Burnside1900} and has played a significant role in the development of permutation group theory, for which the reader is referred to~\cite[Problem~3]{Neumman1994}. Various classification results have been obtained for $H$ metacyclic~\cite{Jones1972,LPX2021,Nagai1961,Scott1957} or nilpotent~\cite{BL2021,Li2003}. In many applications, however, the permutation group $G$ is only required to be \emph{quasiprimitive}, meaning that every nontrivial normal subgroup of $G$ is transitive. For instance, quasiprimitive groups containing a metacyclic subgroup were classified in~\cite{LPX2021} towards a characterization of metacirculants.

With the aid of Theorem~\ref{thm:Xia-6}, we are able to describe quasiprimitive groups containing a solvable subgroup, as stated in Theorem~\ref{thm:QP} below. According to the O'Nan--Scott--Praeger theorem, finite quasiprimitive groups fall into eight types, as described in~\cite[Section~5]{Praeger1997}, and we adopt the terminology used therein.
The proof of Theorem~\ref{thm:QP} will be given in Section~\ref{sec:app}.

\begin{theorem}\label{thm:QP}
Let $G\leqslant\mathrm{Sym}(\Omega)$ be a finite quasiprimitive group with a solvable transitive subgroup $H$. Then the pair $(G,H)$ satisfies one of the following:
\begin{enumerate}[\rm(a)]
\item\label{thm:QPa} $G$ is primitive of type \textup{HA};
\item\label{thm:QPb} $G$ is almost simple with socle $L$ and point-stabilizer $K$ such that $G=HK$, the pair $(L,K\cap L$ is known by~\cite[Theorem~1.1]{LX2022}, and $H$ is known by Theorem~$\ref{thm:Xia-6}$;
\item\label{thm:QPc} $G$ is primitive of type \textup{HS} or \textup{SD}, $\Soc(G)=L^2$ with
\[
L=\PSL_2(q),\ \PSL_3(3),\ \PSL_3(4),\ \PSL_3(8),\ \PSU_3(8),\ \PSU_4(2)\text{ or }\mathrm{M}_{11},
\]
where $q\geqslant4$ is a prime power, and $H\cap L^2\leqslant M_1\times M_2$ such that $(L,M_1,M_2)$ lies in Table~$\ref{TabDiagonal}$;
\item\label{thm:QPd} $G$ is of type \textup{HC} or \textup{CD}, and $G\leqslant G_0\wr\Sy_k$ in product action for some permutation group $G_0$ with $\Soc(G)=\Soc(G_0)^k$ and $H\cap G_0^k\leqslant H_1\times\cdots\times H_k$ such that each $(G_0,H_i)$ is a pair $(G,H)$ in~\emph{\eqref{thm:QPc}};
\item\label{thm:QPe} $G$ is of type \textup{PA}, and there is a faithful action $\psi$ of $G$ on some $G$-invariant partition of $\Omega$ such that $G^\psi\leqslant G_0\wr\Sy_k$ in product action for some permutation group $G_0$ with $\Soc(G^\psi)=\Soc(G_0)^k$ and $H^\psi\cap G_0^k\leqslant H_1\times\cdots\times H_k$ such that each $(G_0,H_i)$ is a pair $(G,H)$ in~\emph{\eqref{thm:QPb}}.
\end{enumerate}
\end{theorem}

\begin{table}[htbp]
\caption{The triple $(L,M_1,M_2)$ in Theorem~\ref{thm:QP}}\label{TabDiagonal}
\centering
\begin{tabular}{|l|l|l|l|}
\hline
Row & $L$ & $M_1$ & $M_2$ \\
\hline
1 & $\PSL_2(q)$ & $\D_{2(q+1)/\gcd(2,q-1)}$ & $q{:}((q-1)/\gcd(2,q-1))$ \\
2 & $\PSL_2(7)$ & $7{:}3$ & $\Sy_4$ \\
3 & $\PSL_2(11)$  & $11{:}5$ & $\A_4$ \\
4 & $\PSL_2(23)$ & $23{:}11$ & $\Sy_4$ \\
5 & $\PSL_3(3)$ & $13{:}3$ & $3^2{:}2.\Sy_4$ \\
6 & $\PSL_3(4)$ & $7{:}3$ & $2^4{:}\D_{10}$ \\
7 & $\PSL_3(8)$ & $73{:}9$ & $2^{3+6}{:}7^2$ \\
8 & $\PSU_3(8)$ & $19{:}3$ & $2^{3+6}{:}21$ \\
9 & $\PSU_4(2)$ & $2^4{:}\D_{10}$ & $3_+^{1+2}{:}2.(\A_4)$, $3^3{:}\Sy_4$ \\
10 & $\M_{11}$ & $11{:}5$ & $\M_9.2$ \\
\hline
\end{tabular}
\end{table}

We call a group \emph{alternating-free} if it does not have $\A_m$ as a composition factor for any $m\geqslant4$.
As a corollary of Theorem~\ref{thm:QP}, the following result (see Section~\ref{sec:app} for its proof) says that, if $G\leqslant\Sy_n$ is quasiprimitive and alternating-free, then each solvable transitive subgroup of $G$ has order bounded above by a quasi-polynomial of $n$. Recall that a \emph{minimally transitive} group is a transitive permutation group such that none of its proper subgroups is transitive.

\begin{corollary}\label{thm:small}
Let $G$ be an alternating-free quasiprimitive permutation group on $n$ points, and let $H$ be a solvable transitive subgroup of $G$. Then
the following statements hold:
\begin{enumerate}[\rm(a)]
\item\label{thm:smalla} $\ln|H|=O((\ln n)^\alpha)$ for some absolute constant $\alpha<9/4$;
\item if $G$ is almost simple and primitive and $H$ is minimally transitive, then either $\ln|H|=O(\ln n\ln\ln n/\ln\ln\ln n)$ or $\Soc(G)=\Omega_{2m+1}(q)$ with $q$ odd.
\end{enumerate}
\end{corollary}

\begin{remark}
The upper bound for $|H|$ in Corollary~\ref{thm:small}\,\eqref{thm:smalla} can be made better if $G$ is not primitive of type HA; see the proof of Corollary~\ref{thm:small}.
The case $\Soc(G)=\Omega_{2m+1}(q)$ with $q$ odd is a genuine exception for $\ln|H|=O(\ln n\ln\ln n/\ln\ln\ln n)$; see Remark~\ref{rem:Xia-5}.
The condition ``alternating-free'' cannot be removed from Corollary~\ref{thm:small}. Otherwise we could have counterexamples such as $G=\Sy_n$ for some $4$-power $n$ and $H$ being the imprimitive wreath product of $\log_4n$ copies of $\Sy_4$, which would give $|H|=24^{(n-1)/3}$. In fact, this order is largest possible for a solvable subgroup $H$ of $\Sy_n$, as proved by Dixon~\cite[Theorem~3]{Dixon1967}.
Based on this counterexamples, one may also construct PA type $G$ with socle $\A_m^k$ and $|H|=24^{k(m-1)/3}$, where $n=m^k$ and $m$ is a $4$-power, such that the growth of $|H|$ exceeds quasi-polynomials of $n$ for any fixed $k$.
\end{remark}

\section{The classification}\label{sec:Xia-11}

In this section we demonstrate Theorem~\ref{thm:Xia-6} in detail. The factorizations of non-classical almost simple groups are classified in the literature~\cite{Giudici2006,HLS1987,LPS1990}. In particular, those with a solvable factor are listed in~\cite[Theorem~1.1]{LX2022}. The factorizations $G=HK$ of almost simple classical groups $G$ with solvable $H$ and core-free $K$ are described in Tables~1.1 and~1.2 of~\cite{LX2022}, where Table~1.2 is an explicit list of small exceptions. Adopting the notation in~\cite[\S2.1]{LX2022}, in Table~\ref{TabLX2022} we present~\cite[Table~1.1]{LX2022} but replace the triple $(G,H,K)$ by the corresponding one in the classical group with scalars.
%so that $G$ satisfies one of:
%\begin{enumerate}[(i)]
%\item $\SL_m(q)=L\leqslant G\leqslant\GaL_m(q)$ with $m\geqslant2$ and $(m,q)\neq(2,2)$ or $(2,3)$;
%\item $\SU_{2m}(q)=L\leqslant G\leqslant\GaU_{2m}(q)$ with $m\geqslant2$;
%\item $\Sp_{2m}(q)=L\leqslant G\leqslant\GaSp_{2m}(q)$ with $m\geqslant2$ and $(m,q)\neq(2,2)$;
%\item $\Omega_{2m+1}(q)=L\leqslant G\leqslant\GaO_{2m+1}(q)$ with $q$ odd and $m\geqslant3$;
%\item $\Omega_{2m}^+(q)=L\leqslant G\leqslant\GaO_{2m}^+(q)$ with $m\geqslant4$.
%\end{enumerate}

\begin{table}[htbp]
\captionsetup{justification=centering}
\caption{Infinite families of $(G,H,K)$ from~\cite{LX2022}}\label{TabLX2022}
\begin{tabular}{|c|l|l|l|l|l|}
\hline
Row & $L$ & $H\cap L\leqslant$ & $K\cap L\trianglerighteq$ & Condition\\
\hline
1 & $\SL_m(q)$ & $\GL_1(q^m){:}m$ & $q^{m-1}{:}\SL_{m-1}(q)$ & \\
2 & $\SL_4(q)$ & $q^3{:}(q^3-1).3<\Pa_k$ & $\Sp_4(q)$ & $k\in\{1,3\}$ \\
3 & $\Sp_{2m}(q)$ & $q^{m(m+1)/2}{:}(q^m-1).m<\Pa_m$ & $\Omega_{2m}^-(q)$ & $q$ even \\
4 & $\Sp_4(q)$ & $q^3{:}(q^2-1).2<\Pa_1$ & $\Sp_2(q^2)$ & $q$ even \\
5 & $\Sp_4(q)$ & $q^{1+2}{:}(q^2-1).2<\Pa_1$ & $\Sp_2(q^2)$ & $q$ odd \\
6 & $\SU_{2m}(q)$ & $q^{m^2}{:}(q^{2m}-1).m<\Pa_m$ & $\SU_{2m-1}(q)$ & \\
7 & $\Omega_{2m+1}(q)$ & $(q^{m(m-1)/2}.q^m){:}(q^m-1).m<\Pa_m$ & $\Omega_{2m}^-(q)$ & $q$ odd\\
8 & $\Omega_{2m}^+(q)$ & $q^{m(m-1)/2}{:}(q^m-1).m<\Pa_k$ & $\Omega_{2m-1}(q)$ & $k\in\{m,m-1\}$ \\
9 & $\Omega_8^+(q)$ & $q^6{:}(q^4-1).4<\Pa_1$ & $\Omega_7(q)$ & \\
\hline
\end{tabular}
\vspace{3mm}
\end{table}

The solvable factor $H$ in rows~2,~5 and~7 of Table~\ref{TabLX2022} is well understood in~\cite{BL2021} (see~\cite[Remark~2]{BL2021}).
For row~4, since the graph automorphism $\gamma$ of $L$ does not normalize $\Sp_2(q^2)$, we have $G\leqslant\GaSp_4(q)$, and so the triple $(G^\gamma,H^\gamma,K^\gamma)$ is in row~3 with $m=2$. Similarly, the characterization of row~9 is reduced to row~8 with $m=4$ by the triality automorphism of $\POm_8^+(q)$, and in row~1 we may assume $K\leqslant\Pa_1[G]$ by applying the transpose-inverse if necessary.
Thus the only rows in Table~\ref{TabLX2022} to discuss are~1,~3,~6 and~8, and we make the following hypothesis accordingly.
Here we call a subgroup of $G$ $\max^-$ if it is a maximal one among subgroups of $G$ not containing $L$.

\begin{hypothesis}\label{hypo-1}
Let $G$ be a classical group of Lie type, let $L=G^{(\infty)}$, let $H$ be a solvable subgroup of $G$, and let $B^{(\infty)}\leqslant K\leqslant B$ with $B$ $\max^-$ in $G$ such that one of the following holds:
\begin{enumerate}[(i)]
\item\label{hypo-1i} $\SL_m(q)=L\leqslant G\leqslant\GaL_m(q)$ with $m\geqslant2$, $H\leqslant\GaL_1(q^m)$, and $B=\Pa_1[G]$;
\item\label{hypo-1ii} $\SU_{2m}(q)=L\leqslant G\leqslant\GaU_{2m}(q)$ with $m\geqslant2$ and $(m,q)\neq(2,2)$ or $(3,2)$, $H\leqslant\Pa_m[G]$, and $B=\N_1[G]$;
\item\label{hypo-1iii} $\Omega_{2m}^+(q)=L\leqslant G\leqslant\GaO_{2m}^+(q)$ with $m\geqslant4$ and $(m,q)\neq(4,3)$ or $(6,2)$, $H\leqslant\Pa_m[G]$, and $B=\N_1[G]$;
\item\label{hypo-1iv} $\Sp_{2m}(q)=L\leqslant G\leqslant\GaSp_{2m}(q)$ with $q$ even and $m\geqslant2$ and $(m,q)\neq(2,8)$ or $(6,2)$, $H\leqslant\Pa_m[G]$, and $B\cap L=\mathrm{O}_{2m}^-(q)$.
\end{enumerate}
\end{hypothesis}

\begin{remark}
As explained in the paragraph preceding Hypothesis~\ref{hypo-1}, to classify the factorizations $G=HK$ in Table~\ref{TabLX2022}, we only need to consider rows~1,~3,~6 and~8, which correspond respectively to~\eqref{hypo-1i},~\eqref{hypo-1iv},~\eqref{hypo-1ii} and~\eqref{hypo-1iii} of Hypothesis~\ref{hypo-1}. To be more precise, we make the following clarifications on Hypothesis~\ref{hypo-1}:
\begin{enumerate}[(I)]
\item Some pairs of $(m,q)$ are excluded in the hypothesis for simplicity of subsequent argument. For these pairs, the groups $G$ are small enough that all the factorizations $G=HK$ can be found using \textsc{Magma}~\cite{BCP1997}.
\item For $L=\Omega_{2m}^+(q)$ as in row~8 of Table~\ref{TabLX2022}, one should assume $L\leqslant G\leqslant\C\GaO_{2m}^+(q)$ so that $G/\Z(G)$ runs over almost simple groups with socle $\POm_{2m}^+(q)$ that does not involve the triality automorphism (if $m=4$). However, since $K$ is always contained in $\GaO_{2m}^+(q)$, denoting $G_0=G\cap\GaO_{2m}^+(q)$ and $H_0=H\cap\GaO_{2m}^+(q)$, we have $G=HK$ if and only if $G=G_0H$ and $G_0=H_0K$. Thus the classification of $G=HK$ is reduced to that of $G_0=H_0K$, and so we assume $L\leqslant G\leqslant\GaO_{2m}^+(q)$ in Hypothesis~\ref{hypo-1}\,\eqref{hypo-1iii}. This, again, streamlines the argument of this paper.
\end{enumerate}
\end{remark}

Characterizing the solvable factor $H$ in Hypothesis~\ref{hypo-1}\,\eqref{hypo-1i} is relatively easy; see Subsection~\ref{sec:Xia-8}. For~\eqref{hypo-1ii}--\eqref{hypo-1iv} of Hypothesis~\ref{hypo-1}, let $\widehat G$, $\widehat H$ and $\widehat K$ be overgroups of $G$, $H$ and $K$, respectively, as listed in the Table~\ref{TabHat}, where $q=p^f$ with prime $p$ and integer $f$, and $d=\gcd(2,q-1)$.

\begin{table}[htbp]
\captionsetup{justification=centering}
\caption{The triple $(\widehat{G},\widehat{H},\widehat{K})$ and the parameter $s$}\label{TabHat}
\begin{tabular}{|l|l|l|l|l|l|}
\hline
$\widehat{G}$ & $\widehat{H}$ & $\widehat{K}$ & $\widehat{H}\cap\widehat{K}$ & $s$\\
\hline
$\GaU_{2m}(q)$ & $q^{m^2}{:}\GaL_1(q^{2m})$ & $\SU_{2m-1}(q).(q+1)^2.(2f)$ & $q^{(m-1)^2}.(q+1).(2mf)$ & $2$\\
$\GaO_{2m}^+(q)$ & $q^{m(m-1)/2}{:}\GaL_1(q^m)$ & $\Omega_{2m-1}(q).2^{2d-1}.f$ & $q^{(m-1)(m-2)/2}.d^2.(mf)$ & $1$\\
$\GaSp_{2m}(q)$ & $q^{m(m+1)/2}{:}\GaL_1(q^m)$ & $\Omega_{2m}^-(q).(2f)$ & $q^{m(m-1)/2}.2.(mf)$ & $1$\\
\hline
\end{tabular}
\vspace{3mm}
\end{table}

Then $\widehat G=\widehat H\widehat K$ with $\widehat H\cap\widehat K$ described in the fourth column of the table, as implied by~\cite{LX2022} (refer to its Proposition~5.2 for unitary groups, Proposition~5.9 for orthogonal groups, and Proposition~5.5 for symplectic groups). Define a parameter $s=2$ for the unitary case and $s=1$ for the orthogonal and symplectic cases.
Note that $\widehat H=R{:}T$, where
\[
R=\bfO_p(\widehat H)\ \text{ and }\ T=\GaL_1(q^{sm}).
\]
To determine the solvable factor $H=(H\cap R).S$ (this is in fact a split extension, see Lemma~\ref{lem:Xia-7}) in the factorization $G=HB$, we need to determine both $H\cap R$ and the subgroup $S$ of $T$ with certain transitivity, which are given in the following two subsections. In the final subsection of this Introduction, we also describe how to determine the factor $K$ in the factorization $G=HK$.

\subsection{One-dimensional transitive groups}\label{sec:Xia-8}
\ \vspace{1mm}

Let $W$ be a vector space over a finite field $\bbF$, and let $i$ be a positive divisor $i$ of $|\bbF^\times|$. The subgroup $C_{(i)}$ of order $i$ in $\bbF^\times$ partitions each $1$-subspace of $W$ into $|\bbF^\times|/i$ orbits, each consisting of $i$ nonzero vectors.
Let $W_{(i)}$ be the set of $C_{(i)}$-orbits on $W\setminus\{0\}$, that is,
\[
W_{(i)}=\{w^{C_{(i)}}\mid w\in W\setminus\{0\}\}.
\]
In particular, $W_{(1)}=W\setminus\{0\}$, and $W_{(\gcd(2,q-1))}=\{\{\pm w\}\mid w\in W\setminus\{0\}\}$ if $\bbF=\bbF_q$.
The solvable factors $H$ in Hypothesis~\ref{hypo-1}\,\eqref{hypo-1i} are exactly subgroups of $G$ that are transitive on $(\bbF_q^m)_{(q-1)}$. Let $Z$ be the center of $\GL_m(q)$. Then $H$ is transitive on $(\bbF_q^m)_{(q-1)}$ if and only if $ZH$ is transitive on $\bbF_q^m\setminus\{0\}$. Thus a classification of transitive subgroup of $\GaL_1(q^m)$ on $\bbF_q^m\setminus\{0\}$ is needed.

For simplicity, we replace $q^m$ by the prime power $q$, where $q=p^f$, and consider transitive subgroup of $\GaL_1(p^f)$ on the set of non-zero vectors of $\bbF_p^f$. In fact, seeking a satisfactory description of transitive subgroups of $\GaL_1(p^f)$ is crucial not only in the determination of solvable factors of classical groups of Lie type, but also in the study of some problems in permutation groups and symmetrical graphs~\cite{Foulser1964,FK1978,Hering1985,LLP2009}.

Let $n$ be a positive integer, and let $\sigma$ be a set of primes. Denote the set of prime divisors of $n$ by $\pi(n)$. Let $n_\sigma$ be the $\sigma$-part of $n$, that is, largest divisor $\ell$ of $n$ such that $\pi(\ell)\subseteq\sigma$. Let $n_{\sigma'}$ be the $\sigma'$-part of $n$, that is, $n_{\sigma'}=n/n_\sigma$. For a cyclic group $C$ of order $n$, let $C_\sigma$ denote the unique subgroup of order $n_\sigma$ in $C$, and $C_{\sigma'}$ denote the unique subgroup of order $n_{\sigma'}$ in $C$. When $\sigma=\{r\}$ consists of only one prime $r$, we simply write the subscripts $\sigma$ and $\sigma'$ as $r$ and $r'$, respectively.
%As usual, $\C_n$ and $\Q_{4n}$ denote the cyclic group of order $n$ and generalized quaternion group of order $4n$, respectively.

Now let $\omega$ be a generator of $\bbF_{p^f}^\times$, let $a\in\GL_1(p^f)$ be the right multiplication of $\omega$, and let $\varphi$ be the Frobenius of $\bbF_{p^f}$ taking $p$-th power. Then
\[
\GaL_1(p^f)=\la a\ra{:}\la\varphi\ra=\C_{p^f-1}{:}\C_f
\]
such that $\la a\ra=\GL_1(p^f)$ and $a^\varphi=a^p$. We characterize the transitive subgroups of $\GaL_1(p^f)$ on $\bbF_{p^f}^\times$ as follows.

\begin{theorem}\label{thm:trans}
A subgroup $H$ of $\GaL_1(p^f)=\la a\ra{:}\la\varphi\ra$ is transitive on $\bbF_{p^f}^\times$ if and only if $H=\la a\ra_{\sigma'}{:}(\la b\ra\la c\varphi^k\ra)$ for some divisor $k$ of $f$, subset $\sigma$ of $\pi(p^k-1)\cap\pi(f)$, element $b$ in $\la a\ra_\sigma$ and generator $c$ of $\la a\ra_\sigma$ such that, if $2\in\sigma$ and $p^k\equiv3\pmod{4}$, then $|b|_2=|a|_2/2$.
\end{theorem}

The following corollary of Theorem~\ref{thm:trans} is an explicit description of the minimally transitive subgroups of $\GaL_1(p^f)$ on $\bbF_{p^f}^\times$.

\begin{corollary}\label{thm:min}
A subgroup $H$ of $\GaL_1(p^f)=\la a\ra{:}\la\varphi\ra$ is minimally transitive on $\bbF_{p^f}^\times$ if and only if there exist some divisor $k$ of $f$, subset $\sigma$ of $\pi(p^k-1)\cap\pi(f)$ and generator $c$ of $\la a\ra_\sigma$ such that $k$ is divisible by $f_{\sigma'}$ and one of the following holds:
\begin{enumerate}[\rm(a)]
\item\label{thm:mina} $H=\la a\ra_{\sigma'}{:}\la c\varphi^k\ra$ with either $2\notin\sigma$ or $p^k\equiv1\pmod{4}$;
\item\label{thm:minb} $H=\la a\ra_{\sigma'}{:}(\la a^2\ra_2\la c\varphi^k\ra)$ with $2\in\sigma$ and $p^k\equiv3\pmod{4}$.
\end{enumerate}
Moreover, $H$ is regular on $\bbF_{p^f}^\times$ if and only if either case~\eqref{thm:mina} holds, or $f\equiv2\pmod{4}$ and case~\eqref{thm:minb} holds; in particular, a Sylow subgroup of a regular subgroup of $\GaL_1(p^f)$ on $\bbF_{p^f}^\times$ is either cyclic or a generalized quaternion $2$-group.
\end{corollary}

\begin{remark}
The above ``Moreover'' part is a more group-theoretic approach to the classification of sharply-$2$-transitive subgroups of the one-dimensional affine linear group, rather than the classical approach of nearfields (see, for example,~\cite[\S7.6]{DM1996}).
\end{remark}

Theorem~\ref{thm:trans} and Corollary~\ref{thm:min} will be proved in Section~\ref{sec:Xia-2}, where we also develop a new way to represent subgroups of $\GaL_1(p^f)$ by a triple of parameters, called a \emph{Foulser triple} (see Definition~\ref{def}).
In fact, by definition, for a subgroup $H$ of $\GaL_1(p^f)$ with Foulser triple $(\ell,j,k)$, we have $H=\la a^\ell\ra\la a^j\varphi^k\ra$ with $H\cap\la a\ra=\la a^\ell\ra$ and $j$ divisible by $(p^f-1)_{\sigma'}$, where $\sigma=\pi(\ell)$.

A direct application of Theorem~\ref{thm:trans} is an explicit classification (see Remark~\ref{rem:Xia-1}) for $H$ in Hypothesis~\ref{hypo-1}\,\eqref{hypo-1i} such that $G=HB$.

\subsection{Solvable factors}\label{sec:Xia-9}
\ \vspace{1mm}

From now on we turn to the solvable factor $H$ in~\eqref{hypo-1ii}--\eqref{hypo-1iv} of Hypothesis~\ref{hypo-1}. There is a standard basis $e_1,f_1,\dots,e_m,f_m$ as in~\cite[2.2.3]{LPS1990}, so that, letting $U$ be the $\bbF_{q^s}$-space $\la e_1,\ldots,e_m\ra$, we have $H\leqslant G_U=\Pa_m[G]$ and that $R$ is the unipotent radical of $G_U$.
Moreover, with $W$ denoting the $\bbF_{q^s}$-space $\la f_1,\ldots,f_m\ra$, the group $T=\GaL_1(q^{sm})$ stabilizes both $U$ and $W$ such that $T\cap G$ is the Levi subgroup of $\Pa_m[G]$. According to Table~\ref{TabLX2022}, we may write $H=P{:}S$ with $P\leqslant R$ and $S\leqslant T$, and view $P$ as a submodule of the $S$-module $R$.

In Subsection~\ref{subsec_form} we decompose the $S$-module $R$ as
\[
R=
\begin{cases}
U(1)\oplus \ldots\oplus U(\lfloor(m+s-1)/2\rfloor)&\text{if }L=\SU_{2m}(q)\text{ or }\Omega_{2m}^+(q),\\
U(0)\oplus U(1)\oplus\ldots\oplus U(\lfloor m/2\rfloor)&\text{if }L=\Sp_{2m}(q),
\end{cases}
\]
where $U(i)$'s are pairwise non-isomorphic irreducible $S$-submodules described in Corollary~\ref{Cor_UiIrr} (see Remark~\ref{rem:Xia-3} for an explanation of the parameter $i$ to label the submodule $U(i)$). In particular,
\begin{equation}\label{eq:Ui}
U(i)=
\begin{cases}
q^{sm}&\text{if }0\leqslant i\leqslant\left\lfloor(m+s)/2\right\rfloor-1,\\
q^{sm/2}&\text{if }i=(m+s-1)/2\text{ with $m+s$ odd}.
\end{cases}
\end{equation}
Viewing $P$ as an $S$-submodule of $R$, we may express it as a direct sum of $U(i)$'s with $i$ running over some set $I$ of indices, that is,
\[
P=U(I):=\mbox{$\bigoplus\limits_{i\in I}U(i)$}
\]
for some subset $I$ of $\{0,1,\ldots,\lfloor(m+s-1)/2\rfloor\}$ such that $0\in I$ only if $L=\Sp_{2m}(q)$. Accordingly,
\begin{equation}\label{eq:UI}
U(I)=
\begin{cases}
q^{sm|I|}&\text{if }(m+s-1)/2\notin I,\\
q^{sm(2|I|-1)/2}&\text{if }(m+s-1)/2\in I.
\end{cases}
\end{equation}

It is already known in~\cite{BL2021} that the possible elementary abelian groups $P=q^c$ in a solvable factor $H=P{:}S$ are precisely those with $c$ satisfying
\[
\begin{cases}
2m\leqslant c\leqslant m^2\text{ and }c\text{ divisible by }m\gcd(2,m)&\text{if }L=\SU_{2m}(q),\\
m\leqslant c\leqslant m(m-1)/2\text{ and }c\text{ divisible by }m/\gcd(2,m)&\text{if }L=\Omega_{2m}^+(q),\\
m\leqslant c\leqslant m(m+1)/2\text{ and }c\text{ divisible by }m/\gcd(2,m)&\text{if }L=\Sp_{2m}(q).
\end{cases}
\]
However, our aim is to determine exactly which subgroups $P$ of $R$, or equivalently, which submodules $U(I)$, can occur in a solvable factor $H=P{:}S$. Let $I\setminus\{0\}=\{i_1,\dots,i_k\}$. The condition for $H=P{:}S=U(I){:}S$ to be a solvable factor depends heavily on the parameter $d(I)$ defined by
\begin{equation}\label{eq:d}
d(I)=
\begin{cases}
\gcd(2i_1-1,\ldots,2i_k-1,m)&\textup{if $s=2$},\\
\gcd(i_1,\ldots,i_k)&\textup{if $s=1$ and $m/2\in I$},\\
\gcd(i_1,\ldots,i_k,m)&\textup{if $s=1$ and $m/2\notin I$}.
\end{cases}
\end{equation}
With the above notation of $U(I)$, $d(I)$, $T$ and $W_{(i)}$ for a vector space $W$ over some finite field, and with the result of Subsection~\ref{sec:Xia-8} in mind, we are now ready to state the classification theorems for the solvable factor $H$ in Hypothesis~\ref{hypo-1}\,\eqref{hypo-1ii}--\eqref{hypo-1iv}.

\begin{theorem}\label{thm:Xia-2}
Suppose Hypothesis~$\ref{hypo-1}$\,\eqref{hypo-1ii} so that $L=\SU_{2m}(q)$, $H\leqslant\Pa_m[G]$ and $B=\N_1[G]$. Then $G=HB$ if and only if $H=U(I){:}S$ with $I\subseteq\{1,\ldots,\lfloor(m+1)/2\rfloor\}$ and $S\leqslant\GaL_1(q^{2m})$ such that $d(I)=1$ and $S$ is transitive on $(\bbF_{q^2}^m)_{(q+1)}$.
\end{theorem}

\begin{theorem}\label{thm:Xia-3}
Suppose Hypothesis~$\ref{hypo-1}$\,\eqref{hypo-1iii} so that $L=\Omega_{2m}^+(q)$, $H\leqslant\Pa_m[G]$ and $B=\N_1[G]$, and let $S_0=\GL_1(q^m)$ be the Singer group in $T=\GaL_1(q^m)$. Then $G=HB$ if and only if $H=U(I){:}S$ with $I\subseteq\{1,\ldots,\lfloor m/2\rfloor\}$ and $S\leqslant T$ such that one of the following holds:
\begin{enumerate}[\rm(a)]
\item\label{thm:Xia-3a} $d(I)=1$, and $S$ is transitive on $(\bbF_q^m)_{(\gcd(2,q-1))}$;
\item\label{thm:Xia-3b} $d(I)=2$, $q=2$, and $S$ is transitive on $\bbF_q^m\setminus\{0\}$ with $|T|/|SS_0|$ odd;
\item\label{thm:Xia-3c} $d(I)=2$, $q=4$, $G\geqslant L.2$ with $G\neq\mathrm{O}_{2m}^+(4)$, and $S$ is transitive on $\bbF_q^m\setminus\{0\}$ with $|T|/|SS_0|$ odd.
\end{enumerate}
\end{theorem}

\begin{theorem}\label{thm:Xia-4}
Suppose Hypothesis~$\ref{hypo-1}$\,\eqref{hypo-1iv} so that $L=\Sp_{2m}(q)$, $H\leqslant\Pa_m[G]$ and $B\cap L=\mathrm{O}_{2m}^-(q)$, and let $S_0=\GL_1(q^m)$ be the Singer group in $T=\GaL_1(q^m)$. Then $G=HB$ if and only if $H=U(I){:}S$ with $I\subseteq\{0,1,\ldots,\lfloor m/2\rfloor\}$ and $S\leqslant T$ such that $S$ is transitive on $\bbF_q^m\setminus\{0\}$ and one of the following holds:
\begin{enumerate}[\rm(a)]
\item\label{thm:Xia-4d} $0\in I$;
\item\label{thm:Xia-4a} $0\notin I$, $d(I)=1$, and $q=2$;
\item\label{thm:Xia-4b} $0\notin I$, $d(I)=2$, $q=2$, and $|T|/|SS_0|$ is odd;
\item\label{thm:Xia-4c} $0\notin I$, $d(I)=1$, $q=4$, $G=\GaSp_{2m}(4)$, and $|T|/|SS_0|$ is odd.
\end{enumerate}
\end{theorem}

\begin{remark}
We remark that the subgroups $S$ of $T$ in Theorems~\ref{thm:Xia-2}--\ref{thm:Xia-4} are characterized by Theorem~\ref{thm:trans}. In fact, letting $W=\bbF_{q^s}^m$ and letting $i$ be a divisor of $q^s-1$, the subgroups $S$ that are transitive on $W_{(i)}$ are precisely those such that $Z_{(i)}S$ is transitive on $W\setminus\{0\}$, where $Z_{(i)}$ is the subgroup of order $i$ in the center of $\GL(W)$. Further, if we represent $S$ by the Foulser triple mentioned in Subsection~\ref{sec:Xia-8}, then the conditions on $S$ are equivalent to the arithmetic conditions on these parameters in Remark~\ref{rem:Xia-2}.
\end{remark}

The solvable factors $H$ in Theorem~\ref{thm:Xia-2}, Theorem~\ref{thm:Xia-3}\,\eqref{thm:Xia-3a} or Theorem~\ref{thm:Xia-4}\,\eqref{thm:Xia-4d} are called \emph{basic}, as otherwise $H$ is a factor in a factorization $A=H(A\cap B)$ for some $\max^-$ subgroup $A$ of $G$. To be precise, $A$ is the $\max^-$ subgroup of $G$ such that $A\cap L=\Omega_m^+(q^2).2^2$ in Theorem~\ref{thm:Xia-3}\,\eqref{thm:Xia-3b}--\eqref{thm:Xia-3c} and $A\cap L=\mathrm{O}_{2m}^+(q)$ in Theorem~\ref{thm:Xia-4}\,\eqref{thm:Xia-4a}--\eqref{thm:Xia-4c} (see the proof of Theorems~\ref{thm:Xia-3} and~\ref{thm:Xia-4} for more details). For these non-basic factors $H$, the factorization $G=HB$ is in one-to-one correspondence to the factorization $A=H(A\cap B)$ of $A$.

For basic solvable factors $H$, the necessity of the conditions in Theorems~\ref{thm:Xia-2}--\ref{thm:Xia-4} is relatively easy to see (through the propositions in Section~\ref{sec:Xia-1} together with Proposition~\ref{thm:dValue}). For example, if $d(I)=d>1$ in the unitary case, then $H$ would be contained in some field extension subgroup of type $\GU_{2m/d}(q^d)$, not possible by~\cite[Theorem~A]{LPS1990}. The essential part to prove the sufficiency of these conditions is done in Subsection~\ref{subsec_UIorb}.
Then in Section~\ref{sec:Xia-6} we conclude the proof of Theorems~\ref{thm:Xia-2}--\ref{thm:Xia-4}.

Based on Theorems~\ref{thm:Xia-2}--\ref{thm:Xia-4}, we also prove the following corollary in Section~\ref{sec:Xia-6}. Let $q=p^f$ with prime $p$ and positive integer $f$, and let $\phi$ be the field automorphism of order $sf$ taking $p$-th power. Note that
\[
\GaU_{2m}(q)/\SU_{2m}(q)=\la\delta\ra{:}\la\phi\ra=\la\delta,\phi\mid\delta^{q+1}=\phi^{2f}=1,\,\delta^\phi=\delta^p\ra=\C_{q+1}{:}\C_{2f}
\]
with $\la\delta\ra=\GU_{2m}(q)/\SU_{2m}(q)$, and that
\[
\GaO_{2m}^+(q)/\Omega_{2m}^+(q)=\la\delta'\ra\times\la\delta''\ra\times\la\phi\ra=\C_2\times\C_{\gcd(2,q-1)}\times\C_f
\]
with $\la\delta'\ra=\mathrm{SO}_{2m}^+(q)/\Omega_{2m}^+(q)$ and $\la\delta''\ra=\mathrm{GO}_{2m}^+(q)/\mathrm{SO}_{2m}^+(q)$.

\begin{corollary}\label{cor_G0}
Let $\phi$, $\delta$, $\delta'$ and $\delta''$ be as above, and let $G$ and $B$ be as in Hypothesis~$\ref{hypo-1}$\,\eqref{hypo-1ii}--\eqref{hypo-1iv}. Then there exists some $H$ as in Hypothesis~$\ref{hypo-1}$ such that $G=HB$ if and only if $G$ lies in the following table,
\[
\begin{array}{lll}
\hline
L & G & \textup{Conditions}\\
\hline
\SU_{2m}(q) & L.(\la\delta^\ell\ra\la\delta^d\phi^e\ra)=L.\big(\C_{(q+1)/\ell}.\C_{2f/e}\big) &
e\textup{ divides }2f,\;\ell\textup{ divides }q+1,\;\eqref{cor_G0a},\;\eqref{cor_G0b},\;\eqref{cor_G0c}\\
\Omega_{2m}^+(q) & \textup{any} & q\not\equiv1\ (\bmod\;4) \\
 & L.\calO & q\equiv1\ (\bmod\;4),\ \calO\nleqslant\la\delta'',\phi\ra \\
\Sp_{2m}(q) & \textup{any} & \\
\hline
\end{array}
\]
where $d\in\{1,\dots,q+1\}$ such that the following conditions hold with $i=(q^{2m}-1)/(q+1)$:
\begin{enumerate}[\rm(a)]
\item\label{cor_G0a} $(p^e-1)\ell$ divides $(q^2-1)d$;
\item\label{cor_G0b} $\pi(\ell)\cap\pi(i)\subseteq\pi(2mf)\cap\pi(p^{me}-1)\setminus\pi(d)$;
\item\label{cor_G0c} if $\gcd(\ell,i)$ is even and $p^{me}\equiv3\pmod{4}$, then $\gcd(\ell,i)\equiv2\pmod{4}$.
\end{enumerate}
\end{corollary}

\subsection{Determine the factor $K$}\label{sec:Xia-10}
\ \vspace{1mm}

Let us describe how to determine the factor $K$ in a factorization $G=HK$, given a pair $(G,H)$ as in Theorems~\ref{thm:Xia-2}--\ref{thm:Xia-4}. Note that, since $G=HB$, we have $G=HK$ if and only if $B=(H\cap B)K$. Then as $K\trianglerighteq B^{(\infty)}$ (see Table~\ref{TabLX2022}), it follows that $B=(H\cap B)K$ if and only if
\begin{equation}\label{eqn:Xia-19}
\overline{B}=\overline{H\cap B}\,\overline{K},
\end{equation}
where $\overline{\phantom{w}}\colon B\to B/B^{(\infty)}$ is the quotient modulo $B^{(\infty)}$. For basic factors $H$,
the intersection $H\cap B$ is determined in Propositions~\ref{Xia:Unitary01},~\ref{Xia:Omega01} and~\ref{Xia:Symplectic01}, so that the necessary and sufficient conditions for $K$ to satisfy $G=HK$ can be derived from~\eqref{eqn:Xia-19}. For non-basic factors $H$, the factorization $G=HK$ of $G$ is reduced to the factorization $A=H(A\cap K)$ for some $\max^-$ subgroup $A$ of $G$ as in the previous paragraph.

As an outcome of this approach, the ensuing proposition classifies all the possible $K$ with unspecified $H$.
Let $q=p^f$ with prime $p$ and positive integer $f$, let $\phi$ be the field automorphism of order $sf$ taking $p$-th power, and let $\la\delta'\ra=\mathrm{SO}_{2m}^+(q)/\Omega_{2m}^+(q)$. Define
\[
\ddot{G}=
\begin{cases}
\GaU_{2m}(q)&\textup{if $L=\SU_{2m}(q)$},\\
\Omega_{2m}^+(q).(\la(\delta')^{\gcd(2,q)}\ra\times\la\phi\ra)&\textup{if $L=\Omega_{2m}^+(q)$},\\
\GaSp_{2m}(q)&\textup{if $L=\Sp_{2m}(q)$}.
\end{cases}
\]
It is worth noting that $\ddot{G}$ is a subgroup of index $2$ in $\GaO_{2m}^+(q)$ if $L=\Omega_{2m}^+(q)$, and that $RT<\ddot{G}$.

\begin{proposition}\label{prop_K}
Let $\phi$, $\delta'$ and $\ddot{G}$ be as above, and let $G$, $K$ and $B$ be as in Hypothesis~$\ref{hypo-1}$\,\eqref{hypo-1ii}--\eqref{hypo-1iv} such that there exists some $H$ as in Hypothesis~$\ref{hypo-1}$ with $G=HB$. Then there exists some $H$ as in Hypothesis~$\ref{hypo-1}$ such that $G=HK$ if and only if $K$ satisfies the following conditions:
\begin{enumerate}[\rm(a)]
\item\label{prop_Ka} if $G\nleqslant\ddot{G}$, then $\GaO_{2m}^+(q)=\ddot{G}K$;
\item\label{prop_Kb} $N=J\big((K\cap\ddot{G})/B^{(\infty)}\big)$ with $N$ and $J$ in the following table,
\end{enumerate}
\[
\begin{array}{llll}
\hline
B^{(\infty)} & N & J \\
\hline
\SU_{2m}(q) & (\la\delta_1\ra\times\la\delta_2\ra){:}\la\phi\ra=(\C_{q+1}\times\C_{q+1}){:}\C_{2f} & Z{:}\la\phi\ra=\C_{q+1}{:}\C_{2f} \\
\Omega_{2m-1}(q) & \C_{\gcd(2,q-1)}\times\C_{\gcd(2,q-1)}\times\C_f & Z\times\la\phi\ra=\C_{\gcd(2,q-1)}\times\C_f \\
\Omega_{2m}^-(q) & \C_{2f} & \C_{2f} \\
\hline
\end{array}
\]
where $N=\Nor_{\ddot{G}}\big(B^{(\infty)}\big)/B^{(\infty)}$, $J=\Nor_{RT}\big(B^{(\infty)}\big)B^{(\infty)}/B^{(\infty)}$, $\la\delta_1\ra=\GU_{2m-1}(q)/\SU_{2m-1}(q)$, $\la\delta_2\ra=\GU_1(q)$, and $Z$ is the center of $\GU_{2m}(q)$ or $\mathrm{SO}_{2m}^+(q)$ in the unitary or orthogonal case respectively.
\end{proposition}

We anticipate that it would be too messy for an explicit list of the factors $K$ for all the possible pairs $(G,H)$. However, via the above described approach, specific questions on the factorization $G=HK$ with solvable $H$ can be tackled. For example, in the next theorem, we are able to determine the exact factorizations in symplectic groups.

\begin{theorem}\label{thm:Xia-5}
Let $G$ be an almost simple group with socle $\Sp_{2m}(q)$ such that $m\geqslant2$. If $G=HK$ with $H$ solvable, $K$ core-free and $H\cap K=1$, then $m$ is odd, $K=\Omega_{2m}^-(q).\mathcal{O}$ such that $|\mathcal{O}|$ is an odd divisor of $|G/\Sp_{2m}(q)|$, and $H=U(0){:}S=q^m{:}S$ for some transitive subgroup $S$ of $\GaL_1(q^m)$ on $\bbF_q^m\setminus\{0\}$ with stabilizer of order $|G/\Sp_{2m}(q)|/|\mathcal{O}|$.
\end{theorem}

The proof of Theorem~\ref{thm:Xia-5} is given at the end of Section~\ref{sec:Xia-6}.

\section{Subgroups of $\GaL_1(p^f)$}\label{sec:Xia-2}

Throughout this section, let $p$ be a prime, $f$ be a positive integer, $\omega$ be a generator of $\bbF_{p^f}^\times$, $a\in\GL_1(p^f)$ be the right multiplication of $\omega$, and $\varphi$ be the Frobenius of $\bbF_{p^f}$ taking $p$-th power.
The ensuing Subsection~\ref{sec:Xia-5} is a discussion on different parameters to represent subgroups of $\GaL_1(p^f)$.
Notably, we find a new triple of parameters, which is handy in the classification of transitive subgroups of $\GL_1(p^f)$. Then at the end of Subsection~\ref{sec:Xia-4} we prove Theorem~\ref{thm:trans} and Corollary~\ref{thm:min}. These will be applied in Subsection~\ref{sec:Xia-3} to characterize the subgroup $S$ in Theorems~\ref{thm:Xia-2}--\ref{thm:Xia-4}.
%Throughout this paper, we adopt Notation~\ref{ntn} and let $q=p^f$, $\bbF_q^*=\bbF_q^\times$ and $\pi=\pi(q-1)\cap\pi(f)$.
%Note that $\la\varphi\ra$ is the stabilizer in $\GaL_1(q)$ of the point $1\in\bbF_q^*$, and
%\[
%\la a\ra=\la a\ra_{\pi'}\times\la a\ra_\pi,
%\]
%where $\la a\ra_{\pi'}$ is the unique Hall $\pi'$-subgroup of $\GaL_1(q)$.

\subsection{Represent subgroups of $\GaL_1(p^f)$}\label{sec:Xia-5}
\ \vspace{1mm}

For a subgroup $H$ of $\GaL_1(p^f)$, there exists a divisor $\ell$ of $p^f-1$ such that $H\cap\la a\ra=\la a^\ell\ra$. Consider the quotient $\overline{\phantom{n}}$ of $\GaL_1(p^f)$ modulo $\GL_1(p^f)=\la a\ra$. Since the image $\la\overline{\varphi}\ra$ is a cyclic group of order $f$, we have $\overline{H}=\la(\overline{\varphi})^k\ra=\la\overline{\varphi^k}\ra$ for some divisor $k$ of $f$.
Hence there exists $d\in\la a\ra$ such that $d\varphi^k\in H$ and hence
\begin{equation}\label{eqn:Xia-3}
H=(H\cap\la a\ra)\la d\varphi^k\ra=\la a^\ell\ra\la d\varphi^k\ra=\la a^\ell,d\varphi^k\ra.
\end{equation}
Let $j\in\{0,1,\dots,\ell-1\}$ such that $d\in\la a^\ell\ra a^j$. Then $H=\la a^\ell,a^j\varphi^k\ra=\la a^\ell\ra\la a^j\varphi^k\ra$.
The idea of representing $H$ in terms of the parameters $\ell$, $j$ and $k$ originates from Fouler~\cite{Foulser1964} and is developed by the first author, Lim and Praeger~\cite{LLP2009} to characterize various properties of $H$ including its transitivity on $\bbF_{p^f}^\times$. However, the characterization of transitivity there (see~\cite[Lemma~4.7]{LLP2009}) is given by rather sophisticated conditions on $(\ell,j,k)$, which makes it difficult to describe the group structure of transitive $H$, let alone classify the minimally transitive and regular ones. To achieve our aim in this paper, we start our approach by changing the parameter $j$.

\begin{lemma}\label{lem:Xia-6}
Let $H\leqslant\GaL_1(p^f)$ with divisors $\ell$ and $k$ of $p^f-1$ and $f$, respectively, such that $H\cap\la a\ra=\la a^\ell\ra$ and $\la a\ra H=\la a\ra\la\varphi^k\ra$, and let $\sigma=\pi(\ell)$. Then there exists $j$ in $\{1,\ldots,p^f-1\}$ divisible by $(p^f-1)_{\sigma'}$ such that $H=\la a^\ell,a^j\varphi^k\ra=\la a^\ell\ra\la a^j\varphi^k\ra$.
\end{lemma}

\begin{proof}
Noticing that $\ell$ divides $(p^f-1)_{\sigma}$, we have $\la a^\ell\ra\geqslant\la a\ra_{\sigma'}$. Let $d\in\la a\ra$ such that $H=\la a^\ell,d\varphi^k\ra$ as in~\eqref{eqn:Xia-3}. Since $\la a\ra=\la a\ra_{\sigma'}\times\la a\ra_\sigma$, there exists $d'\in\la a\ra_\sigma$ such that $d\in\la a\ra_{\sigma'}d'$.
Since $\la a\ra_\sigma=\la a^{(p^f-1)_{\sigma'}}\ra$, there exists $j$ in $\{1,\ldots,p^f-1\}$ divisible by $(p^f-1)_{\sigma'}$ such that $d'=a^j$. Thus $d\in\la a^\ell\ra d'=\la a^\ell\ra a^j$, and so $H=\la a^\ell,d\varphi^k\ra=\la a^\ell,a^j\varphi^k\ra=\la a^\ell\ra\la a^j\varphi^k\ra$.
\end{proof}

\begin{definition}\label{def}
Let $H$, $\ell$, $j$ and $k$ be as in Lemma~$\ref{lem:Xia-6}$.  We call $(\ell,j,k)$ a \emph{Foulser triple} of $H$. Let $\sigma=\pi(\ell)$, $b=a^{(p^f-1)_{\sigma'}\ell}$ and $c=a^j$. The we call $(\sigma,b,c,k)$ a \emph{Foulser quadruple} of $H$.
\end{definition}

To prove some properties of Foulser quadruples, we need the consequence of a formula in the following lemma that calculates powers of an element in $\GaL_1(p^f)$. The proof of the formula is routine by induction.

\begin{lemma}\label{lem:Xia-5}
Let $g\in\GL_1(p^f)$, let $k\in\{1,\ldots,f\}$, and let $i$ be a positive integer. Then
\[
(g\varphi^k)^i=g^\frac{1-p^{(f-k)i}}{1-p^{f-k}}\varphi^{ki}=g^\frac{(p^{ki}-1)p^{k+fi-ki}}{p^k-1}\varphi^{ki}.
\]
In particular, if $i$ is divisible by the order of $\varphi^k$, then $(g\varphi^k)^i=g^{(p^f-1)p^k/(p^k-1)}$.
\end{lemma}

We list some properties of Foulser triple and quadruple of a subgroup $H$ of $\GaL_1(p^f)$ in the following lemma, where properties~\eqref{lem:Xia-3a}--\eqref{lem:Xia-3c} can be viewed as definitions of $\sigma$, $b$ and $c$ directly from $H$.

\begin{lemma}\label{lem:Xia-3}
Let $H$, $\ell$, $j$, $k$, $\sigma$, $b$ and $c$ be as in Definition~$\ref{def}$. Then the following statements hold:
\begin{enumerate}[\rm(a)]
  \item\label{lem:Xia-3a} $\la a\ra_{\sigma'}$ is the largest Hall subgroup of $\la a\ra$ contained in $H$;
  \item\label{lem:Xia-3b} $H\cap\la a\ra=\la a\ra_{\sigma'}\la b\ra=\la a\ra_{\sigma'}\times\la b\ra$ with $b\in\la a\ra_\sigma$;
  \item\label{lem:Xia-3c} $H=\la a\ra_{\sigma'}{:}(\la b\ra\la c\varphi^k\ra)=(\la a\ra_{\sigma'}\times\la b\ra)\la c\varphi^k\ra$ with $c\in\la a\ra_\sigma$;
  \item\label{lem:Xia-3d} $\ell=|a|_\sigma/|b|=(p^f-1)_\sigma/|b|$, and in particular, $\pi(|a|_\sigma/|b|)=\sigma$;
  \item\label{lem:Xia-3e} $c^{(p^f-1)/(p^k-1)}\in\la b\ra$;
  \item\label{lem:Xia-3f} $(p^k-1)\ell$ divides $(p^f-1)_\sigma j$;
  \item\label{lem:Xia-3g} $|H|=(q-1)_{\sigma'}|b|f/k=(q-1)f/(\ell k)$.
\end{enumerate}
\end{lemma}

\begin{proof}
For convenience, write $q=p^f$. Since $H\cap\la a\ra=\la a^\ell\ra$, statement~\eqref{lem:Xia-3a} follows from $\sigma=\pi(\ell)$, and then statement~\eqref{lem:Xia-3b} follows from $b=a^{(q-1)_{\sigma'}\ell}$. As $c=a^j$ with $j$ divisible by $(q-1)_{\sigma'}$, we have $c\in\la a\ra_\sigma$ and
\[
H=\la a^\ell\ra\la a^j\varphi^k\ra=(H\cap\la a\ra)\la a^j\varphi^k\ra
=(\la a\ra_{\sigma'}\times\la b\ra)\la c\varphi^k\ra
\]
Since $\la b\ra\la c\varphi^k\ra\leqslant\la a\ra_\sigma\la\varphi^k\ra$ and $\la a\ra_{\sigma'}\cap(\la a\ra_\sigma\la\varphi^k\ra)=\la a\ra_{\sigma'}\cap\la a\ra_\sigma=1$, it follows that
\[
H=\la a\ra_{\sigma'}\la b\ra\la c\varphi^k\ra=\la a\ra_{\sigma'}{:}(\la b\ra\la c\varphi^k\ra),
\]
which completes the proof of statement~\eqref{lem:Xia-3c}. Moreover, statement~\eqref{lem:Xia-3d} follows from
\[
\ell=\frac{|a|}{|H\cap\la a\ra|}
=\frac{|a|}{|\la a\ra_{\sigma'}\times\la b\ra|}=\frac{|a|_\sigma}{|b|}.
\]

Write $t=p^k$. Then taking $i=f/k$ in Lemma~\ref{lem:Xia-5} gives
\begin{equation}\label{eqn:Xia-4}
(c\varphi^k)^{f/k}=c^\frac{(q-1)t}{t-1}.
\end{equation}
Thus $c^{(q-1)t/(t-1)}\in H\cap\la a\ra$, and so $c^{(q-1)/(t-1)}\in H\cap\la a\ra$ as $t=p^k$ is coprime to $|a|$. Since $c^{(q-1)/(t-1)}\in\la c\ra\leqslant\la a\ra_\sigma$ and $H\cap\la a\ra=\la a\ra_{\sigma'}\la b\ra$ with $\la b\ra\leqslant\la a\ra_\sigma$, it follows that $c^{(q-1)/(t-1)}\in\la b\ra$, as statement~\eqref{lem:Xia-3e} asserts.
%This together with $\pi(|\la a\ra_\sigma|/|\la b\ra|)=\sigma$ leads to $\sigma\subseteq\pi((q-1)/(t-1))$.
Also, this yields that $(q-1)_{\sigma'}\ell$ divides $(q-1)j/(t-1)$, as $b=a^{(q-1)_{\sigma'}\ell}$ and $c=a^j$.
Hence $(t-1)\ell$ divides $(q-1)_\sigma j$, proving statement~\eqref{lem:Xia-3f}.

Finally,~\eqref{eqn:Xia-4}  combined with statement~\eqref{lem:Xia-3e} indicates that
\[
\la b\ra,\la b\ra(c\varphi^k),\ldots,\la b\ra(c\varphi^k)^{\frac{f}{k}-1}
\]
are the cosets of $\la b\ra$ in $\la b\ra\la c\varphi^k\ra$.
Hence $|\la b\ra\la c\varphi^k\ra|=|b|(f/k)$, and so
\[
|H|=|\la a\ra_{\sigma'}{:}(\la b\ra\la c\varphi^k\ra)|=(q-1)_{\sigma'}|b|\frac{f}{k}
=(q-1)_{\sigma'}|a^{(q-1)_{\sigma'}\ell}|\frac{f}{k}=\frac{q-1}{\ell}\cdot\frac{f}{k},
\]
as statement~\eqref{lem:Xia-3g} asserts.
\end{proof}

\begin{remark}
Definition~\ref{def} gives the explicit expression of the Foulser quadruple $(\sigma,b,c,k)$ in terms of a Foulser triple $(\ell,j,k)$. Conversely, $(\sigma,b,c,k)$ determines $(\ell,j,k)$, as $\ell=(p^f-1)_\sigma/|b|$ (see Lemma~\ref{lem:Xia-3}\,\eqref{lem:Xia-3d}) and $j$ is the unique integer in $\{1,\dots,p^f-1\}$ divisible by $(p^f-1)_{\sigma'}$ such that $c=a^j$.
\end{remark}

\begin{remark}
One may compare the triple $(\ell,j,k)$ of the so-called  ``standard parameters'' in~\cite{LLP2009} with our Foulser triple. Due to different parameter $j$, the divisibility condition
\[
(p^k-1)\ell\mid(p^f-1)j
\]
in~\cite[Lemma~4.4]{LLP2009} for the standard parameters is changed to $(p^k-1)\ell\mid(p^f-1)_\sigma j$ in Lemma~\ref{lem:Xia-3}\,\eqref{lem:Xia-3f} for the Foulser triple.
\end{remark}

\subsection{Classify transitive subgroups of $\GaL_1(p^f)$}\label{sec:Xia-4}
\ \vspace{1mm}

For usage in the next two lemmas, define a mapping $\psi\colon\GaL_1(p^f)=\la a\ra\la\varphi\ra\to\la a\ra$ by letting
\[
\psi(xy)=x
\]
for all $x\in\la a\ra$ and $y\in\la\varphi\ra$.

\begin{lemma}\label{lem:Xia-2}
A subgroup $H$ of $\GaL_1(p^f)$ is transitive on $\bbF_{p^f}^\times$ if and only if $\psi(H)=\la a\ra$.
\end{lemma}

\begin{proof}
If $\psi(H)=\la a\ra$, then $\GaL_1(q)=\la a\ra\la\varphi\ra=H\la\varphi\ra$, and so $H$ is transitive as $\la\varphi\ra$ is the stabilizer in $\GaL_1(q)$ of the point $1\in\bbF_{p^f}^\times$.
Conversely, suppose that $H$ is transitive on $\bbF_{p^f}^\times$.
We deduce from $H\subseteq\psi(H)\la\varphi\ra$ that $H=H^{-1}\subseteq\la\varphi\ra\psi(H)^{-1}$. Then, since $H$ is transitive and the element $1\in\bbF_{p^f}^\times$ is fixed by $\varphi$, it follows that
\[
\bbF_{p^f}^\times=1^H\subseteq1^{\la\varphi\ra\psi(H)^{-1}}=1^{\psi(H)^{-1}}.
\]
Since $\psi(H)^{-1}\subseteq\la a\ra$ and $\la a\ra$ is semiregular, we infer that $\psi(H)^{-1}=\la a\ra$. Hence $\psi(H)=\la a\ra$.
\end{proof}

In the next lemma we derive some necessary conditions for a subgroup of $\GaL_1(p^f)$ to be transitive on $\bbF_{p^f}^\times$ in terms of its Foulser quadruple.

\begin{lemma}\label{lem:Xia-4}
Let $H$ be a subgroup of $\GaL_1(p^f)$ transitive on $\bbF_{p^f}^\times$, and let $(\sigma,b,c,k)$ be a Foulser quadruple of $H$. Then the following statements hold:
\begin{enumerate}[\rm(a)]
\item\label{lem:Xia-4a} $\la c\ra=\la a\ra_\sigma$;
\item\label{lem:Xia-4d} $\sigma\subseteq\pi(f)\cap\pi(p^k-1)$;
\item\label{lem:Xia-4e} if $2\in\sigma$ and $p^k\equiv3\pmod{4}$, then $|b|_2=|a|_2/2$.
\end{enumerate}
\end{lemma}

\begin{proof}
Write $q=p^f$ and $\pi=\pi(q-1)\cap\pi(f)$. Since $H$ is transitive on $\bbF_q^\times$, the order of $H$ is divisible by $(q-1)_{\pi'}$. Then, as $\la a\ra_{\pi'}$ is the unique Hall $\pi'$-subgroup of $\GaL_1(q)$, this implies that $\la a\ra_{\pi'}\leqslant H$. Note from Lemma~\ref{lem:Xia-3}\,\eqref{lem:Xia-3a} that $\la a\ra_{\sigma'}$ is the largest Hall subgroup of $\la a\ra$ contained in $H$. We then deduce $\pi(q-1)\setminus\pi\subseteq\pi(q-1)\setminus\sigma$, that is, $\sigma\subseteq\pi$. In particular, $\sigma\subseteq\pi(f)$.

Since $H=\la a\ra_{\sigma'}\la b\ra\la c\varphi^k\ra\leqslant(\la a\ra_{\sigma'}\la b\ra\la c\ra){:}\la\varphi^k\ra$, we have $\psi(H)\subseteq\la a\ra_{\sigma'}\la b\ra\la c\ra$. Moreover, Lemma~\ref{lem:Xia-2} asserts that $\psi(H)=\la a\ra$. Therefore,
\begin{equation}\label{eqn:Xia-1}
\la a\ra_{\sigma'}\la b\ra\la c\ra=\la a\ra.
\end{equation}
If $|a|_\sigma/|c|$ is divisible by some $r\in\sigma$, then each of $\la a\ra_{\sigma'}$, $\la b\ra$ and $\la c\ra$ has index in $\la a\ra$ divisible by $r$ and hence is contained in $\la a^r\ra$, contradicting~\eqref{eqn:Xia-1}. Thus $\la c\ra=\la a\ra_\sigma$, proving statement~\eqref{lem:Xia-4a}.

To prove $\sigma\subseteq\pi(p^k-1)$, suppose for a contradiction that $p^k-1$ is not divisible by some $r\in\sigma$. Then there exists $i\in\{0,1,\dots,r-1\}$ such that $(1-p^k)i\equiv jp^k\pmod{r}$. Now consider the proper subset $\la\omega^r\ra\omega^i$ of $\bbF_q^\times$. It is stabilized by $c\varphi^k$ because
\[
(\la\omega^r\ra\omega^i)^{c\varphi^k}=(\la\omega^r\ra\omega^{i+j})^{\varphi^k}=\la\omega^{rp^k}\ra\omega^{(i+j)p^k}
=\la\omega^r\ra\omega^{(i+j)p^k}=\la\omega^r\ra\omega^i.
\]
Moreover, since $\la a\ra_{\sigma'}$ and $\la b\ra$ both have index in $\la a\ra$ divisible by $r$, they are contained in $\la a^r\ra$ and hence stabilize $\la\omega^r\ra\omega^i$. As a consequence, $\la\omega^r\ra\omega^i$ is stabilized by $\la a\ra_{\sigma'}\la b\ra\la c\varphi^k\ra=H$, contradicting the transitivity of $H$ on $\bbF_q^\times$.

Thus we conclude that $\sigma\subseteq\pi(p^k-1)$, completing the verification of statement~\eqref{lem:Xia-4d}. Now let $2\in\sigma$ and $p^k\equiv3\pmod{4}$, as in the assumption of statement~\eqref{lem:Xia-4e}. Then $(|a|/|b|)_2\geqslant2$ as $\pi(|a|_\sigma/|b|)=\sigma$ (see Lemma~\ref{lem:Xia-3}\,\eqref{lem:Xia-3d}). Suppose for a contradiction that $(|a|/|b|)_2\geqslant4$. Since $c=a^j$ is a generator of $\la a\ra_\sigma$, the integer $j$ is odd, which indicates that $2j\equiv2\pmod{4}$. Hence
\begin{align*}
(\la\omega^4\ra\omega^2)^{c\varphi^k}&=(\la\omega^4\ra\omega^{2+j})^{\varphi^k}
=\la\omega^{4p^k}\ra\omega^{(2+j)p^k}=\la\omega^4\ra\omega^{3(2+j)}=\la\omega^4\ra\omega^j,\\
(\la\omega^4\ra\omega^j)^{c\varphi^k}&=(\la\omega^4\ra\omega^{2j})^{\varphi^k}
=\la\omega^{4p^k}\ra\omega^{2jp^k}=\la\omega^4\ra\omega^{6j}=\la\omega^4\ra\omega^2,
\end{align*}
and so $(\la\omega^4\ra\omega^2)\cup(\la\omega^4\ra\omega^j)$ is stabilized by $c\varphi^k$.
Moreover, since both $\la a\ra_{\sigma'}$ and $\la b\ra$ have index in $\la a\ra$ divisible by $4$, they both stabilize $(\la\omega^4\ra\omega^2)\cup(\la\omega^4\ra\omega^j)$. Thus $H=\la a\ra_{\sigma'}\la b\ra\la c\varphi^k\ra$ stabilizes $(\la\omega^4\ra\omega^2)\cup(\la\omega^4\ra\omega^j)$, contradicting the transitivity of $H$ on $\bbF_q^\times$. Therefore, $(|a|/|b|)_2=2$, which proves statement~\eqref{lem:Xia-4e}.
\end{proof}

The following is~\cite[Lemma~8.16\,(b)]{LX2022}.

\begin{lemma}\label{lem:Xia-1}
Let $n\geqslant2$ and $\ell\geqslant1$ be integers. Then the following statements hold:
\begin{enumerate}[\rm(a)]
\item if $r$ is an odd prime dividing $n-1$, then $(n^\ell-1)_r=\ell_r(n-1)_r$;
\item if $4$ divides $n-1$, then $(n^\ell-1)_2=\ell_2(n-1)_2$.
\end{enumerate}
\end{lemma}

We are now in a position to prove Theorem~\ref{thm:trans} and Corollary~\ref{thm:min}.

\begin{proof}[Proof of Theorem~$\ref{thm:trans}$ and Corollary~$\ref{thm:min}$]
The ``if'' part of Theorem~\ref{thm:trans} follows from Lemma~\ref{lem:Xia-4} by considering any Foulser quadruple of $H$. Conversely, let
\[
H=\la a\ra_{\sigma'}{:}(\la b\ra\la c\varphi^k\ra)
\]
with some divisor $k$ of $f$, subset $\sigma$ of $\pi(p^k-1)\cap\pi(f)$, element $b$ in $\la a\ra_\sigma$ and generator $c$ of $\la a\ra_\sigma$ satisfying the condition of Theorem~\ref{thm:trans}. (We do not assume that $(\sigma,b,c,k)$ is a Foulser quadruple.)

Write $q=p^f$, $e=f/k$ and $t=p^k$. Taking $i=e$ in Lemma~\ref{lem:Xia-5} gives
\begin{equation}\label{eqn:Xia-9}
(c\varphi^k)^e=c^\frac{(q-1)t}{t-1}.
\end{equation}
As $\la c\ra=\la a\ra_\sigma$, this implies that $\la c\varphi^k\ra$ is semiregular on $\bbF_q^\times$ and that
\begin{equation}\label{eqn:Xia-8}
|c\varphi^k|=e(t-1)_\sigma.
\end{equation}
Clearly, $t-1$ is divisible by every $r\in\sigma$.
If either $r=2\in\sigma$ and $4$ divides $t-1$ or $r\in\sigma\setminus\{2\}$, then by Lemma~\ref{lem:Xia-1} we have
\begin{equation}\label{eqn:Xia-15}
|c\varphi^k|_r=e_r(t-1)_r=(t^e-1)_r=(q-1)_r.
\end{equation}

\textsf{Case~1}: $2\notin\sigma$ or $4$ divides $t-1$. In this case, $|c\varphi^k|_r=(q-1)_r$ for each $r\in\sigma$, and so $|c\varphi^k|_\sigma=(q-1)_\sigma$.
Since $\la a\ra_{\sigma'}$ and $\la c\varphi^k\ra_\sigma$ are both semiregular on $\bbF_q^\times$ and have coprime order, the group $\la a\ra_{\sigma'}{:}\la c\varphi^k\ra_\sigma$ is semiregular on $\bbF_q^\times$ and has order
\[
|a|_{\sigma'}|c\varphi^k|_\sigma=(q-1)_{\sigma'}(q-1)_\sigma=q-1.
\]
Therefore, $\la a\ra_{\sigma'}{:}\la c\varphi^k\ra_\sigma$ is regular on $\bbF_q^\times$. As a consequence, $H$ is transitive on $\bbF_q^\times$, as the ``only if'' part of Theorem~\ref{thm:trans} requires. It also follows that $H$ is minimally transitive on $\bbF_q^\times$ if and only if $b=1$ and $|c\varphi^k|_{\sigma'}=1$, and the same is true for $H$ to be regular. In view of~\eqref{eqn:Xia-8} we see that $|c\varphi^k|_{\sigma'}=1$ if and only if $e_{\sigma'}=1$, that is, $k$ is divisible by $f_{\sigma'}$.
Thus Corollary~\ref{thm:min} holds in this case.

\textsf{Case~2}: $2\in\sigma$ and $4$ does not divide $t-1$. In this case, recall from the condition of Theorem~\ref{thm:trans} that $|b|_2=|a|_2/2$. Thus $\la a\ra_{\sigma'}\la b\ra_2$ has order $(q-1)_{\sigma'}(q-1)_2/2$. Since $|c\varphi^k|_r=(q-1)_r$ for each $r\in\sigma\setminus\{2\}$, which implies $|c\varphi^k|_{\sigma\setminus\{2\}}=(q-1)_{\sigma\setminus\{2\}}$, it follows that
\[
K:=\la a\ra_{\sigma'}{:}(\la b\ra_2\la c\varphi^k\ra_{\sigma\setminus\{2\}})
\]
is semiregular on $\bbF_q^\times$ and has order
\[
|\la a\ra_{\sigma'}\la b\ra_2|\cdot|c\varphi^k|_{\sigma\setminus\{2\}}
=\frac{(q-1)_{\sigma'}(q-1)_2}{2}\cdot(q-1)_{\sigma\setminus\{2\}}=\frac{q-1}{2}.
\]
Since $|c\varphi^k|=e(t-1)_\sigma$ is even, we have $\la c\varphi^k\ra_{\sigma\setminus\{2\}}\leqslant\la(c\varphi^k)^2\ra$.
Write $c=a^j$ with integer $j$. Then $j$ is odd as $\la c\ra=\la a\ra_\sigma$. Therefore, for $i\in\{0,1\}$,
\begin{equation}\label{eqn:Xia-2}
(\la\omega^2\ra\omega^i)^{c\varphi^k}=(\la\omega^2\ra\omega^{i+j})^{\varphi^k}=\la\omega^{2p^k}\ra\omega^{(i+j)p^k}
=\la\omega^2\ra\omega^{i+j}=\la\omega^2\ra\omega^{1-i}.
\end{equation}
As a consequence, $\la\omega^2\ra$ is stabilized by $(c\varphi^k)^2$ and hence by $\la c\varphi^k\ra_{\sigma\setminus\{2\}}$.
This together with $\la a\ra_{\sigma'}\la b\ra_2\leqslant\la a^2\ra$ yields that $K$ stabilizes $\la\omega^2\ra$.
Since $K$ is semiregular on $\bbF_q^\times$ and has order $(q-1)/2$, it follows that the orbits of $K$ on $\bbF_q^\times$ are $\la\omega^2\ra$ and $\la\omega^2\ra\omega$. Moreover,~\eqref{eqn:Xia-2} shows that $c\varphi^k$ does not stabilize $\la\omega^2\ra$. Hence $\la a\ra_{\sigma'}{:}(\la b\ra_2\la c\varphi^k\ra_{\sigma\setminus\{2\}})$ is transitive on $\bbF_q^\times$, and so is $H$.
This completes the proof of Theorem~\ref{thm:trans}.

Since $|\la a\ra|/|\la b\ra|_2=2$, it also follows that $H$ is minimally transitive on $\bbF_q^\times$ if and only if $\la b\ra=\la b\ra_2=\la a^2\ra_2$ and $|c\varphi^k|_{\sigma'}=1$. Note from~\eqref{eqn:Xia-8} that $|c\varphi^k|_{\sigma'}=1$ if and only if $k$ is divisible by $f_{\sigma'}$, and note from $2\in\sigma\subseteq\pi(t-1)$ that $4$ does not divide $t-1$ if and only if $t\equiv3\pmod{4}$. Thus $H$ is minimally transitive on $\bbF_q^\times$ if and only if it satisfies Corollary~\ref{thm:min}\,\eqref{thm:minb} with $k$ divisible by $f_{\sigma'}$. From now on, let
\begin{equation}\label{eqn:Xia-16}
H=\la a\ra_{\sigma'}{:}(\la a^2\ra_2\la c\varphi^k\ra)
\end{equation}
be such a minimally transitive subgroup. Since $|c\varphi^k|_{\sigma'}=1$ and $|c\varphi^k|_r=(q-1)_r$ for each $r\in\sigma\setminus\{2\}$, we derive that $|H|=(q-1)_{2'}|\la a^2\ra_2\la c\varphi^k\ra|_2$, which together with~\eqref{eqn:Xia-9} implies that
\[
|H|=(q-1)_{2'}\cdot\frac{(q-1)_2}{2}\cdot e_2.
\]
Hence $H$ is regular on $\bbF_q^\times$ if and only if $e_2=2$. Note that $p^k=t\equiv3\pmod{4}$ implies $k_2=1$ and hence $f_2=e_2$.
We conclude that $H$ is regular in this case if and only if $f_2=2$, or equivalently, $f\equiv2\pmod{4}$.
Finally, assume the condition $f_2=2$, so that $H$ is regular on $\bbF_q^\times$. It follows from~\eqref{eqn:Xia-8} that $|c\varphi^k|_2=e_2(t-1)_2=f_2(t-1)_2=2\cdot2=4$, which implies
\[
|c\varphi^k|=|c\varphi^k|_{\sigma'}|c\varphi^k|_\sigma=|c\varphi^k|_\sigma=4|c\varphi^k|_{\sigma\setminus\{2\}}=4\cdot(q-1)_{\sigma\setminus\{2\}}
\]
by~\eqref{eqn:Xia-15}. Let $x=a^{2(q-1)_{2'}}$, $y=(c\varphi^k)^{(q-1)_{\sigma\setminus\{2\}}}$ and $z=(c\varphi^k)^4$.
Then $\la a^2\ra_2=\la x\ra=\C_{(q-1)_2/2}$, and
\begin{equation}\label{eqn:Xia-17}
\la c\varphi^k\ra=\la y\ra\times\la z\ra
\end{equation}
with $\la y\ra=\C_4$ and $\la z\ra=\C_{(q-1)_{\sigma\setminus\{2\}}}$. By Lemma~\ref{lem:Xia-1},
\[
\big(t^{2(q-1)_{\sigma\setminus\{2\}}}-1\big)_2=(t^2-1)_2=\big((t^2)^{e/2}-1\big)_2=(q-1)_2.
\]
Since $(t^{(q-1)_{\sigma\setminus\{2\}}}-1)_2=2$ as $t\equiv3\pmod{4}$, it follows that $(t^{(q-1)_{\sigma\setminus\{2\}}}+1)_2=(q-1)_2/2$. Hence
\[
x^yx=\big(a^{2(q-1)_{2'}}\big)^{(c\varphi^k)^{(q-1)_{\sigma\setminus\{2\}}}}a^{2(q-1)_{2'}}
=a^{2(q-1)_{2'}(t^{(q-1)_{\sigma\setminus\{2\}}}+1)}=a^{(q-1)_2(q-1)_{2'}}=1.
\]
In view of~\eqref{eqn:Xia-16} and~\eqref{eqn:Xia-17}, we have $|\la x\ra\la y\ra|=|H|_2=(q-1)_2$. Therefore, the Sylow $2$-subgroup $\la x,y\ra=\la x\ra\la y\ra$ of $H$ is a generalized quaternion group of order $(q-1)_2$. This together with~\eqref{eqn:Xia-16} and~\eqref{eqn:Xia-17} completes the proof of Corollary~\ref{thm:min}.
\end{proof}

\subsection{Apply to solvable factors}\label{sec:Xia-3}
\ \vspace{1mm}

Let $H$ be a subgroup of $\GaL_1(p^f)=\la a\ra{:}\la\varphi\ra$ with a Foulser triple $(\ell,j,k)$ and a corresponding Foulser quadruple $(\sigma,b,c,k)$ as in Definition~\ref{def}. Then $H=\la a^\ell\ra\la a^j\varphi^k\ra$, $H\cap\la a\ra=\la a^\ell\ra$, $b=a^{(p^f-1)_{\sigma'}\ell}$, and $c=a^j$. Therefore, $\la c\ra=\la a\ra_\sigma$ if and only if $\sigma\cap\pi(j)=\varnothing$, and in the case $2\in\sigma$, we have $|b|_2=|a|_2/2$ if and only if $\ell\equiv2\pmod{4}$.
It then follows from Theorem~\ref{thm:trans} and Lemma~\ref{lem:Xia-4} that $H$ is transitive on $\bbF_{p^f}^\times$ if and only if the following conditions hold:
\begin{itemize}
\item $\pi(\ell)\subseteq\pi(f)\cap\pi(p^k-1)\setminus\pi(j)$;
\item if $\ell$ is even and $p^k\equiv3\pmod{4}$, then $\ell\equiv2\pmod{4}$.
\end{itemize}
For a divisor $i$ of $p^f-1$, the subgroup $\la a^i\ra H=\la a^{\gcd(\ell,i)}\ra\la a^j\varphi^k\ra$ of $\GaL_1(p^f)$ has a Foulser triple $(\gcd(\ell,i),j',k)$ such that $j'-j$ is divisible by $\gcd(\ell,j)$. Note that $\pi(\gcd(\ell,i))=\pi(\ell)\cap\pi(i)$ and that $\pi(\gcd(\ell,i))\cap\pi(j')=\varnothing$ if and only if $\pi(\gcd(\ell,i))\cap\pi(j)=\varnothing$. Then the argument of this paragraph leads to the following theorem.

\begin{theorem}\label{thm:Xia-1}
Let $H\leqslant\GaL_1(p^f)=\la a\ra{:}\la\varphi\ra$, let $(\ell,j,k)$ be a Foulser triple of $H$, and let $i$ be a divisor of $p^f-1$. Then $H$ is transitive on the set of orbits of $\la a^i\ra$ on $\bbF_{p^f}^\times$ if and only if the following conditions hold:
\begin{enumerate}[\rm(a)]
\item\label{thm:Xia-1a} $\pi(\ell)\cap\pi(i)\subseteq\pi(f)\cap\pi(p^k-1)\setminus\pi(j)$;
\item\label{thm:Xia-1b} if $\gcd(\ell,i)$ is even and $p^k\equiv3\pmod{4}$, then $\gcd(\ell,i)\equiv2\pmod{4}$.
\end{enumerate}
%In this case, the stabilizer in $\la a^i\ra H$ of $\la\omega^i\ra\in\bbF_{p^f}^\times/\la a^i\ra$ is $\la a^i\ra{:}\la\varphi^{\gcd(\ell,i) k}\ra$.
\end{theorem}

The two remarks below show how Theorem~\ref{thm:Xia-1} is applied to the characterization of solvable factors $H$ in Hypothesis~\ref{hypo-1}.

\begin{remark}\label{rem:Xia-1}
Under Hypothesis~\ref{hypo-1}\,\eqref{hypo-1i}, let $(\ell,j,k)$ be a Foulser triple of $H$ as a subgroup of $\GaL_1(p^f)$ with $p^f=q^m$. Since $G=HB$ if and only if $H$ is transitive on $(\bbF_q^m)_{q-1}$, we derive from Theorem~\ref{thm:Xia-1} that $G=HB$ is equivalent to~\eqref{thm:Xia-1a} and~\eqref{thm:Xia-1b} in Theorem~\ref{thm:Xia-1} with $i=(q^m-1)/(q-1)$.
\end{remark}

\begin{remark}\label{rem:Xia-2}
Let $(\ell,j,k)$ be a Foulser triple of the group $S$ in Theorems~\ref{thm:Xia-2}--\ref{thm:Xia-4}, where $S\leqslant\GaL_1(p^f)$ such that $p^f=q^{2m}$ for Theorem~\ref{thm:Xia-2} and $p^f=q^m$ for Theorems~\ref{thm:Xia-3} and~\ref{thm:Xia-4}.
Then the condition on $S$ in Theorem~\ref{thm:Xia-2} is equivalent to~\eqref{thm:Xia-1a} and~\eqref{thm:Xia-1b} in Theorem~\ref{thm:Xia-1} with $i=(q^{2m}-1)/(q+1)$. Similarly, the condition on $S$ in Theorem~\ref{thm:Xia-3}\,\eqref{thm:Xia-3a} is equivalent to~\eqref{thm:Xia-1a} and~\eqref{thm:Xia-1b} in Theorem~\ref{thm:Xia-1} with $i=(q^m-1)/\gcd(2,q-1)$.
For the rest of Theorems~\ref{thm:Xia-3} and~\ref{thm:Xia-4}, the transitivity of $S$ on $\bbF_q^m\setminus\{0\}$ is equivalent to $\pi(\ell)\subseteq\pi(f)\cap\pi(p^k-1)\setminus\pi(j)$, and the condition that $|T|/|SS_0|$ is odd turns out to be that $k$ is odd.
\end{remark}

\section{Reduction}\label{sec:Xia-1}

As mentioned in the Introduction, the solvable factor $H$ in~\eqref{hypo-1ii}--\eqref{hypo-1iv} of Hypothesis~\ref{hypo-1} can be written as $H=P.S$ with $P=H\cap R\leqslant R$ and $S\leqslant T$, where $R$ is the unipotent radical of $\Pa_m[G]$ and
%$T=\GaL_1(q^{sm})$ such that
$T\cap G$ is the Levi subgroup of $\Pa_m[G]$. We first prove that $H$ is necessarily the split extension $P{:}S$.

\begin{lemma}\label{lem:Xia-7}
Let $G=HB$ with $(G,H,B)$ in Hypothesis~$\ref{hypo-1}$\,\eqref{hypo-1ii}--\eqref{hypo-1iv} such that $H=P.S$ with $P=H\cap R$ and $S\leqslant T$. Then $H=P{:}(H\cap T)$.
\end{lemma}

\begin{proof}
By the conditions in Hypothesis~\ref{hypo-1}\,\eqref{hypo-1ii}--\eqref{hypo-1iv}, there exists a primitive prime divisor (for convenience, we view $7$ as a primitive prime divisor of $2^6-1$) of $q^{sm/\gcd(2,m+s-1)}-1$, say, $r$. Since $|H|$ is divisible by $|G|/|B|$, it follows that $|H|$ is divisible by $r$.
%Let $S_0=\GL_1(q^{sm})$ be the Singer group in $T$.
Hence $H$ contains the unique cyclic subgroup $\la h\ra$ of order $r$ in $T$. Let $N=\mathbf{N}_{\widehat{H}}(P\la h\ra)$, where $\widehat{H}=R{:}T$. It can be seen from~\cite[\S2]{BL2021} (or see Subsection~\ref{subsec_form} below) that the decomposition of $R$ into irreducible submodules as an $S$-module is the same as a $T$-module. This implies that $T$ normalizes $P$. Since $T$ also normalizes $\la h\ra$, we derive that $T\leqslant\mathbf{N}_{\widehat{H}}(P\la h\ra)=N$. Therefore, $N=(R\cap N)T$.

Take an arbitrary $x\in R\cap N$. Then $\la h\ra^x\leqslant(P\la h\ra)^x=P\la h\ra$, and so by Sylow's theorem, there exists $y\in P$ such  that $\la h\ra^x=\la h\ra^y$. In other words, writing $z=xy^{-1}$, we have $h^z\in\la h\ra$. Now
\[
h^{-1}z^{-1}hz=(h^{-1}z^{-1}h)z=h^{-1}(z^{-1}hz)
\]
lies in both $R$ and $\la h\ra$. It follows that $h^{-1}z^{-1}hz=1$, or equivalently, $z^h=z$. Since the action of $h$ on $R\setminus\{1\}$ is fixed-point-free (see~\cite[\S2]{BL2021} or Subsection~\ref{subsec_form} below), this implies $z=1$. Consequently, $x=y\in P$.

Thus we conclude that $R\cap N\leqslant P$. Since $P\leqslant P\la h\ra\leqslant\mathbf{N}_{\widehat{H}}(P\la h\ra)=N$, this leads to $R\cap N=P$ and so $N=(R\cap N)T=PT$. Moreover, as $\la h\ra\trianglelefteq S$, we have $P\la h\ra\trianglelefteq P.S=H$, which means $H\leqslant\mathbf{N}_{\widehat{H}}(P\la h\ra)=N$. Hence $H\cap(PT)=P(H\cap T)$. Then since $P\cap(H\cap T)\leqslant R\cap T=1$, the conclusion of the lemma holds.
\end{proof}

Through the rest of this section we show that, under certain condition on $P$, the transitivity of $H$ on $[G:B]$ is reduced to the transitivity of $S$ on $(\bbF_{q^2}^m)_{(q+1)}$ or $(\bbF_q^m)_{(\gcd(2,q-1))}$, according to whether $G$ is the unitary group or not.
In the next section we will show that the required condition on $P$ is satisfied for all basic (as defined in the Introduction) $H$.

\subsection{Unitary groups}
\ \vspace{1mm}

Throughout this subsection, let $q=p^f$ be a power of a prime $p$, let $m\geqslant2$ be an integer, let $V$ be a vector space of dimension $2m$ over $\bbF_{q^2}$ equipped with a nondegenerate Hermitian form $\beta$, and let $\perp$ denote the perpendicularity with respect to $\beta$. The unitary space $V$ has a standard basis $e_1,f_1,\dots,e_m,f_m$ as in~\cite[2.2.3]{LPS1990}.
Let $\phi\in\GaU(V)=\GaU_{2m}(q)$ such that
\[
\phi\colon a_1e_1+b_1f_1+\dots+a_me_m+b_mf_m\mapsto a_1^pe_1+b_1^pf_1+\dots+a_m^pe_m+b_m^pf_m
\]
for $a_1,b_1\dots,a_m,b_m\in\bbF_{q^2}$, let $U=\la e_1,\dots,e_m\ra$, and let $W=\la f_1,\dots,f_m\ra$.
Moreover, let $\omega$ be a generator of $\bbF_{q^2}^\times$, let $\lambda\in\bbF_{q^2}$ with $\lambda+\lambda^q=1$ (note that such $\lambda$ exists as the trace of the field extension $\bbF_{q^2}/\bbF_q$ is surjective), and let
\[
v=e_1+\lambda f_1.
\]
Then $v$ is nonsingular as $\beta(v,v)=\lambda+\lambda^q=1$, and $\GaU(V)_U$ has a subgroup $R{:}T$, where
\[
R=q^{m^2}
\]
is the kernel of $\GaU(V)_U$ acting on $U$, and
\[
T=\GaL_1(q^{2m})
\]
stabilizes both $U$ and $W$.
%For $\tau\in\bbF_{q^2}^\times$, let $W/\la\tau\ra$ denote the set of orbits of the scalar multiplication of $\tau$ on $W\setminus\{0\}$.
Assume $(m,q)\neq(2,2)$ in the following, so that $\SU_{2m-1}(q)$ is nonsolvable.

\begin{proposition}\label{Xia:Unitary01}
Let $\SU_{2m}(q)=L\leqslant G\leqslant\GaU_{2m}(q)$, let $B=G_{\la v\ra}=\N_1[G]$, and let $H=P{:}S<G$ with $P\leqslant R$ and $S\leqslant T$ such that $|P|/|P\cap B|=q^{2m-1}$. Then the following statements hold:
\begin{enumerate}[\rm(a)]
\item\label{Xia:Unitary01d} $P\cap B=P\cap B^{(\infty)}$;
\item\label{Xia:Unitary01a} $H\cap B=(P\cap B){:}S_{\la\omega^{q-1}\ra\lambda f_1}$;
\item\label{Xia:Unitary01b} $H\cap B^{(\infty)}=(P\cap B){:}(S\cap L)_{f_1}$;
\item\label{Xia:Unitary01c} $G=HB$ if and only if $S$ is transitive on $W_{(q+1)}$.
\end{enumerate}
\end{proposition}

\begin{proof}
Let $X=\GaU(V)=\GaU_{2m}(q)$, $Y=X_{\la v\ra}$, $M=P{:}T$ and $U_1=\la e_2,\dots,e_m\ra$. Since $Y$ fixes $\la v\ra$, it stabilizes $v^\perp$. Thus $M\cap Y$ stabilizes $U\cap v^\perp=\la e_2,\dots,e_m\ra=U_1$. Note that $B^{(\infty)}=\SU_{2m-1}(q)=L_v$ and hence $H\cap B^{(\infty)}=H\cap L_v=(H\cap L)_v$. Since $P\leqslant R$ is a $p$-group in $L$ while $L_v$ is a normal subgroup of $L_{\la v\ra}$ with index coprime to $p$, we deduce
\[
P\cap B=P\cap Y=P\cap L_{\la v\ra}=P\cap L_v=P\cap B^{(\infty)}
\]
and $H\cap L=P{:}(S\cap L)$. In particular, statement~\eqref{Xia:Unitary01d} holds.

Take an arbitrary element in $M\cap Y$, say, $\phi^ig$ with integer $i$ and element $g$ in $\GU(V)$.
Then $v^{\phi^ig}=\eta v$ for some $\eta\in\bbF_{q^2}^\times$. Hence
\[
1=\beta(v,v)^{p^i}=\beta(v^{\phi^i},v^{\phi^i})=\beta((\eta v)^{g^{-1}},(\eta v)^{g^{-1}})=\beta(\eta v,\eta v)=\eta^{1+q}\beta(v,v)=\eta^{1+q},
\]
which means that $\eta\in\la\omega^{q-1}\ra$. Write $e_1^g=\mu e_1+e$ with $\mu\in\bbF_{q^2}$ and $e\in U_1$. It follows that
\begin{align*}
\mu\eta^q\lambda^q=\beta(\mu e_1+e,\eta(e_1+\lambda f_1))&=\beta(e_1^{\phi^ig},(e_1+\lambda f_1)^{\phi^ig})\\
&=\beta(e_1^g,(e_1+\lambda^{p^i} f_1)^g)=\beta(e_1,e_1+\lambda^{p^i} f_1)=\lambda^{p^iq},
\end{align*}
and so $\mu=\eta^{-q}\lambda^{(p^i-1)q}$. Now $\phi^ig\in M\cap Y$ stabilizes $U_1$, and $e_1^g=\eta^{-q}\lambda^{(p^i-1)q}e_1+e$. Thus
\begin{align*}
(\la\omega^{q-1}\ra\lambda^{-q}e_1+U_1)^{\phi^ig}&=(\la\omega^{q-1}\ra\lambda^{-q}e_1)^{\phi^ig}+U_1^{\phi^ig}\\
&=\la\omega^{(q-1)p^i}\ra\lambda^{-qp^i}e_1^g+U_1\\
&=\la\omega^{q-1}\ra\lambda^{-qp^i}(\eta^{-q}\lambda^{(p^i-1)q}e_1+e)+U_1\\
&=\la\omega^{q-1}\ra\eta^{-q}\lambda^{-q}e_1+U_1\\
&=\la\omega^{q-1}\ra\lambda^{-q}e_1+U_1.
\end{align*}
This shows that $M\cap Y$ stabilizes $\la\omega^{q-1}\ra\lambda^{-q}e_1+U_1$, that is, the induced group of $M\cap Y$ on $U$ is contained in $T_{U_1,\la\omega^{q-1}\ra\lambda^{-q}e_1+U_1}$.
Let $E=\la\omega^{q-1}\ra\lambda^{-q}e_1$. Note that $P$ is the kernel of $M$ acting on $U$, and so $P\cap Y$ is the kernel of $M\cap Y$ acting on $U$. We then have
\begin{equation}\label{eqn:Xia-6}
(M\cap Y)/(P\cap Y)\cong(M\cap Y)^U\leqslant T_{U_1,E+U_1}.
\end{equation}

Let $t$ be an arbitrary element in $T$ with integer $j$ and element $s$ in $\GU(V)$. For each $\xi e_1\in E$, where $\xi\in\la\omega^{q-1}\ra\lambda^{-q}$, the vector $\xi^{-q}f_1$ is the unique one in $W$ such that $\beta(x,\xi^{-q}f_1)=1$ for all $x\in\xi e_1+U_1$. Thus, if $t$ stabilizes $E+U_1$, then $t$ stabilizes $(\la\omega^{q-1}\ra\lambda^{-q})^{-q}f_1=\la\omega^{q-1}\ra\lambda f_1$. Conversely, suppose that $t$ stabilizes $\la\omega^{q-1}\ra\lambda f_1$. Then $t$ stabilizes $U\cap f_1^\perp=U_1$.
Moreover, for each $\zeta\in\la\omega^{q-1}\ra\lambda$, the set of $y\in U$ such that $\beta(y,\zeta f_1)=1$ is $\zeta^{-q}e_1+U_1$.
Accordingly, $t$ stabilizes $(\la\omega^{q-1}\ra\lambda)^{-q}e_1+U_1=E+U_1$. This shows that
\begin{equation}\label{eqn:Xia-7}
T_{U_1,E+U_1}=T_{\la\omega^{q-1}\ra\lambda f_1},
\end{equation}
which in conjunction with~\eqref{eqn:Xia-6} and
\[
\frac{|M\cap Y|}{|P\cap Y|}\geqslant\frac{|M||Y|}{|P\cap Y||X|}=\frac{|P||T||(q+1)|\GaU_{2m-1}(q)|}{|P_v||\GaU_{2m}(q)|}
=\frac{|T|(q+1)}{q^{2m}-1}=|T_{\la\omega^{q-1}\ra\lambda f_1}|
\]
implies that $(M\cap Y)^U=T_{U_1,E+U_1}=T_{\la\omega^{q-1}\ra\lambda f_1}$.
Since $P$ is the kernel of $M$ acting on $U$, it follows that $M_{U_1,E+U_1}=(M\cap Y)P$.
As a consequence, $H_{U_1,E+U_1}=(H\cap Y)P$, and so
\begin{equation}\label{eqn:Xia-5}
(H\cap Y)/(P\cap Y)\cong(H\cap Y)P/P=H_{U_1,E+U_1}/P=S_{U_1,E+U_1}P/P\cong S_{U_1,E+U_1}.
\end{equation}
Note from~\eqref{eqn:Xia-7} that $S_{U_1,E+U_1}=S_{\la\omega^{q-1}\ra\lambda f_1}$. This together with~\eqref{eqn:Xia-5} yields
\[
H\cap B=H_{\la v\ra}=H\cap Y=(P\cap Y).S_{\la\omega^{q-1}\ra\lambda f_1}=P_v{:}S_{\la\omega^{q-1}\ra\lambda f_1},
\]
proving statement~\eqref{Xia:Unitary01a}. Moreover, since $H\cap B^{(\infty)}=(H\cap L)_v$ and $H\cap L=P{:}(S\cap L)$, we obtain by replacing both $\eta$ and $\omega^{q-1}$ by $1$ in the above proof that $H\cap B^{(\infty)}=P_v{:}(S\cap L)_{f_1}$, as statement~\eqref{Xia:Unitary01b} asserts.

Finally, it follows from statement~\eqref{Xia:Unitary01a} that
\[
\frac{|H|}{|H\cap B|}=\frac{|P||S|}{|P_v||S_{\la\omega^{q-1}\ra\lambda f_1}|}=\frac{q^{2m-1}|S|}{|S_{\la\omega^{q-1}\ra\lambda f_1}|}.
\]
Then since $\la\omega^{q-1}\ra\lambda f_1\in W_{(q+1)}$ and
\[
\frac{|G|}{|B|}=\frac{|X|}{|Y|}=\frac{|\GaU_{2m}(q)|}{|\GaU_{2m-1}(q)|(q+1)}=\frac{q^{2m-1}(q^{2m}-1)}{q+1}=q^{2m-1}|W_{(q+1)}|,
\]
we have $|H|/|H\cap B|=|G|/|B|$ if and only if $S$ is transitive on $W_{(q+1)}$, proving statement~\eqref{Xia:Unitary01c}.
\end{proof}

\begin{remark}\label{Xia:Unitary03}
The condition $|P|/|P\cap B|=q^{2m-1}$ in Proposition~\ref{Xia:Unitary01} holds if we take $P=R$. In fact, for each $r\in R_v$, since $r$ fixes $e_1$ and $v$, we deduce that $r$ fixes $\langle e_1,v\rangle=\langle e_1,f_1\rangle$ pointwise. Hence $R_v$ is isomorphic to the pointwise stabilizer in $\SU(\la e_2,f_2,\dots,e_m,f_m\ra)$ of $\la e_2,\dots,e_m\ra$. Then by~\cite[3.6.2]{Wilson2009} we have $R_v=q^{(m-1)^2}$, which gives $|R|/|R_v|=q^{m^2}/q^{(m-1)^2}=q^{2m-1}$, and so $|R|/|R\cap B|=|R|/|R_{\la v\ra}|=q^{2m-1}$.
\end{remark}

\subsection{Orthogonal groups of plus type}
\ \vspace{1mm}

Throughout this subsection, let $q=p^f$ be a power of a prime $p$, let $m\geqslant4$ be an integer,
%let $\,\overline{\phantom{\phi}}\,$ be the homomorphism from $\GaO_{2m}^+(q)$ to $\mathrm{P\Gamma O}_{2m}^+(q)$ modulo scalars,
let $V$ be a vector space of dimension $2m$ over $\bbF_q$ equipped with a nondegenerate quadratic form $Q$ of plus type.
%whose associated bilinear form is $\beta$, and let $\perp$ denote the perpendicularity with respect to $\beta$.
The orthogonal space $V$ has a standard basis $e_1,f_1,\dots,e_m,f_m$ as in~\cite[2.2.3]{LPS1990}. Let $\phi\in\GaO(V)=\GaO_{2m}^+(q)$ such that
\[
\phi\colon a_1e_1+b_1f_1+\dots+a_me_m+b_mf_m\mapsto a_1^pe_1+b_1^pf_1+\dots+a_m^pe_m+b_m^pf_m
\]
for $a_1,b_1\dots,a_m,b_m\in\bbF_q$, let $U=\langle e_1,\dots,e_m\rangle$, let $W=\langle f_1,\dots,f_m\rangle$, and let
\[
v=e_1+f_1.
\]
Then $Q(e_1+f_1)=1$. From~\cite[3.7.4~and~3.8.2]{Wilson2009} we see that $\GaO(V)_U$ has a subgroup $R{:}T$, where
\[
R=q^{m(m-1)/2}
\]
is the kernel of $\GaO(V)_U$ acting on $U$, and
\[
T=\GaL_1(q^m)
\]
stabilizes both $U$ and $W$. Note that the subgroup $\la-1\ra$ of $\bbF_q^\times$ has order $\gcd(2,q-1)$, and so $W_{(\gcd(2,q-1))}$ is the set of orbits of the scalar multiplication of $-1$ on $W\setminus\{0\}$.
%We omit the proof of the following proposition, as it resembles all the details in the proof of Proposition~\ref{Xia:Unitary01}.

\begin{proposition}\label{Xia:Omega01}
Let $\Omega_{2m}^+(q)=L\leqslant G\leqslant\GaO_{2m}^+(q)$, let $B=G_{\la v\ra}=\N_1[G]$, and let $H=P{:}S<G$ with $P\leqslant R$ and $S\leqslant T$ such that $|P|/|P\cap B|=q^{m-1}$. Then the following statements hold:
\begin{enumerate}[\rm(a)]
\item\label{Xia:Omega01d} $P\cap B=P\cap B^{(\infty)}$;
\item\label{Xia:Omega01a} $H\cap B=(P\cap B){:}S_{\la-1\ra f_1}$;
\item\label{Xia:Omega01b} $H\cap B^{(\infty)}=(P\cap B){:}(S\cap L)_{f_1}$;
\item\label{Xia:Omega01c} $G=HB$ if and only if $S$ is transitive on $W_{(\gcd(2,q-1))}$.
\end{enumerate}
\end{proposition}

\begin{proof}
Let $X=\GaO(V)=\GaO_{2m}^+(q)$, $Y=X_{\la v\ra}$, $M=P{:}T$, $U_1=\langle e_2,\dots,e_m\rangle$, and $E=\la-1\ra e_1$.
The proof of the proposition follows the same lines as that of Proposition~\ref{Xia:Unitary01}, by working with $Q$ instead of $\beta$.
\end{proof}

\begin{remark}\label{Xia:Omega02}
Similarly as in Remark~\ref{Xia:Unitary03}, the condition $|P|/|P\cap B|=q^{m-1}$ in Proposition~\ref{Xia:Omega01} holds if we take $P=R$.
\end{remark}

\subsection{Symplectic groups}
\ \vspace{1mm}

Throughout this subsection, let $q=2^f$ be a power of $2$, let $m\geqslant2$ be an integer with $(m,q)\neq(2,2)$, let $V$ be a vector space of dimension $2m$ over $\bbF_q$ equipped with a nondegenerate alternating form $\beta$, and let $\perp$ denote the perpendicularity with respect to $\beta$. The symplectic space $V$ has a standard basis $e_1,f_1,\dots,e_m,f_m$ as in~\cite[2.2.3]{LPS1990}.
Let $\phi\in\GaSp(V)=\GaSp_{2m}(q)$ such that
\begin{equation}\label{eqn:Xia-18}
\phi\colon a_1e_1+b_1f_1+\dots+a_me_m+b_mf_m\mapsto a_1^2e_1+b_1^2f_1+\dots+a_m^2e_m+b_m^2f_m
\end{equation}
for $a_1,b_1\dots,a_m,b_m\in\bbF_q$, let $U=\langle e_1,\dots,e_m\rangle_{\bbF_q}$, and let $W=\langle f_1,\dots,f_m\rangle_{\bbF_q}$. From~\cite[3.5.4]{Wilson2009} we see that $\GaSp(V)_U=\Pa_m[\GaSp(V)]$ has a subgroup $R{:}T$, where
\[
R=q^{m(m+1)/2}
\]
is the kernel of $\GaSp(V)_U$ acting on $U$, and
\[
T=\GaL_1(q^m)
\]
stabilizes both $U$ and $W$. Take $\mu\in\bbF_q$ such that the polynomial $x^2+x+\mu$ is irreducible over $\bbF_q$.
Let $Q$ be a nondegenerate quadratic form of minus type (namely, elliptic form) on $V$ with associated bilinear form $\beta$ such that $e_1,f_1,\dots,e_m,f_m$ is a standard basis for the orthogonal space $(V,Q)$. Then we have
\[
Q(e_i)=Q(f_i)=0,\quad Q(e_m)=1,\quad Q(f_m)=\mu
\]
for $i\in\{1,\dots,m-1\}$.

\begin{proposition}\label{Xia:Symplectic01}
Let $\Sp_{2m}(q)=L\leqslant G\leqslant\GaSp_{2m}(q)$ with $q$ even, let $B=G\cap\GaO(V,Q)$, and let $H=P{:}S<G$ with $P\leqslant R$ and $S\leqslant T$ such that $|P|/|P\cap B|=q^m/2$. Then the following statements hold:
\begin{enumerate}[\rm(a)]
\item\label{Xia:Symplectic01a} $H\cap B=(P\cap B){:}S_{f_m}$;
\item\label{Xia:Symplectic01b} $|P|/|P\cap B^{(\infty)}|=q^m/2$ or $q^m$;
\item\label{Xia:Symplectic01c} if $|P|/|P\cap B^{(\infty)}|=q^m/2$, then $H\cap B^{(\infty)}=(P\cap B){:}(S\cap B^{(\infty)})_{f_m}$;
\item\label{Xia:Symplectic01d} if $|P|/|P\cap B^{(\infty)}|=q^m$, then $H\cap B^{(\infty)}=(P\cap B^{(\infty)}){:}(S\cap L)_{f_m}$;
\item\label{Xia:Symplectic01e} $G=HB$ if and only if $S$ is transitive on $W_{(1)}$.
\end{enumerate}
\end{proposition}

\begin{proof}
Let $X=\GaSp(V)=\GaSp_{2m}(q)$, $Y=\GaO(V,Q)=\GaO_{2m}^-(q)$, $M=P{:}T$ and $U_1=\langle e_1,\dots,e_{m-1}\rangle$.
Note that, if $v\in V$ with $Q(v)\in\bbF_2$, then $Q(v^g)=Q(v)$ for all $g\in Y$. Then as $M$ stabilizes $U$, the subgroup $M\cap Y$ stabilizes the sets $\{u\in U\mid Q(u)=0\}=U_1$ and $\{u\in U\mid Q(u)=1\}=U_1+e_m$. Hence $(M\cap Y)^U\leqslant T_{U_1,U_1+e_m}$.
Note that $P$ is the kernel of $M$ acting on $U$, and so $P\cap Y$ is the kernel of $M\cap Y$ acting on $U$. We then have
\begin{equation}\label{eqn:Xia-26}
(M\cap Y)/(P\cap Y)\cong(M\cap Y)^U\leqslant T_{U_1,U_1+e_m}.
\end{equation}
Let $t$ be an arbitrary element in $T$ with integer $j$ and element $s$ in $\Sp(V)$. Note that the vector $f_m$ is the unique one in $W$ such that $\beta(x,f_m)=1$ for all $x\in U_1+e_m$. Thus, if $t$ stabilizes $U_1+e_m$, then $t$ stabilizes $f_m$. Conversely, suppose that $t$ stabilizes $f_m$. Then $t$ stabilizes $U\cap f_m^\perp=U_1$.
Moreover, the set of $y\in U$ such that $\beta(y,f_m)=1$ is $U_1+e_m$.
Accordingly, $t$ stabilizes $U_1+e_m$. This shows that
\begin{equation}\label{eqn:Xia-27}
T_{U_1,U_1+e_m}=T_{f_m},
\end{equation}
which in conjunction with~\eqref{eqn:Xia-26} and
\[
\frac{|M\cap Y|}{|P\cap Y|}\geqslant\frac{|M||Y|}{|P\cap Y||X|}=\frac{|P||T||Y|}{|P\cap B||X|}=\frac{|T|}{q^m-1}=|T_{f_m}|
\]
implies that $(M\cap Y)^U=T_{f_m}$. Since $P$ is the kernel of $M$ acting on $U$, it follows that $M_{U_1,U_1+e_m}=(M\cap Y)P$.
As a consequence, $H_{U_1,U_1+e_m}=(H\cap Y)P$, and so
\begin{equation}\label{eqn:Xia-25}
(H\cap Y)/(P\cap Y)\cong(H\cap Y)P/P=H_{U_1,U_1+e_m}/P=S_{U_1,U_1+e_m}P/P\cong S_{U_1,U_1+e_m}.
\end{equation}
Note from~\eqref{eqn:Xia-27} that $S_{U_1,U_1+e_m}=S_{f_m}$. This together with~\eqref{eqn:Xia-25} yields
\[
H\cap B=H\cap Y=(P\cap Y).S_{f_m}=(P\cap B){:}S_{f_m},
\]
proving statement~\eqref{Xia:Symplectic01a}.

Since $P\leqslant R<L$ and $B^{(\infty)}=\Omega(V,Q)$ has index $2$ in $\mathrm{O}(V,Q)=L\cap B$, it follows that $|P\cap B|/|P\cap B^{(\infty)}|=1$ or $2$. Hence we derive from $|P|/|P\cap B|=q^m/2$ that $|P|/|P\cap B^{(\infty)}|=q^m$ or $q^m/2$, proving statement~\eqref{Xia:Symplectic01b}. If $|P\cap B|/|P\cap B^{(\infty)}|=1$, that is, $P\cap B<B^{(\infty)}$, then
\begin{align*}
H\cap B^{(\infty)}=(H\cap B)\cap B^{(\infty)}&=((P\cap B){:}S_{f_m})\cap B^{(\infty)}\\
&=(P\cap B){:}(S_{f_m}\cap B^{(\infty)})=(P\cap B){:}(S\cap B^{(\infty)})_{f_m},
\end{align*}
as statement~\eqref{Xia:Symplectic01c} asserts. If $|P\cap B|/|P\cap B^{(\infty)}|=2$, then replacing $X$, $Y$, $M$ and $H$ by $L$, $B^{(\infty)}$, $M\cap L$ and $H\cap L$, respectively, in the proof of statement~\eqref{Xia:Symplectic01a}, we obtain $H\cap B^{(\infty)}=(P\cap B^{(\infty)}){:}(S\cap L)_{f_m}$, as statement~\eqref{Xia:Symplectic01d} asserts.

Finally, it follows from statement~\eqref{Xia:Symplectic01a} that
\[
\frac{|H|}{|H\cap B|}=\frac{|P||S|}{|P\cap B||S_{f_m}|}=\frac{q^m|S|}{2|S_{f_m}|}.
\]
Then since $|G|/|B|=|X|/|Y|=2q^m(q^m-1)/2$, we have $|H|/|H\cap B|=|G|/|B|$ if and only if $S$ is transitive on $W$. Thus statement~\eqref{Xia:Symplectic01e} holds.
\end{proof}

\begin{remark}\label{rem:Xia-4}
If we take $P=R$ in Proposition~\ref{Xia:Symplectic01}, then $P\cap B^{(\infty)}=R\cap\Omega(V,Q)$ is the pointwise stabilizer of $\la e_1,\dots,e_{m-1}\ra$ in $\Omega(V,Q)_{e_m}=\Omega_{2m-1}(q)$. As the (setwise) stabilizer in $\Omega(V,Q)_{e_m}$ of $\la e_1,\dots,e_{m-1}\ra$ is $\Pa_{m-1}[\Omega_{2m-1}(q)]$, it follows that $R\cap B^{(\infty)}=q^{m(m-1)/2}$, and so $|R|/|R\cap B^{(\infty)}|=q^m$. Moreover, since the reflection on $V$ with respect to $e_m$ is an element in $R\cap\mathrm{O}(V,Q)=R\cap B$, we have $|R\cap B|=2|R\cap B^{(\infty)}|$. Thus $|R|/|R\cap B|=q^m/2$.
\end{remark}

\section{Unipotent radical}\label{sec:Xia-7}

Throughout this section, let $q=p^f$ be a power of a prime $p$, and let $m\geqslant2$ be an integer satisfying the conditions in Table~\ref{Form}. For convenience, we introduce parameters $s\in\{1,2\}$ and $\epsilon\in\{-1,-1/2,0\}$ as in Table~\ref{Form}, and set
\[
r=q^s.
\]
Let $V=\bbF_{r^m}\times\bbF_{r^m}$, and for positive integers $i$ and $j$, let $\Tr_{q^{ij}/q^i}$ denote the relative trace function from $\bbF_{q^{ij}}$ to $\bbF_{q^i}$.
Consider the form $\kappa_\epsilon$ on $V$ as defined in Table~\ref{Form}, where $a,b,c,d$ are arbitrary in $\bbF_{r^m}$.

\begin{table}[h]
\centering\caption{The parameters and form of the isometry group $G_0$}\label{Form}
\begin{tabular}{|c|c|c|c|c|}
\hline
$G_0$ & $s$ & $\epsilon$ & $\kappa_\epsilon$ & Conditions\\
\hline
$\GU_{2m}(q)$ & $2$ & $-1/2$ & $\kappa_{-1/2}((a,b),(c,d))=\Tr_{r^m/r}(ad^q+b^rc^q)$ & $(m,q)\neq(2,2)$ \\
$\mathrm{O}_{2m}^+(q)$ & $1$ & $-1$ & $\kappa_{-1}((a,b))=\Tr_{r^m/r}(ab)$ & $m\geqslant4$ \\
$\Sp_{2m}(q)$ & $1$ & $0$ & $\kappa_{0}((a,b),(c,d))=\Tr_{r^m/r}(ad+bc)$ & $q$ even, $(m,q)\neq(2,2)$ \\
\hline
\end{tabular}
\end{table}

For our purpose, regard $V$ as a $2m$-dimensional vector space over $\bbF_r$. Then $\kappa_{-1/2}$ is a nondegenerate Hermitian form,  $\kappa_{-1}$ is a nondegenerate quadratic form, and $\kappa_0$ is a nondegenerate alternating form. Hence $G_0$ is the isometry group of the polar space $(V,\kappa_\epsilon)$, and in the notation of Section~\ref{sec:Xia-1}, the pointwise stabilizer of $U=\bbF_{r^m}\times\{0\}$ in $G_0$ is the unipotent radical $R$ of $(G_0)_U=\Pa_m[G_0]$.
The aim of this section is to describe the subgroup $H\cap R$ of $R$ for the solvable factor $H$.

Note from the conditions in Table~\ref{Form} that we always have
\begin{equation}\label{eqn:Xia-10}
smf\geqslant3.
\end{equation}
Also, the parameter $\epsilon$ is defined in such a way that the number of $1$-dimensional totally isotropic (singular) subspaces of $V$ is $(r^{m+\epsilon}+1)(r^m-1)/(r-1)$.

\subsection{Construction of irreducible modules}\label{subsec_AuxIrrMod}
\ \vspace{1mm}

Regarding $R$ as a module of the Singer group (or any of its subgroup of index dividing $\gcd(smf,r^m-1)$) in $G_0^U=\GL_m(r)$, the key to describe submodules of $R$ is to decompose $R$ into a sum of irreducible submodules (such a decomposition exists by Maschke's theorem). In this subsection, we construct a decomposition of $R$ into a sum of irreducible $\bbF_{r^m}^\times$-modules. However, it remains to show that the constructed action of $\bbF_{r^m}^\times$ on $R$ here is the same as that of the Singer group in $G_0^U=\GL_m(r)$.
This will be shown in the next subsection by embedding $\bbF_{r^m}^\times$ into $G_0^U$ to establish its compatibility with the action of $G_0^U$ on $R$.

To construct irreducible $\bbF_{r^m}^\times$-modules, we work in the vector space
\[
\mc{L}_m(\bbF_r)=\left\{\sum_{i=0}^{m-1}a_iX^{r^i}\;\Big|\; a_0,\ldots,a_{m-1}\in\bbF_{r^m}\right\}\subseteq\bbF_{r^m}[X]/(X^{r^m}-X)
\]
of dimension $sm^2$ over $\bbF_q$ (elements of $\mc{L}_m(\bbF_r)$ are called \emph{linearized polynomials} over $\bbF_{r^m}$), and take the following $\bbF_q$-subspaces of $\mc{L}_m(\bbF_r)$.
\begin{align*}
&M(0)=\{aX\mid a\in\bbF_{r^m}\},\\
&M(i)=\{aX^{r^i}-a^{r^{m-i}q^{s-1}}X^{r^{m-i}q^{2s-2}}(\bmod\,X^{r^m}-X)\mid a\in\bbF_{r^m}\}\,
\text{ for }1\leqslant i\leqslant \left\lfloor(m+s)/2\right\rfloor-1,\\
&M\left((m+s-1)/2\right)=\big\{aX^{r^{m/2}q^{s-1}}\mid a\in\bbF_{r^m},\, a+a^{r^{m/2}}=0\big\}\,\text{ if $m+s$ is odd}.
\end{align*}
It is easy to see that
\begin{equation}\label{eqn:Xia-20}
\dim_{\bbF_q}M(i)=
\begin{cases}
sm&\text{if }0\leqslant i\leqslant\left\lfloor(m+s)/2\right\rfloor-1,\\
sm/2&\text{if }i=(m+s-1)/2\text{ with $m+s$ odd}.
\end{cases}
\end{equation}
Since $M(0),M(1),\ldots,M(\lfloor(m+s-1)/2\rfloor)$ are distinguished by the monomials occurring in their elements, their sum is a direct sum. Let
\begin{equation}\label{eqn_overM}
\overline{M}=\begin{cases}
M(1)\oplus\cdots\oplus M(\lfloor(m+s-1)/2\rfloor)&\text{if }\epsilon=-1 \text{ or }-1/2,\\
M(0)\oplus M(1)\oplus\cdots\oplus M(\lfloor(m+s-1)/2\rfloor)&\text{if }\epsilon=0.
\end{cases}
\end{equation}
Then it holds for all $\epsilon\in\{-1,-1/2,0\}$ that
\[
\dim_{\bbF_q}(\overline{M})=\frac{sm(m+1)}{2}+\epsilon ms=\dim_{\bbF_q}(R).
\]

To turn $\mc{L}_m(\bbF_r)$ into an $\bbF_{r^m}^\times$-module, define an action of $\bbF_{r^m}^\times$ on $\mc{L}_m(\bbF_r)$ by letting
\[
(a.\hbar)(X)=a^{q^{s-1}}\hbar(aX)\,\textup{ for }a\in \bbF_{r^m}^\times\text{ and }\hbar(X)\in\mc{L}_m(\bbF_r).
\]
It is routine to check that the $M(i)$'s are $\bbF_{r^m}^\times$-submodules of $\mc{L}_m(\bbF_r)$. Now we show that they are pairwise non-isomorphic irreducible $\bbF_{r^m}^\times$-modules.

\begin{lemma}\label{lem_MiIrr}
For $0\leqslant i\leqslant\lfloor(m+s-1)/2\rfloor$, let $\chi_i$ be the character of $\bbF_{r^m}^\times$ afforded by the module $M(i)$. Then each $M(i)$ is an irreducible $\bbF_{r^m}^\times$-module, and for $x\in\bbF_{r^m}^\times$ we have
\begin{equation*}
  \chi_i(x)=\begin{cases}
  \Tr_{r^m/q}(x^{q^{s-1}+1})&\textup{if } i=0,\\
  \Tr_{r^m/q}(x^{r^{i-1}q+1})&\textup{if } 1\leqslant i\leqslant\lfloor(m+s)/2\rfloor-1,\\
  \Tr_{r^{m/2}/q}(x^{r^{m/2}+1})&\textup{if $i=(m+s-1)/2$ with $m+s$ odd.}
  \end{cases}
\end{equation*}
In particular, $M(0),M(1),\ldots,M(\lfloor(m+s-1)/2\rfloor)$ are pairwise non-isomorphic as $\bbF_{r^m}^\times$-modules.
\end{lemma}

\begin{proof}
Fix a generator $\omega$ of $\bbF_{r^m}^\times$ and an integer $i$ with $1\leqslant i\leqslant \lfloor(m+s)/2\rfloor-1$. For $a\in\bbF_{r^m}$, let
\[
f_a^{(i)}=aX^{r^i}-a^{r^{m-i}q^{s-1}}X^{r^{m-i}q^{2s-2}}.
\]
Since $x.f_a^{(i)}=f_{a'}^{(i)}$ with $a'=x^{r^i+q^{s-1}}a$,  by \cite[Exercise~2.26]{LN97} we deduce that
\[
\chi_i(x)=\Tr_{r^m/q}\left(x^{r^{i}+q^{s-1}}\right)=\Tr_{r^m/q}\left(x^{r^{i-1}q+1}\right).
\]

Let $\mathcal{F}_i$ be the $\bbF_{q}$-span of $\{\mu^{r^i+q^{s-1}}\mid\mu\in\bbF_{r^m}^\times\}$, which is an $\bbF_{q}$-subspace of $\bbF_{r^m}$. We claim that $\mathcal{F}_i=\bbF_{r^m}$. The finite set $\mathcal{F}_i$ is closed under addition and multiplication, and so is a subfield of $\bbF_{r^m}$ that contains $\bbF_{q}$. It is straightforward to check that $(r^{i-1}q+1)(r^{m/2}-1)<r^m-1$ for $i\leqslant\lfloor(m+s)/2\rfloor-1$. Therefore, $\omega^{r^i+q^{s-1}}$ has order strictly larger than $r^{m/2}-1$, and so does not lie in any proper subfield of $\bbF_{r^m}$. The claim $\mathcal{F}_i=\bbF_{r^m}$ then follows.

For any nonzero element $f_a^{(i)}$ in $M(i)$, we have  $\bbF_{r^m}.f_a^{(i)}=\{f_b^{(i)}\mid b\in \mathcal{F}_i\cdot a\}=M(i)$ by the fact $\mathcal{F}_i=\bbF_{r^m}$. Thus we conclude that $M(i)$ is irreducible for $1\leqslant i\leqslant\lfloor(m+s)/2\rfloor-1$.

The remaining $M(i)$'s are handled in a similar way, and we omit the details. The $\bbF_{r^m}^\times$-modules $M(0),M(1),\ldots,M(\lfloor(m+s-1)/2\rfloor)$ are irreducible  with distinct characters, whence they are pairwise non-isomorphic. This completes the proof.
\end{proof}

The conclusion in Lemma \ref{lem_MiIrr} that $M(0),\ldots,M(\lfloor(m+s-1)/2\rfloor)$ are pairwise non-isomorphic irreducible $\bbF_{r^m}^\times$-modules can be strengthened to the next result. We mention that the condition
\begin{equation}\label{eqn_exception}
(s,m,q)\notin\{(2,3,2), (1,2,8), (1,4,3), (1,6,2)\}
\end{equation}
excludes the groups $\SU_6(2)$, $\Sp_4(8)$, $\Omega_8^+(3)$, $\Omega_{12}^+(2)$ and $\Sp_{12}(2)$ as Hypothesis~\ref{hypo-1} does.

\begin{lemma}\label{lem_MiIrrNd0}
Suppose \eqref{eqn_exception} and let $N$ be a subgroup of index dividing $\gcd(smf,r^m-1)$ in $\bbF_{r^m}^\times$. Then $M(0),M(1),\ldots, M(\lfloor(m+s-1)/2\rfloor)$ are pairwise non-isomorphic irreducible $N$-modules.
\end{lemma}

\begin{proof}
Let $d=\gcd(smf,r^m-1)$ and fix a generator $\omega$ of $\bbF_{r^m}^\times$. It suffices to prove the lemma under the assumption that $N$ has index $d$ in $\bbF_{q^{sm}}^\times$. For $0\leqslant i\leqslant\lfloor(m+s-1)/2\rfloor$, let $\omega_0=\omega^{r^i+q^{s-1}}$ if $i\leqslant\lfloor(m+s)/2\rfloor-1$ and let $\omega_0=\omega^{r^{m/2}+1}$ otherwise. Let $\mathcal{F}_i=\bbF_q[\omega_0^d]$ be the smallest extension field of $\bbF_q$ that contains $\omega_0^d$. By the same arguments as in the proof of Lemma \ref{lem_MiIrr}, the $M(i)$'s are irreducible $N$-modules if and only if the following claims hold:
\begin{enumerate}[\rm(a)]
\item\label{MiIrrNd0a} $\mathcal{F}_i=\bbF_{r^m}$ for $0\leqslant i\leqslant \lfloor(m+s)/2\rfloor-1$;
\item\label{MiIrrNd0b} $\mathcal{F}_i=\bbF_{r^{m/2}}$ if $m+s$ is odd and $i=(m+s-1)/2$.
\end{enumerate}
Recall from~\eqref{eqn:Xia-10} that $smf\geqslant3$.

First assume that $0\leqslant i\leqslant\lfloor(m+s)/2\rfloor-1$. If $p^{smf}-1$ has no primitive prime divisor, then by~\cite{Zsigmondy1892}, $p=2$ and $smf=6$. In this case, $d=3$, and after examining the possible $(s,m,f)$ triples, we see that $\mathcal{F}_i=\bbF_{r^m}$ holds unless $(s,m,f)$ is one of $(1,2,3),(2,3,1),(1,6,1)$. However, this gives the triple $(s,m,q)$ contradicting~\eqref{eqn_exception}. Now we consider the case where $p^{smf}-1$ has a primitive prime divisor, say, $t$. In this case, $t\geqslant 1+smf$, and so $\gcd(d,t)=1$. Suppose to the contrary that $\omega_0^d=\omega^{d(r^i+q^{s-1})}$ is in a proper subfield $\bbF_{q^{d_0}}$ of $\bbF_{r^m}$. Then $d_0<sm$, and
\[
(r^{m}-1)\mid d(r^i+q^{s-1})(q^{d_0}-1).
\]
In particular, $i\geqslant 1$. It follows that $t$ divides $r^i+q^{s-1}=q^{s-1}(q^{si-s+1}+1)$ and hence divides $q^{2(si-s+1)}-1$, which implies that $smf\leqslant2f(si-s+1)$, that is, $sm\leqslant2(si-s+1)$. However, this does not hold for $1\leqslant i\leqslant\lfloor(m+s)/2\rfloor-1$, a contradiction. This establishes~\eqref{MiIrrNd0a}.

Next assume that $m+s$ is odd and $i=(m+s-1)/2$. Then $\omega_0=\omega^{r^{m/2}+1}$ is a generator of $\bbF_{r^{m/2}}^\times$. If $p^{smf/2}-1$ has a primitive prime divisor, then similar argument as above proves~\eqref{MiIrrNd0b}. Thus, by~\cite{Zsigmondy1892}, it remains to consider the cases $smf=4$ and $(p,smf)=(2,12)$, respectively. If $smf=4$, then~\eqref{MiIrrNd0b} is equivalent to $\bbF_{p^f}[\omega_0^{d}]=\bbF_{p^2}$. This does not hold if and only if $f=1$ and $(p^2-1)/d$ divides $p-1$, and the latter occurs only for the tuple $(s,m,q)=(1,4,3)$, which is excluded by~\eqref{eqn_exception}. If $(p,smf)=(2,12)$, then $d=3$, and~\eqref{MiIrrNd0b} is equivalent to $\bbF_{2^f}[\omega_0^{3}]=\bbF_{2^6}$. The latter always holds, since $\omega_0^{3}$ has order $21$ and does not lie in $\bbF_{2^2}$ or $\bbF_{2^3}$.

We have now shown that each $M(i)$ is an irreducible $N$-module. The claim that they are pairwise non-isomorphic follows by comparing their characters as in Lemma \ref{lem_MiIrr}. This completes the proof.
\end{proof}

\subsection{Decomposition of unipotent radical}\label{subsec_form}
\ \vspace{1mm}

Recall the definition of $\mc{L}_m(\bbF_r)$ in Subsection~\ref{subsec_AuxIrrMod}. We consider its subset
\[
\mc{L}_m^\mathrm{P}(\bbF_r)=\{\hbar\in\mc{L}_m(\bbF_r)\mid \hbar \text{ induces a permutation of $\bbF_{r^m}$}\}.
\]
It is well known that an element $\hbar\in\mc{L}_m(\bbF_r)$ induces a permutation of $\bbF_{r^m}$ if and only if $0$ is its only root in $\bbF_{r^m}$. For two elements $\hbar_1$ and $\hbar_2$ in $\mc{L}_m^\mathrm{P}(\bbF_r)$, we define
\[(\hbar_1\circ \hbar_2)(X)=\hbar_1(\hbar_2(X))\pmod{X^{q^m}-X}.\]
Then $(\mc{L}_m^\mathrm{P}(\bbF_r),\circ)$ forms a group isomorphic to $\GL_m(r)$. For  $\hbar\in\mc{L}_m^\mathrm{P}(\bbF_r)$, write $\hbar^{-1}$ for its inverse. We refer the reader to \cite[Chapter 3.4]{LN97} for more details.

\begin{lemma}\label{prop_hbars}
For each $\hbar\in \mc{L}_m^\mathrm{P}(\bbF_r)$, there is a unique element $\hbar^{(s)}\in\mc{L}_m^\mathrm{P}(\bbF_r)$ such that
\begin{equation}\label{eqn_tt2LP}
\Tr_{r^m/r}\left(\hbar(x)\big(\hbar^{(s)}(y)\big)^{q^{s-1}}\right)=\Tr_{r^m/r}(xy^{q^{s-1}})\,\textup{ for all }x,y\in\bbF_{r^m}.
\end{equation}
In particular, if $\hbar(X)=aX$ for some $a\in\bbF_{r^m}^\times$, then $\hbar^{(s)}(X)=a^{-r^{m-1}q}X$. Moreover, we have $(\hbar_1\circ\hbar_2)^{(s)}=\hbar_1^{(s)}\circ\hbar_2^{(s)}$ and $(\hbar^{-1})^{(s)}=(\hbar^{(s)})^{-1}$ for $\hbar_1,\hbar_2,\hbar\in\mc{L}_m^\mathrm{P}(\bbF_r)$.
\end{lemma}

\begin{proof}
For brevity, write $\Tr=\Tr_{r^m/r}$ in this proof. Take an arbitrary $\hbar\in \mc{L}_m^\mathrm{P}(\bbF_r)$, and write $\hbar^{-1}(X)=\sum_{i=0}^{m-1}b_iX^{r^i}$. By setting $z=\hbar(x)$ and raising both sides to the $q$-th power, the condition \eqref{eqn_tt2LP} translates to $\Tr(z^{q}(\hbar^{(s)}(y))^r)=\Tr((\hbar^{-1}(z))^{q}y^r)$ for all $y,z\in\bbF_{r^m}$. We have
\begin{align*}
\Tr\left(\big(\hbar^{-1}(x)\big)^qy^r\right)=\sum_{i=0}^{m-1}\Tr\left(b_i^{q}z^{r^iq}y^r\right)
&=\sum_{i=0}^{m-1}\Tr\left(b_i^{r^{m-i}q}y^{r^{m-i}+1}z^{q}\right)\\
&=\sum_{j=1}^{m}\Tr\left(b_{m-j}^{r^jq}y^{r^j+1}z^{q}\right)=\Tr\left(z^{q}\big(f(y)\big)^r\right),
\end{align*}
where $f(X)=\sum_{j=0}^{m-1}b_{m-j}^{r^{j-1}q}X^{r^j}$ with $b_m=b_0$. Suppose that $f(y)=0$ for some $y\in\bbF_{r^m}^\times$. Then $\Tr((\hbar^{-1}(x))^qy^r)=0$ for all $x\in\bbF_{r^m}$. Since $x\mapsto(\hbar^{-1}(x))^q$ is a permutation of $\bbF_{r^m}$, we deduce that $\Tr(ay)=0$ for all $a\in\bbF_{r^m}$, a contradiction. Therefore, $f(X)$ belongs to $\mc{L}_m^\mathrm{P}(\bbF_r)$.

Suppose that $f_0(X)$ is an element of $\mc{L}_m^\mathrm{P}(\bbF_r)$ distinct from $f(X)$ satisfying the condition $\Tr(z^{q}(f_0(y))^r)=\Tr((\hbar^{-1}(z))^qy^r)$ for all $y,z\in\bbF_{r^m}$. It follows that $\Tr(z^{r}(f_0(y)-f(y))^q)=0$ for all $y,z\in\bbF_{r^m}$. Since $f-f_0$ is a nonzero reduced polynomial, there is an element $y_0\in\bbF_{r^m}$ such that $c=f_0(y_0)-f(y_0)\ne 0$. It follows that $\Tr(c^rz^{q})=0$ for all $z\in\bbF_{r^m}$, a contradiction. Hence $\hbar^{(s)}(X)=f(X)$ is the unique element desired in the lemma.

If $\hbar(X)=aX$, then $\hbar^{-1}(X)=a^{-1}X$ and so $\hbar^{(s)}(X)=a^{-r^{m-1}q}X$ by the previous analysis. Take arbitrary $\hbar_1,\hbar_2\in\mc{L}_m^\mathrm{P}(\bbF_r)$. For all $x,y\in\bbF_{r^m}$, we have
\begin{align*}
\Tr\left(\hbar_1(\hbar_2(x))\Big(\hbar_1^{(s)}\big(\hbar_2^{(s)}(y)\big)\Big)^{q^{s-1}}\right)
=\Tr\left(\hbar_2(x)\Big(\hbar_2^{(s)}(y)\Big)^{q^{s-1}}\right)=\Tr(xy^{q^{s-1}}).
\end{align*}
It follows that $(\hbar_1\circ\hbar_2)^{(s)}=\hbar_1^{(s)}\circ\hbar_2^{(s)}$ by the above uniqueness result. Then taking $\hbar_1=\hbar$ and $\hbar_2=\hbar^{-1}$, we deduce that $(\hbar^{-1})^{(s)}=(\hbar^{(s)})^{-1}$, completing the proof.
\end{proof}

For each element $\hbar\in\overline{M}$, there is a vector $\bfa=(a_0,\ldots,a_{m-1})\in\bbF_{r^m}^m$ such that $\hbar=\hbar_{\bfa}$, where
\begin{equation}\label{eqn:Xia-14}
\hbar_{\bfa}=\hbar_{\bfa}(X):=\sum_{i=0}^{m-1}a_iX^{r^i}.
\end{equation}

\begin{lemma}\label{lem_Mprop}
Let $\hbar=\hbar_\bfa\in\overline{M}$ with $\bfa=(a_0,\ldots,a_{m-1})\in\bbF_{r^m}^m$. The following statements hold:
\begin{enumerate}[{\rm(a)}]
\item\label{lem_Mpropa} $a_{m-i+2(1-s^{-1})}+a_i^{r^{m-i}q^{s-1}}=0$ for $1\leqslant i\leqslant \lfloor(m+s-1)/2\rfloor$, where the subscripts of $a$ are taken modulo $m$;
\item\label{lem_Mpropb} if $\epsilon=-1$, then $a_0=0$;
\item\label{lem_Mpropc} if $\epsilon=0$, then $\Tr_{q^m/q}(\hbar(x)y+\hbar(y)x)=0$ for all $x,y\in\bbF_{r^m}$;
\item\label{lem_Mpropd} if $\hbar\in M(1)\oplus\cdots\oplus M(\lfloor(m+s-1)/2\rfloor)$, then $\Tr_{r^m/q}(\hbar(x)x^{q^{s-1}})=0$ for all $x\in\bbF_{r^m}$.
\end{enumerate}
\end{lemma}

\begin{proof}
Statements~\eqref{lem_Mpropa} and~~\eqref{lem_Mpropb} can be deduced directly from the definition of $\overline{M}$. The proofs of the statements~\eqref{lem_Mpropc} and~\eqref{lem_Mpropd} are similar, so here we only include the proof for~\eqref{lem_Mpropd} in the case $\epsilon=-1$. Suppose that $\hbar\in M(1)\oplus\cdots\oplus M(\lfloor(m+s-1)/2\rfloor)$ and $\epsilon=-1$. Then $s=1$, $a_0=0$ and $a_{m-i}+a_i^{q^{m-i}}=0$ for $1\leqslant i\leqslant \lfloor m/2\rfloor$.
Note that $\Tr_{q^m/q}(\hbar(x)x)=\sum_{i=1}^{m-1}\sum_{j=0}^{m-1}a_i^{q^j}x^{q^{i+j}+q^j}$ for all $x\in\bbF_{r^m}$. If $i\neq m-i$, then
\begin{align*}
\sum_{j=0}^{m-1}a_i^{q^j}x^{q^{i+j}}x^{q^j}+\sum_{j=0}^{m-1}a_{m-i}^{q^j}x^{q^{m-i+j}}x^{q^j}
=&\sum_{j=0}^{m-1}a_i^{q^j}x^{q^{i+j}}x^{q^j}-\sum_{j=0}^{m-1}a_{i}^{q^{m-i+j}}x^{q^{m-i+j}}x^{q^j}\\
=&\sum_{j=0}^{m-1}a_i^{q^j}x^{q^{i+j}}x^{q^j}-\sum_{k=0}^{m-1}a_i^{q^k}x^{q^k}x^{q^{i+k}}=0.
\end{align*}
If $i=m-i$, then $a_i+a_i^{q^i}=0$ and hence
\[\sum_{j=0}^{m-1}a_i^{q^j}x^{q^{i+j}}x^{q^j}
=\sum_{j=0}^{i-1}a_i^{q^j}x^{q^{i+j}}x^{q^j}+\sum_{j=0}^{i-1}a_i^{q^{i+j}}x^{q^{i+j}}x^{q^j}
=\sum_{j=0}^{i-1}(a_i+a_i^{q^i})^{q^j}x^{q^{i+j}}x^{q^j}=0.\]
Therefore, $\Tr_{r^m/q}(\hbar(x)x^{q^{s-1}})=\Tr_{q^m/q}(\hbar(x)x)=0$, as required.
\end{proof}
%If $\epsilon=0$ and $\hbar_\bfa\in\overline{M}$, then $\hbar'_{\bfa}(X):=\hbar_{\bfa}(X)-a_0X\in M$. Moreover, it can be shown similarly that $\Tr_{q^m/q}(y(\hbar'_{\bfa}(y^{q^m/2})^2)=0$ for all $y\in\bbF_{q^m}$.

For each $\hbar_\bfa\in\overline{M}$ with $\bfa\in\bbF_{r^m}^m$ and for each $\hbar\in \mc{L}_m^\mathrm{P}(\bbF_r)$, we define
\begin{align*}
u_\bfa:&\,(x,y)\mapsto(x+\hbar_\bfa(y),y)\;\textup{ for } (x,y)\in V, \\
\ell_{\hbar}:&\, (x,y)\mapsto(\hbar(x),\hbar^{(s)}(y))\;\textup{ for } (x,y)\in V,
\end{align*}
where $\hbar_\bfa$ is as in~\eqref{eqn:Xia-14} and $\hbar^{(s)}$ is as in Lemma~\ref{prop_hbars}.
Both $u_\bfa$ and $\ell_{\hbar}$ are nondegenerate linear transformations of $V$, and they are isometries of $(V,\kappa_\epsilon)$ by Lemma \ref{lem_Mprop} and Lemma~\ref{prop_hbars}. Define
\begin{equation}\label{eqn_RLevi}
R_m=\{u_\bfa\mid \hbar_\bfa\in \overline{M}\}\ \text{ and }\ L_m=\{\ell_{\hbar}\mid \hbar\in\mc{L}_m^\mathrm{P}(\bbF_r)\}.
\end{equation}
They both lie in the stabilizer $(G_0)_U=\Pa_m[G_0]$ of $U=\bbF_{r^m}\times\{0\}$ in the isometry group $G_0$ of $(V,\kappa_\epsilon)$.
Since $|R_m|=|\overline{M}|=r^{\epsilon m+m(m+1)/2}$ and $L_m\cong\GL_m(r)$, we have $(G_0)_U=R_m{:}L_m$ by comparing orders.
Thus we may (and will) identify $R_m$ with $R$, the unipotent radical of $(G_0)_U=\Pa_m[G_0]$.

Let $\omega$ be a generator of $\bbF_{r^m}^\times$, and for $a\in\bbF_{r^m}^\times$, let $\rho_a\in\mc{L}_m^\mathrm{P}(\bbF_r)$ be such that $\rho_a(X)=a^{-q^{s-1}}X$.  For brevity, we write $\ell_a=\ell_{\rho_a}$ for $a\in\bbF_{r^m}^\times$. As Lemma~\ref{prop_hbars} asserts,
\[
\ell_\omega((x,y))=(\omega^{-q^{s-1}}x,\omega y)\,\textup{ for }(x,y)\in V.
\]
Define a Singer group $S_0$ of the Levi subgroup $L_m$ by taking
\begin{equation*}%\label{eqn_Cdef}
S_0=\la\ell_\omega\ra.
\end{equation*}
The action of $S_0$ on $R$ via conjugation is $\ell_\omega.t_\bfa=t_{\bfa'}$, where $\hbar_{\bfa'}(X)=\omega^{q^{s-1}}\hbar_\bfa(\omega X)$. This action coincides with that of $\bbF_{r^m}^\times$ on $\overline{M}$ by identifying $S_0$ and $R$ with $\bbF_{r^m}^\times$ and $\overline{M}$ respectively.
For each $M(i)$ contained in $\overline{M}$, define a subgroup $U(i)$ of the unipotent radical $R$ by letting
\[
U(i)=\{u_{\bfa}\mid\hbar_\bfa\in M(i)\}.
\]
%According to the dimension of $M(i)$ over $\bbF_q$, we have
%\begin{align*}
%&U(i)=q^{sm}\,\text{ for }0\leqslant i\leqslant \left\lfloor(m+s)/2\right\rfloor-1, \\
%&U\left((m+s-1)/2\right)=q^{sm/2}\,\text{ if $m+s$ is odd}.
%\end{align*}
The result below is an immediate corollary of~\eqref{eqn:Xia-20} and Lemmas~\ref{lem_MiIrr} and~\ref{lem_MiIrrNd0}.

\begin{corollary}\label{Cor_UiIrr}
Let $\mathcal{U}=\{U(i)\mid0\leqslant i\leqslant\lfloor(m+s-1)/2\rfloor,\,U(i)\leqslant R\}$. Then the following statements hold:
\begin{enumerate}[{\rm(a)}]
\item $U(i)=q^{sm}$ for $0\leqslant i\leqslant\lfloor(m+s)/2\rfloor-1$, and $U((m+s-1)/2)=q^{sm/2}$ if $m+s$ is odd;
\item for $U(i)\in\mathcal{U}$, the character $\chi_i$ of $S_0$ afforded by the module $U(i)$ satisfies
    \begin{equation*}%\label{eqn_chUi2}
    \chi_i(x)=\begin{cases}
    \Tr_{r^m/q}(x^{q^{s-1}+1})&\textup{if } i=0,\\
    \Tr_{r^m/q}(x^{r^{i-1}q+1})&\textup{if } 1\leqslant i\leqslant \lfloor(m+s)/2\rfloor-1,\\
    \Tr_{r^{m/2}/q}\left(x^{r^{m/2}+1}\right)&\textup{if $i=(m+s-1)/2$ with $m+s$ odd;}
    \end{cases}
    \end{equation*}
\item the groups in $\mathcal{U}$ are  all the minimal  $S_0$-invariant subgroups of the unipotent radical $R$, and they are pairwise non-isomorphic as $S_0$-modules;
\item if~\eqref{eqn_exception} holds and $N$ is a subgroup of $S_0$ of index dividing $\gcd(smf,r^m-1)$, then an $N$-invariant subgroup of $R$ is a direct sum of some $U(i)$'s in $\mathcal{U}$.
\end{enumerate}
\end{corollary}

\begin{remark}\label{rem:Xia-3}
For $1\leqslant i\leqslant\lceil m/2\rceil$ in unitary case, the number $2i-1$ completely determines an irreducible $\bbF_qS_0$-module up to the action of ${\rm Gal}(\bbF_{q^{2m}}/\bbF_q)$. This is why we label this corresponding module as $U(i)$.
Similarly, for $1\leqslant i\leqslant\lfloor m/2\rfloor$ in orthogonal case or symplectic case, the number $i$ determines
an irreducible $\bbF_qS_0$-module up to the action of ${\rm Gal}(\bbF_{q^m}/\bbF_q)$, labelled as $U(i)$.
\end{remark}

\subsection{Submodules of unipotent radical}
\ \vspace{1mm}

For a subset $I$ of $\{0,1,\ldots,\lfloor(m+s-1)/2\rfloor\}$, we define
\begin{equation}\label{eqn_UIdef}
M(I)=\bigoplus_{i\in I}M(i)\ \text{ and }\ U(I)=\prod_{i\in I}U(i).
\end{equation}
Write $I\setminus\{0\}=\{i_1,\dots,i_k\}$ with $k\geqslant0$. Then define
\[
d(I)=
\begin{cases}
\gcd(2i_1-1,\ldots,2i_k-1,m)&\textup{if $s=2$},\\
\gcd(i_1,\ldots,i_k)&\textup{if $s=1$ and $m/2\in I$},\\
\gcd(i_1,\ldots,i_k,m)&\textup{if $s=1$ and $m/2\notin I$}.
\end{cases}
\]

\begin{lemma}\label{prop_UiElinear}
Let $I=\{n_1,\ldots,n_k\}$ be a nonempty subset of $\{1,\ldots,\lfloor(m+s-1)/2\rfloor\}$, let $d=d(I)$, and for $b\in\bbF_{r^{d}}$ let $\tau_b$ be the linear transformation on $V$ such that $\tau_b((x,y))=(b^{r^{n_1}}x,by)$. Then $F:=\{\tau_b\mid b\in\bbF_{r^{d}}\}$ is isomorphic to $\bbF_{r^d}$, and each element of $U(I){:}S_0$ is $F$-linear. In particular, $U(I){:}S_0$ is contained in a field extension subgroup of $\GL(V)$ defined over $\bbF_{r^d}$.
\end{lemma}

\begin{proof}
We only prove the case $s=2$ here, as the case $s=1$ is similar. Suppose that $s=2$. It suffices to show that for each $b\in\bbF_{q^{2d}}^\times$, the element $\tau_b$ commutes with each element of $S_0$ and each element of $U(I)$. This is trivial for $S_0$, and next we show that $\tau_b$ commutes with $u_{\bfa}\in U(n_i)$ for each $i$.
Observe from the definition of $d=d(I)$ that
\begin{equation}\label{eqn_bpower}
b^{q^{2n_1}}=b^{q^{2n_i}}=b^{q^{2m-2n_i+2}}=b^{q^{m+1}}
\end{equation}
for $b\in\bbF_{r^d}$ and $1\leqslant i\leqslant k$.

First assume that $n_i\neq(m+1)/2$, and take $u_\bfa\in U(n_i)$ such that
\[
\hbar_{\bfa}(X)=aX^{q^{2n_i}}-a^{q^{2m-2n_i+1}}X^{q^{2m-2n_i+2}}
\]
for some $a\in\bbF_{q^{2m}}$. Then for $(x,y)\in V$ we have
\begin{align*}
(\tau_b\circ u_\bfa)((x,y))&=(b^{q^{2n_1}}x+b^{q^{2n_1}}ay^{q^{2n_i}}-b^{q^{2n_1}}a^{q^{2m-2n_i+1}}y^{q^{2m-2n_i+2}},by)\\
&=(b^{q^{2n_1}}x+ab^{q^{2n_i}}y^{q^{2n_i}}-b^{q^{2m-2n_i+2}}a^{q^{2m-2n_i+1}}y^{q^{2m-2n_i+2}},by)=(u_\bfa\circ\tau_b)((x,y)).
\end{align*}
(Here we used \eqref{eqn_bpower} in the second equality).

Next assume that $m$ is odd and $n_i=(m+1)/2$. Take $u_\bfa\in U((m+1)/2)$ such that $\hbar_{\bfa}(X)=aX^{q^{m+1}}$ for some $a\in\bbF_{q^{2m}}$ with $a+a^{q^m}=0$. For $(x,y)\in V$, we deduce from \eqref{eqn_bpower} that
\begin{align*}
(\tau_b\circ u_\bfa)((x,y))&=(b^{q^{2n_1}}x+b^{q^{2n_1}}ay^{q^{m+1}},by)\\
&=(b^{q^{2n_1}}x+ab^{q^{m+1}}y^{q^{m+1}},by)=(u_\bfa\circ\tau_b)((x,y)).
\end{align*}
The proof is now complete.
\end{proof}

The following proposition gives necessary conditions for the solvable factor $H$ in Hypothesis~\ref{hypo-1}\,\eqref{hypo-1ii}--\eqref{hypo-1iv}.

\begin{proposition}\label{thm:dValue}
Let $L$, $G$, $B$, $H$, $P$ and $S$ be as in Proposition~$\ref{Xia:Unitary01}$,~$\ref{Xia:Omega01}$ or~$\ref{Xia:Symplectic01}$, and let $P=U(I)$ with $I\subseteq\{0,1,\ldots,\lfloor(m+s-1)/2\rfloor\}$ such that $0\in I$ only if $L=\Sp_{2m}(q)$. If $G=HB$, then one of the following holds:
\begin{enumerate}[{\rm(a)}]
\item\label{thm:dValuea} $L=\SU_{2m}(q)$, and $d(I)=1$;
\item\label{thm:dValueb} $L=\Omega_{2m}^+(q)$, and either $d(I)=1$, or $d(I)=2$ with $q\in\{2,4\}$;
\item\label{thm:dValuec} $L=\Sp_{2m}(q)$, and either $0\in I$, or $0\notin I$ and $d(I)\in\{1,2\}$ with $q\in\{2,4\}$.
\end{enumerate}
\end{proposition}

\begin{proof}
First assume that $\epsilon=-1/2$, that is, $G_0=\GU_{2m}(q)$. Suppose that $d(I)=d>1$. Then Lemma~\ref{prop_UiElinear} implies that $H$ is contained in a field extension subgroup $A$ of type $\GU_{2m/d}(q^d)$ in $G$. From~\cite[Theorem~A]{LPS1990} we see that $G\neq AB$ and thus $G\neq HB$, proving part~\eqref{thm:dValuea}.

In the case $\epsilon=-1$, the same argument as above leads to part~\eqref{thm:dValueb}.

Now assume $\epsilon=0$. Then $G_0=\Sp_{2m}(q)$, and there is a maximal subgroup $A$ of $G$ such that $A\cap G_0=\mathrm{O}_{2m}^+(q)$ and $\bfO_2(\Pa_m[G])=U(0)\times\bfO_2(\Pa_m[A])$. If $0\notin I$, then $U(I)\leqslant\Pa_m[A]$ and so $H<A$. In this case, we deduce from $G=HB$ that $G=AB$, which forces $q\in\{2,4\}$ by~\cite[Theorem~A]{LPS1990}. Moreover, this implies $A=H(A\cap B)$, which leads to $d(I)\in\{1,2\}$ by the conclusion in part~\eqref{thm:dValueb}. Thus part~\eqref{thm:dValuec} holds.
\end{proof}

\subsection{Orbits of $U(I)$ on $[G:B]$}\label{subsec_UIorb}
\ \vspace{1mm}

Let $I$ be a subset of $\{0,1,\ldots,\lfloor(m+s-1)/2\rfloor\}$. In order to apply Propositions~\ref{Xia:Unitary01},~\ref{Xia:Omega01} and~\ref{Xia:Symplectic01}, we need to calculate the orbit length of $U(I)$ on $[G:B]$. Define a mapping $\kappa$ from $V=\bbF_{r^m}\times\bbF_{r^m}$ to $\bbF_q$ by letting
\[
\nu(x,y)=\Tr_{r^m/q}(xy^{q^{s-1}})\,\text{ for }(x,y)\in V.
\]
This is exactly $\kappa_\epsilon$ if $\epsilon=-1$, and is the norm on $(V,\kappa_\epsilon)$ if $\epsilon=-1/2$.
Hence we may identify the coset space $[G:B]$ with the set
\begin{equation*}%\label{eqn_N1Plus}
\Lambda_\epsilon=\{\langle(x,y)\rangle\mid (x,y)\in V \textup{ with }\nu(x,y)=1\}\,\text{ if }\epsilon\in\{-1/2,-1\}.
\end{equation*}
%We observe the following information regarding the set $\calN_1^+$.
%\begin{enumerate}[{\rm(i)}]
%\item If $V$ is equipped with the Hermitian form $\kappa_{-1/2}$, then $\calN_1^+$ consists of all non-singular points and it has size $q^{2m-1}\frac{q^{2m}-1}{q+1}$.
%\item If $V$ is equipped with the quadratic form $\kappa_{-1}$, then $\calN_1^+$ is one isometry class of non-singular points and it has size $q^{m-1}{q^m-1\over(2,q-1)}$.
%\item If $G$ with socle $\PSU_{2m}(q)$ or $\POm_{2m}^+(q)$ acts transitively on $\calN_1^+$ with stabilizer $B$, then a subgroup $H$ is such that $G=HK=HB$ if and only if $H$ is transitive on $\calN_1^+$.
%\end{enumerate}
For $\epsilon=0$, the vector space $V$ is equipped with the alternating form $\kappa_0$. In this case, we let
\[
\mbox{$\Lambda_0=\{$elliptic quadrics on $V$ whose associated bilinear form is $\kappa_0\}$.}
\]
%The set $\mathcal{E}$ has size $\frac{1}{2}q^m(q^m-1)$.
The symplectic group $\Sp_{2m}(q)$ is transitive on $\Lambda_0$ with stabilizer $\rmO_{2m}^-(q)$.
Thus $[G:B]$ can be identified with $\Lambda_\epsilon$ for all $\epsilon\in\{-1,-1/2,0\}$.
In this subsection we determine the orbits of $U(I)$ on $\Lambda_\epsilon$ and thus on $[G:B]$ (see Proposition~\ref{thm_UIorbit}).
%\[\Omega=\begin{cases}
%\calN_1^+, &\text{if $V$ is equipped with either $\kappa_{-1/2}$ or $\kappa_{-1}$},\\
%\mathcal{E},&\text{if $V$ is equipped with $\kappa_0$}.
%\end{cases}\]

Define for each $a,b\in\bbF_{q^m}$ a mapping $\kappa_{a,b}$ as follows.
\begin{equation*}%\label{eqn_Kappaab}
\kappa_{a,b}:V\rightarrow\bbF_q,\ \ (x,y)\mapsto\Tr_{q^m/q}(ax^2+xy+by^2).
\end{equation*}
The subsequent lemma explicitly describes $\Lambda_0$ for us to work with in this subsection.

\begin{lemma}\label{lem_EllipKcd}
For $\epsilon=0$, we have ${\Lambda_\epsilon}=\{\kappa_{a,b}\mid\Tr_{q^m/2}(ab)=1\}$.
\end{lemma}

\begin{proof}
Recall that $r=q$ is even, as $\epsilon=0$. Take an element $\kappa\in{\Lambda_0}$, and set $\kappa'=\kappa+\nu$. The associated bilinear form of $\kappa'$ is $\kappa_0+\kappa_0=0$, and so $\kappa'$ is additive. Moreover, since $\kappa'$ is a quadratic form, we have $\kappa'(cx,cy)=c^2\kappa'(x,y)$ for all $c\in\bbF_q$ and $(x,y)\in V$. It follows that $(x,y)\mapsto(\kappa'(x,y))^{q/2}$ is an $\bbF_q$-linear mapping from $\bbF_{q^m}$ to $\bbF_q$, and so $\kappa'(x,y)=\Tr_{q^m/q}(ax^2+by^2)$ for some $a,b\in\bbF_{q^m}$. Accordingly, $\kappa=\kappa'+\nu=\kappa_{a,b}$.

Next, we deduce the necessary and sufficient condition $\Tr_{q^m/2}(ab)=1$ for $\kappa_{a,b}$ to be an elliptic form. The vector space $V$ can be viewed as a $2$-dimensional vector space over $\bbF_{q^m}$, and we denote it by $V_\sharp$. Consider the quadratic form $\kappa_\sharp$ on $V_\sharp$ defined by $\kappa_\sharp(x,y)=ax^2+xy+by^2$ for $(x,y)\in V_\sharp$. Thus $\kappa_{a,b}=\Tr_{q^m/q}\circ \kappa_\sharp$. By~\cite[Table~4.3.A]{KL1990}, the form $\kappa_{a,b}$ is elliptic if and only if $\kappa_\sharp$ is elliptic. The latter holds exactly when the polynomial $aX^2+XY+bY^2$ is irreducible over $\bbF_{q^m}$. Moreover, by \cite[Corollary~3.79]{LN97}, $aX^2+X+b$ is irreducible over $\bbF_{q^m}$ if and only if $\Tr_{q^m/2}(ab)=1$. This completes the proof.
\end{proof}

We will need the following elementary arithmetic result in the proof of Lemma~\ref{lem_key}.

\begin{lemma}\label{lem_gcd}
Let $a\geqslant2$ be an integer, and let $n_1,\ldots,n_k$ be positive integers with $k\geqslant1$. Then
\[
\gcd(a^{n_1}-1,a^{n_2}-1,\ldots,a^{n_k}-1)=a^{\gcd(n_1,\ldots,n_k)}-1.
\]
\end{lemma}

\begin{proof}
It suffices to prove the lemma for $k=2$, since the general case follows by induction on $k$. Suppose that $k=2$, and write $d=\gcd(n_1,n_2)$ and $D=\gcd(a^{n_1}-1,a^{n_2}-1)$. Then $d=bn_1+cn_2$ for some integers $b$ and $c$, and we have $a^{n_1}\equiv a^{n_2}\equiv1\pmod{D}$. It is clear that $a^d-1$ divides both $a^{n_1}-1$ and $a^{n_2}-1$, and hence divides $D$. Conversely, it follows from
\[
a^d-1=a^{bn_1+cn_2}-1=(a^{n_1})^b(a^{n_2})^c-1\equiv1^b\cdot1^c-1=0\pmod{D}
\]
that $D$ divides $a^d-1$. This completes the proof.
\end{proof}

In the proof of the next lemma, we adopt the Iverson bracket notation $[\![P]\!]$ for a property $P$, which takes value $1$ or $0$ according as $P$ holds or not.

\begin{lemma}\label{lem_key}
Let $n_1,\ldots,n_k\in\{1,\ldots,m-1\}$ with $k\geqslant1$, and let $x\in\bbF_{r^m}^\times$.
\begin{enumerate}[{\rm(a)}]
\item\label{lem_keya} If $\mathcal{N}$ is the set of tuples $(a_1,\ldots,a_k)\in\bbF_{r^{m}}^k$ such that
\begin{equation*}
\sum_{i=1}^k\left(a_ix^{r^{n_i}}-a_i^{r^{m-n_i}q^{s-1}}x^{r^{m-n_i}q^{2s-2}}\right)=0,
\end{equation*}
then $|\mathcal{N}|=r^{m(k-1)}q^{\gcd(sn_1-s+1,\ldots,sn_k-s+1,m)}$.
\item\label{lem_keyb} If $m+s$ is odd and $\mathcal{N}$ is the set of $(a_0,a_1,\ldots,a_k)\in\bbF_{r^m}^{k+1}$ such that $\Tr_{r^m/r^{m/2}}(a_0)=0$ and
\begin{equation}\label{eqn_tt1}
a_0x^{r^{m/2}q^{s-1}}+\sum_{i=1}^k\left(a_ix^{r^{n_i}}-a_i^{r^{m-n_i}q^{s-1}}x^{r^{m-n_i}q^{2s-2}}\right)=0,
\end{equation}
then $|\mathcal{N}|=r^{m(2k-1)/2}q^{\gcd(sn_1-s+1,\ldots,sn_k-s+1,sm/2)}$.
\end{enumerate}
\end{lemma}

\begin{proof}
We only give details for part~\eqref{lem_keyb} here, since the proof for part~\eqref{lem_keya} is similar. Suppose that $m+s$ is odd and $\mathcal{N}$ is as defined in~\eqref{lem_keyb}. Let $d=\gcd(sn_1-s+1,\ldots,sn_k-s+1,sm/2)$. Recall that, for a positive integer $i$, the canonical additive character $\psi_{p^i}$ of $\bbF_{p^i}$ is defined by $\psi_{p^i}(a)=\zeta_p^{\Tr_{p^i/p}(a)}$ for $a\in\bbF_{p^i}$, where $\zeta_p=\exp(2\pi\sqrt{-1}/p)$. For $a\in\bbF_{p^i}$, according to~\cite[(5.9)]{LN97},
\[
\sum_{b\in\bbF_{p^i}}\psi_{p^i}(ab)=p^i\cdot[\![a=0]\!].
\]
Write $\psi=\psi_{r^m}$ for brevity. Then for each $a\in\bbF_{r^m}$,
\begin{equation}\label{eqn_tt0}
\sum_{b\in\bbF_{r^{m/2}}}\psi(ab)=\sum_{b\in\bbF_{r^{m/2}}}\psi_{r^{m/2}}(\Tr_{r^m/r^{m/2}}(a)\cdot b)=r^{m/2}\cdot[\![ \Tr_{r^m/r^{m/2}}(a)=0]\!].
\end{equation}
Take a tuple $(a_0,a_1,\ldots,a_k)$ in the set $\mathcal{N}$. We see from~\eqref{eqn_tt1} that $a_0$ is uniquely determined by $(a_1,\ldots,a_k)$, as $a_0=x^{-q^{s-1}(r^{m/2}+1)}f(a_1,\ldots,a_k)$ with
\begin{equation*}
f(a_1,\ldots,a_k)= \sum_{i=1}^k\left(-a_ix^{r^{n_i}+q^{s-1}}+a_i^{r^{m-n_i}q^{s-1}}x^{r^{m-n_i}q^{2s-2}+q^{s-1}}\right).
\end{equation*}
Since $x^{-q^{s-1}(r^{m/2}+1)}$ belongs to $\bbF_{r^{m/2}}^\times$, it follows that $\Tr_{r^m/r^{m/2}}$ takes value $0$ at $a_0$ if and only it takes $0$ at $f(a_1,\ldots,a_k)$. Hence $\mathcal{N}$ has the same size as the set of $(a_1,\ldots,a_k)\in\bbF_{r^m}^k$ such that $\Tr_{r^m/r^{m/2}}(f(a_1,\ldots,a_k))=0$. Then according to \eqref{eqn_tt0}, we have
\begin{align*}
r^{m/2}\cdot|\mathcal{N}|=\sum_{b\in\bbF_{r^{m/2}}}\;\sum_{a_1,\ldots,a_k\in\bbF_{r^m}}\psi\big((f(a_1,\ldots,a_k)\cdot b\big)
=r^{mk}+\sum_{b\in\bbF_{r^{m/2}}^\times}\prod_{i=1}^kS_i(b),
\end{align*}
where $S_i(b)=\sum_{a_i\in\bbF_{r^m}}\psi(-a_ix^{r^{n_i}+q^{s-1}}b+a_i^{r^{m-n_i}q^{s-1}}x^{r^{m-n_i}q^{2s-2}+q^{s-1}}b)$ for $1\leqslant i\leqslant k$ and $b\in\bbF_{r^{m/2}}^\times$.
By \cite[Theorem 2.23(v)]{LN97} one has $\psi(y)=\psi(y^p)$ for all $y\in\bbF_{r^m}$. Therefore, since
\[
\big(a_i^{r^{m-n_i}q^{s-1}}x^{r^{m-n_i}q^{2s-2}+q^{s-1}}b\big)^{r^{n_i-1}q}=a_i^{r^{m-1}q^s}x^{q^{s-1}+r^{n_i}}b^{r^{n_i-1}q},
\]
it follows that
\begin{align*}
S_i(b)&=\sum_{a_i\in\bbF_{r^m}}\psi\left(-a_ix^{r^{n_i}+q^{s-1}}b\right)
+\sum_{a_i\in\bbF_{r^m}}\psi\left(a_i^{r^{m-1}q^s}x^{r^{n_i}+q^{s-1}}b^{r^{n_i-1}q}\right)\\
&=\sum_{a_i\in\bbF_{r^m}}\psi\left(-a_ix^{r^{n_i}+q^{s-1}}b\right)
+\sum_{a_i\in\bbF_{r^m}}\psi\left(a_ix^{r^{n_i}+q^{s-1}}b^{r^{n_i-1}q}\right)\\
&=\sum_{a_i\in\bbF_{r^m}}\psi\left(a_ix^{r^{n_i}+q^{s-1}}(-b+b^{r^{n_i-1}q})\right)=r^m\cdot[\![b^{r^{n_i-1}q-1}=1]\!].
\end{align*}
This leads to $\prod_{i=1}^kS_i(b)=r^{mk}\cdot[\![b\in C]\!]$, where $C=\{c\in\bbF_{r^{m/2}}^\times\mid c^{D}=1\}$ with
\[
D=\gcd(r^{n_1-1}q-1,\ldots,r^{n_k-1}q-1,r^{m/2}-1).
\]
Note that $D=q^d-1$ by Lemma~\ref{lem_gcd}. Hence $C=\bbF_{q^d}^\times$, and thus
\[
r^{m/2}\cdot |\mathcal{N}|=r^{mk}+r^{mk}\sum_{b\in\bbF_{r^{m/2}}^\times}[\![b\in C]\!]=r^{mk}+r^{mk}|C|=r^{mk}+r^{mk}(q^d-1)=r^{mk}q^d,
\]
which gives $|\mathcal{N}|=r^{m(2k-1)/2}q^{\gcd(sn_1-s+1,\ldots,sn_k-s+1,sm/2)}$, as required.
\end{proof}

%\begin{notation}\label{nota_dI}
%For a nonempty subset $I=\{n_1,\ldots,n_k\}$ of $\{1,\ldots,\lfloor\frac{m+s-1}{2}\rfloor\}$, we define $d(I)$ as follows.
%If $s=1$, then set
%\[
%  d(I)=\begin{cases}(n_1,\ldots,n_k,m/2),&\quad \textup{if $m$ is even and $\max(I)=\frac{m}{2}$},\\
%                    (n_1,\ldots,n_k,m),&\quad \textup{otherwise},
%       \end{cases}
%\]
%and if $s=2$, then set
%\[\mbox{$d(I)=(2n_1-1,\ldots,2n_k-1,m)$}.\]
%\end{notation}

%\begin{definition}\label{irredu}
%We call $I$ \emph{irredundant} relative to $d$
%if $d(I)=d$, but $d(J)\neq d$ for any proper subset $J$ of $I$.
%%if $I$ is minimal satisfying
%%$d(I)=d$.
%%but $\gcd(m_{i_1},\dots,m_{i_t},m)\not=d$
%%for any proper subset $\{m_{i_1},\dots,m_{i_t}\}\subset I$.
%\end{definition}

Recall the notation defined in~\eqref{eqn_UIdef}.

\begin{lemma}\label{prop_MI}
Let $I=\{n_1,\ldots,n_k\}$ be a nonempty subset of $\{1,\ldots,\lfloor(m+s-1)/2\rfloor\}$, let $d=d(I)$, let $\hbar(X)\in M(I)$, and let $x\in\bbF_{r^m}$. Then the following statements hold:
\begin{enumerate}[{\rm(a)}]
\item\label{prop_MIa} $\Tr_{r^m/q^d}\big(\hbar(x)x^{q^{s-1}}\big)=0$;
\item\label{prop_MIb} if $\epsilon=0$, then $\Tr_{r^m/q^d}\big(x\big(\hbar(x^{q^m/2})\big){}^2\big)=0$.
\end{enumerate}
\end{lemma}

\begin{proof}
We only prove statement~\eqref{prop_MIa}, as the proof of statement~\eqref{prop_MIb} is essentially the same. Write $\Tr=\Tr_{r^m/q^d}$ for brevity. To prove~\eqref{prop_MIa}, it suffices to show that $\Tr(h(x)x^{q^{s-1}})=0$ for all $h$ in each component $M(n_i)$ of $M(I)$.

First assume that $n_i\neq(m+s-1)/2$. Then $h(x)=ax^{r^{n_i}}-a^{r^{m-n_i}q^{s-1}}x^{r^{m-n_i}q^{2s-2}}$ for some $a\in\bbF_{r^m}$.
Sine $d=d(I)$ divides both $sn_i-s+1$ and $m$, it divides $sm-sn_i+s-1$. Then as $r^{m-n_i}q^{s-1}=q^{sm-sn_i+s-1}$ and $\Tr=\Tr_{r^m/q^d}$, it follows that
\[
\Tr\big(ax^{r^{n_i}+q^{s-1}}\big)=\Tr\big(a^{r^{m-n_i}q^{s-1}}x^{(r^{n_i}+q^{s-1})r^{m-n_i}q^{s-1}}\big)
=\Tr\big(a^{r^{m-n_i}q^{s-1}}x^{q^{s-1}+r^{m-n_i}q^{2s-2}}\big).
\]
This shows $\Tr(h(x)x^{q^{s-1}})=0$, as desired.

Next assume that $n_i=(m+s-1)/2$. In this case, $m+s$ is odd, and $h(x)=ax^{r^{m/2}q^{s-1}}$ for some $a\in \bbF_{r^m}$ with $a+a^{r^{m/2}}=0$. Therefore, since $x^{r^{m/2}+1}\in\bbF_{r^{m/2}}$, we deduce that
\[
\Tr\big(h(x)x^{q^{s-1}}\big)=\Tr\big(ax^{(r^{m/2}+1)q^{s-1}}\big)=\Tr_{r^{m/2}/q^d}\big((a+a^{r^{m/2}})x^{(r^{m/2}+1)q^{s-1}}\big)=0.
\]
This completes the proof.
\end{proof}

Let $I$ be a nonempty subset of $\{0,1,\ldots,\lfloor(m+s-1)/2\rfloor\}$. We are now ready to determine the $U(I)$-orbits on $\Lambda_\epsilon$. To state the result, write $d=d(I)$ and define the following subsets of $\Lambda_\epsilon$:
\begin{itemize}
\item if $\epsilon\in\{-1,-1/2\}$, then for each $y\in\bbF_{r^m}^\times$ and each $c\in\bbF_{q^d}$ with $\Tr_{q^d/q}(c)=1$, define
\[
\calN_{d,y,c}=\{\la(x,y)\ra\mid \Tr_{r^m/q^d}(xy^{q^{s-1}})=c\};
\]
\item if $\epsilon=0$, then for each $x\in\bbF_{q^m}^\times$ and each $c\in\bbF_{q^d}$ with $\Tr_{q^d/2}(c)=1$, define
\[
\calE_{d,a,c}=\{\kappa_{a,b}\mid\Tr_{q^m/q^d}(ab)=c\}\ \text{ and }\ \calE_{d,a}=\{\kappa_{a,b}\mid\Tr_{q^m/2}(ab)=1\}.
\]
\end{itemize}
Note that $\calE_{d,a}$ is the disjoint union of $\calE_{d,a,c}$'s for $c\in\bbF_{q^d}$ with $\Tr_{q^d/2}(c)=1$.

%\begin{definition}
%Let $I$ be a nonempty subset of $\{1,\ldots,\lfloor(m+s-1)/2\rfloor\}$ and let $d=d(I)$.
%\begin{enumerate}[{\rm(a)}]
%\item
%Suppose that $\epsilon\in\{-1,-1/2\}$.
%For each $y\in\bbF_{r^m}^\times$ and each $c\in\bbF_{q^d}$ with $\Tr_{q^d/q}(c)=1$, define the following subset of $\mc{N}_1^+$:
%\[X_{y,c}=\{\l(x,y)\r\mid \Tr_{r^m/q^d}(xy^{r/q})=c\}.\]
%\item
%Suppose that $\epsilon=0$. For each $x\in\bbF_{q^m}^\times$ and each $c\in\bbF_{q^d}$ with $\Tr_{q^d/2}(c)=1$, define the following subsets of ${\Lambda_0}$:
%\begin{align*}
%  {\Lambda_0}_{x,c}&=\{\kappa_{x,y}\mid\Tr_{q^m/q^d}(xy)=c\},\\
%  {\Lambda_0}_{x}&=\{\kappa_{x,y}\mid\Tr_{q^m/2}(xy)=1\}.
%\end{align*}
%\end{enumerate}
%\end{definition}

\begin{proposition}\label{thm_UIorbit}
Let $I$ be a nonempty subset of $\{0,1,\ldots,\lfloor(m+s-1)/2\rfloor\}$ such that $0\in I$ only if $\epsilon=0$, and let $d=d(I)$. Then the following statements hold:
\begin{enumerate}[{\rm(a)}]
\item\label{thm_UIorbita}
If $\epsilon\in\{-1,-1/2\}$, then for each $\la(x,y)\ra\in\Lambda_\epsilon$, the $U(I)$-orbit on $\Lambda_\epsilon$ containing $\la(x,y)\ra$ is $\calN_{d,y,c}$ with $c=\Tr_{r^m/q^d}(xy^{q^{s-1}})$, and it has length $r^m/q^d$;
\item\label{thm_UIorbitb}
If $\epsilon=0\notin I$, then for each $\kappa_{a,b}\in{\Lambda_\epsilon}$, the $U(I)$-orbit on $\Lambda_\epsilon$ containing $\kappa_{a,b}$ is $\calE_{d,a,c}$ with $c=\Tr_{q^m/q^d}(ab)$, and it has length $q^m/q^d$;
\item\label{thm_UIorbitc}
If $\epsilon=0\in I$, then for each $\kappa_{a,b}\in{\Lambda_\epsilon}$, the $U(I)$-orbit on $\Lambda_\epsilon$ containing $\kappa_{a,b}$ is $\calE_{d,a}$, and it has length $q^m/2$.
\end{enumerate}
In particular, every orbit of $U(I)$ on $[G:B]$ has length $r^m/q^d$ if $0\notin I$, and length $q^m/2$ if $0\in I$.
\end{proposition}

\begin{proof}
First assume that $\epsilon\in\{-1,-1/2\}$. Take an arbitrary $u_\bfa\in U(I)$, where $\bfa\in\bbF_{r^m}^m$. Then $(x,y)^{u_\bfa}=(x+\hbar_\bfa(y),y)$, and we derive from Lemma~\ref{prop_MI}\,\eqref{prop_MIa} that $\la(x+\hbar_\bfa(y),y)\ra\in\calN_{d,y,c}$. Moreover, $u_\hbar$ fixes $\la(x,y)\ra$ if and only if $\hbar_\bfa(y)=0$.
Let $k=|I\setminus\{(m+s-1)/2\}|$. By Lemmas~\ref{lem_Mprop} and~\ref{lem_key}, the set $\{u_\bfa\mid \hbar_\bfa(y)=0\}$ has size $r^{m(k-1)}q^d$ if $(m+s-1)/2\notin I$, and size $r^{m(2k-1)/2}q^d$ if $(m+s-1)/2\in I$. Thus the $U(I)$-orbit on $\Lambda_\epsilon$ containing $\la(x,y)\ra$ has length $r^m/q^d$. Since $|\calN_{d,y,c}|=r^m/q^d$, it follows that $\calN_{d,y,c}$ is the $U(I)$-orbit containing $\la(x,y)\ra$. This proves statement~\eqref{thm_UIorbita}.

From now on, assume $\epsilon=0$. Then $r=q$ is even, and for $u_\bfa\in U(I)$ with $\bfa\in\bbF_{r^m}^m$, we have
\[
\kappa_{a,b}^{u_\bfa}(x,y)=\kappa_{a,b}((x,y)^{u_\bfa})=\kappa_{a,b}(x+\hbar_\bfa(y),y).
\]
Write $\bfa=(a_0,a_1,\ldots,a_{m-1})\in\bbF_{q^m}^m$ and $\bfa\cdot\bfa=(a_0^2,a_1^2,\ldots,a_{m-1}^2)$.
From Lemma~\ref{lem_Mprop} we see that both $\Tr_{q^{m}/q}((\hbar_{\bfa}(y)-a_0y)y)$ and $\Tr_{q^{m}/q}(a\hbar_{\bfa\cdot\bfa}(y^2)+\hbar_{\bfa\cdot\bfa}(a)y^2)$ are equal to $0$.  Hence
\begin{align*}
\kappa_{a,b}^{u_\bfa}(x,y)&=\Tr_{q^{m}/q}\left(a(x+\hbar_{\bfa}(y))^2+(x+\hbar_{\bfa}(y))y+by^2\right)\\
&=\Tr_{q^{m}/q}\left(ax^2+xy+by^2\right)+\Tr_{q^{m}/q}(\hbar_{\bfa}(y)y)+\Tr_{q^{m}/q}\left(a\hbar_{\bfa\cdot\bfa}(y^2)\right)\\
&=\Tr_{q^{m}/q}\left(ax^2+xy+by^2\right)+\Tr_{q^{m}/q}(a_0y^2)+\Tr_{q^{m}/q}\left(\hbar_{\bfa\cdot\bfa}(a)y^2\right)\\
&=\Tr_{q^{m}/q}\left(ax^2+xy+\left(b+a_0+\big(\hbar_{\bfa}(a^{q^m/2})\big)^2\right)y^2\right).
\end{align*}
In other words, $u_\bfa$ maps $\kappa_{a,b}$ to $\kappa_{a,b'}$ with $b'=b+a_0+(\hbar_{\bfa}(a^{q^m/2}))^2$. Consequently, $u_\bfa$
stabilizes $\kappa_{a,b}$ if and only if $a_0+(\hbar_{\bfa}(a^{q^m/2}))^2=0$.

If $0\not\in I$, then $a_0=0$, and Lemma~\ref{prop_MI}\,\eqref{prop_MIb} implies that the $U(I)$-orbit on $\Lambda_\epsilon$ containing $\kappa_{a,b}$ is contained in $\calE_{d,a,c}$, where $c=\Tr_{q^m/q^d}(ab)$. In this case, noting that $|\calE_{d,a,c}|=q^m/q^d$ and that $u_\bfa$ stabilizes $\kappa_{a,b}$ if and only if $\hbar_{\bfa}(a^{q^m/2})=0$, we conclude by Lemma~\ref{lem_key} that the $U(I)$-orbit containing $\kappa_{a,b}$ is $\calE_{d,a,c}$. Thus statement~\eqref{thm_UIorbitb} holds. Next assume that $0\in I$. The set $\{b+w+aw^2\mid w\in\bbF_{q^m}\}$ has size $q^m/2$ and equals $\{z\in\bbF_{q^m}\mid\Tr_{q^m/2}(az)=1\}$, as $\Tr_{q^m/2}(ab)=1$. Let $\hslash(X)=a_0X+\hbar(X)$, and recall that $u_\bfa$ maps $\kappa_{a,b}$ to $\kappa_{a,b'}$ with
\[
b'=b+a_0+(\hbar_{\bfa}(a^{q^m/2}))^2=b+a_0+a_0^2a+(\hslash_\bfa(a^{q^m/2}))^2.
\]
Since Lemma~\ref{prop_MI}\,\eqref{prop_MIb} implies $\Tr_{q^m/2}(a(\hslash_\bfa(a^{q^m/2}))^2)=0$, we obtain $\Tr(ab')=\Tr(a(b+a_0+aa_0^2))$. Hence the $U(I)$-orbit containing $\kappa_{a,b}$ is contained in $\calE_{d,a}$, and the $U(0)$-orbit containing $\kappa_{a,b}$ is equal to $\calE_{d,a}$. Therefore, the $U(I)$-orbit containing $\kappa_{a,b}$ is $\calE_{d,a}$, which proves statement~\eqref{thm_UIorbitc}.
\end{proof}

\section{Complete the classification}\label{sec:Xia-6}

With the preparation from the previous two sections, we are now able to prove the results in Subsections~\ref{sec:Xia-9} and~\ref{sec:Xia-10}.

By Lemma~\ref{lem:Xia-7} one may write $H=P{:}S$ with $P=U(I)$ (recall the definition of $U(I)$ in~\eqref{eqn_UIdef}) and $S\leqslant T$, where $I\subseteq\{0,1,\ldots,\lfloor(m+s-1)/2\rfloor\}$ with $0\in I$ only in the symplectic case.
Then Theorem~\ref{thm:Xia-2} is an immediate consequence of Propositions~\ref{Xia:Unitary01},~\ref{thm:dValue} and~\ref{thm_UIorbit}. Next we prove Theorems~\ref{thm:Xia-3} and~\ref{thm:Xia-4}.

\begin{proof}[Proof of Theorem~$\ref{thm:Xia-3}$]
By Proposition~\ref{thm:dValue}, a necessary condition for $G=HB$ is that either $d(I)$=1, or $d(I)=2$ with $q\in\{2,4\}$. For $d(I)=1$, Proposition~\ref{thm_UIorbit} implies that $|P|/|P\cap B|=q^{m-1}$, and so by Proposition~\ref{Xia:Omega01}, $G=HB$ if and only if $S$ is transitive on $W_{(\gcd(2,q-1))}$.
Next assume that $d(I)=2$ and $q\in\{2,4\}$. Then $m$ is even, and by Lemma~\ref{prop_UiElinear}, $H$ is contained in a maximal subgroup $A$ of $G$ such that $A\cap L=\Omega_m^+(q^2).2^2$.

\textsf{Case~1}: $q=2$. In this case, as~\cite[3.6.1c]{LPS1990} shows, $G=AB$ with $\N_1[A]=(A\cap B)\la\psi\ra$, where $\psi$ is the field automorphism of $\Omega_m^+(q^2)$. Hence $G=HB$ if and only if $A=H(A\cap B)$.
Note from Proposition~\ref{thm_UIorbit} that $|P|/|P\cap B|=q^m/q^2=(q^2)^{\frac{m}{2}-1}$. Then by Proposition~\ref{Xia:Omega01}, $A=H\N_1[A]$ if and only if $S$ is transitive on $W_{(1)}$. Since $\N_1[A]=(A\cap B)\la\psi\ra$, it follows that $A=H(A\cap B)$ if and only if $S$ is transitive on $W_{(1)}$ and $\la\psi\ra\leqslant S$. Thus $G=HB$ if and only if $S$ is transitive on $W_{(1)}$ with $|T|/|SS_0|$ odd.

\textsf{Case~2}: $q=4$. In this case,~\cite[3.6.1c]{LPS1990} shows that $G=AB$ if and only if $G\geqslant L.2$ and $G\neq\mathrm{O}_{2m}^+(4)$. Hence a necessary condition for $G=HB$ is $G\geqslant L.2$ with $G\neq\mathrm{O}_{2m}^+(4)$. Now suppose that this condition holds. Then $G=AB$, and by the same argument as in Case~1 we conclude that $G=HB$ if and only if $S$ is transitive on $W_{(1)}$ with $|T|/|SS_0|$ odd.
\end{proof}

\begin{proof}[Proof of Theorem~$\ref{thm:Xia-4}$]
Let $A$ be a maximal subgroup of $G$ such that $A\cap L=\mathrm{O}_{2m}^+(q)$ and $\bfO_2(\Pa_m[G])=U(0)\times\bfO_2(\Pa_m[A])$, and let $A_0=\Omega_{2m}^+(q){:}\la\phi\ra=\Omega_{2m}^+(q){:}f$, where $\phi$ is as in~\eqref{eqn:Xia-18}.
By Proposition~\ref{thm:dValue}, a necessary condition for $G=HB$ is that either $0\in I$, or $0\notin I$ and $d(I)\in\{1,2\}$ with $q\in\{2,4\}$. For $0\in I$, Proposition~\ref{thm_UIorbit} implies that $|P|/|P\cap B|=q^m/2$, and so by Proposition~\ref{Xia:Symplectic01}, $G=HB$ if and only if $S$ is transitive on $W_{(1)}$.
Next assume that $0\notin I$ and $d(I)\in\{1,2\}$ with $q\in\{2,4\}$. Then $U(I)\leqslant\Pa_m[A]$ and $H<A$.

\textsf{Case~1}: $q=2$. It can be seen from~\cite[3.2.4e]{LPS1990} that $G=AB$ with $A\cap B=\Sp_{2m-2}(2)\times2$ and $A_0\cap B=\Sp_{2m-2}(2)$. Hence $G=HB$ if and only if $A=H(A\cap B)$. Moreover, as $H<\Pa_m[A]=\Pa_m[A_0]<A_0$, we have $A=H(A\cap B)$ if and only if $A_0=H(A_0\cap B)$. Since $A_0\cap B$ is maximal in $A_0$, it follows from Theorem~\ref{thm:Xia-3} that $G=HB$ if and only if $S$ is transitive on $W_{(1)}$ and either $d(I)=1$, or $d(I)=2$ with $|T|/|SS_0|$ odd.

\textsf{Case~2}: $q=4$. As~\cite[Theorem~A]{LPS1990} shows that $G=AB$ if and only if $G=\GaSp_{2m}(4)$, a necessary condition for $G=HB$ is $G=\GaSp_{2m}(4)$. Now suppose that this condition holds. Then $G=AB$ with $A\cap B=\Sp_{2m-2}(4)\times2<\mathrm{O}_{2m}^+(4)$ (see~\cite[3.2.4e]{LPS1990}), and so $G=HB$ if and only if $A=H(A\cap B)$. Since $H<\Pa_m[A]=\Pa_m[A_0]<A_0$, we have $A=H(A\cap B)$ if and only if $A_0=H(A_0\cap B)$. Then as $A_0\cap B=\Sp_{2m-2}(4)<A^{(\infty)}$, it follows that $G=HB$ if and only if $A^{(\infty)}=(H\cap A^{(\infty)})(A_0\cap B)$ with $A_0=HA^{(\infty)}$. If $d(I)=1$, then by Theorem~\ref{thm:Xia-3} we conclude that $G=HB$ if and only if $S$ is transitive on $W_{(1)}$ with $|T|/|SS_0|$ odd. If $d(I)=2$, then Theorem~\ref{thm:Xia-3} shows that there is no factorization $A^{(\infty)}=(H\cap A^{(\infty)})(A_0\cap B)$. This completes the proof.
\end{proof}

We now prove Corollary~\ref{cor_G0} and Proposition~\ref{prop_K}.

\begin{proof}[Proof of Corollary~$\ref{cor_G0}$]
Write $T=\la a\ra{:}\la\varphi\ra$ such that $\la a\ra=\GL_1(q^{sm})$, $a^\varphi=a^p$ and $\phi=\varphi^m$.

\textsf{Case~1}: $L=\SU_{2m}(q)$. In this case, we may write $G=L.\calO$ with $\calO\leqslant\la\delta\ra{:}\la\phi\ra$, and let $\calO\cap\la\delta\ra=\la\delta^\ell\ra$, where $\ell$ is a divisor of $q+1$. Then $\calO=\la\delta^\ell\ra\la\delta^d\phi^e\ra$ for some $d\in\{1,\dots,q+1\}$ and divisor $e$ of $2f$ such that
\begin{equation}\label{eqn:Xia-21}
(\delta^d\phi^e)^{2f/e}\in\la\delta^\ell\ra.
\end{equation}
Straightforward calculation shows that~\eqref{eqn:Xia-21} holds if and only if $\ell$ divides $d(q^2-1)/(p^e-1)$, that is, condition~\eqref{cor_G0a} in Corollary~\ref{cor_G0} holds. Since $G\cap T=(L\cap T).\calO=\la a^\ell\ra\la a^d\varphi^{me}\ra$,
the Foulser triple of the subgroup $G\cap T$ of $\la a\ra{:}\la\varphi\ra$ is $(\ell,j,me)$, where $j\in\{1,\dots,q^{2m}-1\}$ is a multiple of $(q^{2m}-1)_{\pi(\ell)'}$ such that $j-d$ is divisible by $\ell$.
By Theorem~\ref{thm:Xia-2}, there exists some $H$ as in Hypothesis~$\ref{hypo-1}$ such that $G=HB$ if and only if $G\cap T$ is transitive on $(\bbF_{q^2}^m)_{(q+1)}$. By Theorem~\ref{thm:Xia-1}, this holds if and only if the following two conditions hold with $i=(q^{2m}-1)/(q+1)$:
\begin{enumerate}[\rm(i)]
\item $\pi(\ell)\cap\pi(i)\subseteq\pi(f)\cap\pi(p^{me}-1)\setminus\pi(j)$;
\item if $\gcd(\ell,i)$ is even and $p^{me}\equiv3\pmod{4}$, then $\gcd(\ell,i)\equiv2\pmod{4}$.
\end{enumerate}
Since $j-d$ is divisible by $\ell$, it follows that $\gcd(\ell,i)$ is coprime to $j$ if and only if it is coprime to $d$. Hence condition~(i) is equivalent to condition~\eqref{cor_G0b} in Corollary~\ref{cor_G0}. Moreover, condition~(ii) is exactly condition~\eqref{cor_G0c} in Corollary~\ref{cor_G0}. Thus the conclusion of Corollary~\ref{cor_G0} is true for $L=\SU_{2m}(q)$.

\textsf{Case~2}: $L=\Omega_{2m}^+(q)$. First assume that $q^m\not\equiv1\pmod{4}$. Then Theorem~\ref{thm:Xia-1} shows that the subgroup $\la a^{\gcd(2,q-1)}\ra$ of $L\cap T$ is transitive on $(\bbF_q^m)_{(\gcd(2,q-1))}$. Hence it holds for any $G$ that $G\cap T$ is transitive on $(\bbF_q^m)_{(\gcd(2,q-1))}$, and so by Theorem~\ref{thm:Xia-3}, there exists some $H$ as in Hypothesis~$\ref{hypo-1}$ such that $G=HB$.
Now assume that $q^m\equiv1\pmod{4}$. If $G\leqslant L.\la\delta'',\phi\ra$, then $G\cap T=\la a^2\ra{:}\la\varphi\ra$ is not transitive on $(\bbF_q^m)_{(2)}$, and so by Theorem~\ref{thm:Xia-3}, there does not exist any $H$ as in Hypothesis~$\ref{hypo-1}$ such that $G=HB$.
Conversely, suppose that $G\nleqslant L.\la\delta'',\phi\ra$. Then $G=L.\calO$ such that $\calO$ contains some element of the form $\delta'(\delta'')^x\phi^y$, where $x$ and $y$ are integers. Accordingly, $G\cap T$ contains $\la a^2\ra\la a\phi^k\ra$ for some integer $k$. By Theorem~\ref{thm:trans}, $\la a^2\ra\la a\phi^k\ra$ is transitive on $\bbF_q^m$ and hence on $(\bbF_q^m)_{(2)}$. Thus, by Theorem~\ref{thm:Xia-3}, there exists some $H$ as in Hypothesis~$\ref{hypo-1}$ such that $G=HB$.
This proves Corollary~\ref{cor_G0} for $L=\Omega_{2m}^+(q)$.

\textsf{Case~3}: $L=\Sp_{2m}(q)$. In this case, any $G$ contains $\la a\ra$, which is transitive on $\bbF_q^m$. Then by Theorem~\ref{thm:Xia-4} there exists some $H$ as in Hypothesis~$\ref{hypo-1}$ such that $G=HB$. This completes the proof.
\end{proof}

\begin{proof}[Proof of Proposition~$\ref{prop_K}$]
Let $N=\Nor_{\ddot{G}}\big(B^{(\infty)}\big)/B^{(\infty)}$ and $J=\Nor_{RT}\big(B^{(\infty)}\big)B^{(\infty)}/B^{(\infty)}$. We first prove that $G=HK$ for some $H$ in Hypothesis~\ref{hypo-1} if and only if $K$ satisfies condition~\eqref{prop_Ka} and
\begin{equation}\label{eqn:Xia-22}
N=J\big((K\cap\ddot{G})/B^{(\infty)}\big).
\end{equation}
Then we show that $N$ and $J$ are as described in the table of the proposition. Write
\[
\ddot{H}=RT=R{:}T
\]
and $\ddot{B}=\Nor_{\ddot{G}}\big(B^{(\infty)}\big)$. Note that $\ddot{H}<\ddot{G}$ and that $\ddot{B}$ is a maximal subgroup of $\ddot{G}$ containing $K\cap\ddot{G}$. Let $\overline{\phantom{w}}\colon\ddot{B}\to\ddot{B}/B^{(\infty)}$ be the quotient modulo $B^{(\infty)}$.

Suppose that $G=HK$ for some $H$ as in Hypothesis~\ref{hypo-1}. Then since $H\leqslant\ddot{H}<\ddot{G}$, this implies $G\cap\ddot{G}=H(K\cap\ddot{G})=(G\cap\ddot{H})(K\cap\ddot{G})$. Hence $\ddot{G}=\ddot{H}(G\cap\ddot{G})=\ddot{H}(K\cap\ddot{G})$, which together with $K\cap\ddot{G}\leqslant\ddot{B}$ yields $\ddot{B}=(\ddot{H}\cap\ddot{B})(K\cap\ddot{G})$. Taking $\overline{\phantom{w}}$ on both sides, we then obtain~\eqref{eqn:Xia-22}. To verify condition~\eqref{prop_Ka}, assume that $G\nleqslant\ddot{G}$. Then $L=\Omega_{2m}^+(q)$, and $\GaO_{2m}^+(q)=\ddot{G}G$ as $\ddot{G}$ is a subgroup of index $2$ in $\GaO_{2m}^+(q)$. Since $G=HK$ with $H\leqslant\ddot{H}<\ddot{G}$, it follows that $\GaO_{2m}^+(q)=\ddot{G}G=\ddot{G}K$, satisfying condition~\eqref{prop_Ka}.

Conversely, suppose that $K$ satisfies condition~\eqref{prop_Ka} and~\eqref{eqn:Xia-22}. Since $K\leqslant B^{(\infty)}$, we deduce from~\eqref{eqn:Xia-22} that $\ddot{B}=(\ddot{H}\cap\ddot{B})(K\cap\ddot{G})$, and so $\ddot{G}=\ddot{H}\ddot{B}=\ddot{H}(K\cap\ddot{G})$. As $K\cap\ddot{G}\leqslant G\cap\ddot{G}$, it follows that $G\cap\ddot{G}=(G\cap\ddot{H})(K\cap\ddot{G})$. If $G\leqslant\ddot{G}$, then this already shows that $G=HK$ with $H$ taken to be $G\cap\ddot{H}$. Now assume that $G\nleqslant\ddot{G}$. Then $L=\Omega_{2m}^+(q)$, and condition~\eqref{prop_Ka} states that $\GaO_{2m}^+(q)=\ddot{G}K$. This in conjunction with $\ddot{G}=\ddot{H}(K\cap\ddot{G})$ leads to $\GaO_{2m}^+(q)=\ddot{H}K$. It then follows that $G=(G\cap\ddot{H})K$, as desired.

Next we prove that $N$ and $J$ are as described in the table of the proposition. Write $T=\la a\ra{:}\la\varphi\ra$ such that $\la a\ra=\C_{q^{sm}-1}$ is a Singer group in $\GL_m(q^s)$ and $a^\varphi=a^p$. Note that $J=\overline{\ddot{H}\cap\ddot{B}}$.

\textsf{Case~1}: $L=\SU_{2m}(q)$. In this case, $N=(\la\delta_1\ra\times\la\delta_2\ra){:}\la\phi\ra=(\C_{q+1}\times\C_{q+1}){:}\C_{2f}$ as
\[
\Nor_{\ddot{G}}\big(B^{(\infty)}\big)=\big(\GU_{2m-1}(q)\times\GU_1(q)\big){:}\la\phi\ra.
\]
Since $Z=\mathbf{Z}(\GU_{2m}(q))=\la a^{(q^{2m}-1)/(q+1)}\ra$, we have $\overline{Z}\cong Z=\C_{q+1}$. Moreover, there exists $b\in\la a\ra$ such that $b^{-1}\varphi b$ fixes $\lambda f_1$. Let $\rho=b^{-1}\varphi b$. Then $T_{\la\omega^{q-1}\ra\lambda f_1}=Z{:}\la\rho\ra$ with $\rho^{2f}\in L$ and $\overline{\rho}=\phi$. Therefore, we derive from Proposition~\ref{Xia:Unitary01} that $J=\overline{T_{\la\omega^{q-1}\ra\lambda f_1}}=\overline{Z}{:}\la\phi\ra$, and so the description of $N$ and $J$ in the table of the proposition holds.

\textsf{Case~2}: $L=\Omega_{2m}^+(q)$. In this case, $N=\C_{\gcd(2,q-1)}\times\C_{\gcd(2,q-1)}\times\C_f$, and similar argument as in Case~1 shows that $J=Z{:}\la\phi\ra=Z\times\la\phi\ra=\C_{\gcd(2,q-1)}\times\C_f$ by Proposition~\ref{Xia:Omega01}.

\textsf{Case~3}: $L=\Sp_{2m}(q)$. Then $N=\ddot{B}/B^{(\infty)}=\GaO_{2m}^-(q)/\Omega_{2m}^-(q)=\C_{2f}$, and since Remark~\ref{rem:Xia-4} shows that $|R|/|R\cap B^{(\infty)}|=q^m$, we derive from Proposition~\ref{Xia:Symplectic01} that $J=\C_{2f}$.
\end{proof}

Next, we prove Theorem~\ref{thm:Xia-5}.

\begin{proof}[Proof of Theorem~$\ref{thm:Xia-5}$]
Suppose that $G=HK$ is an exact factorization with $H$ solvable and $K$ core-free.
Let $L=\Sp_{2m}(q)$ be the socle of $G$, where $q=2^f$ is even. Then by~\cite[Theorem~3]{BL2021} we have $m\geqslant3$, $H\leqslant\Pa_m[G]$ and $K\cap L=\Omega_{2m}^-(q)$. Since $\GaO_{2m}^-(q)/\Omega_{2m}^-(q)$ is a cyclic group of order $2f$, it follows that $K=\Omega_{2m}^-(q).\mathcal{O}$ with $|\mathcal{O}|$ odd.

Let $A$ be a maximal subgroup of $G$ such that $A\cap L=\mathrm{O}_{2m}^+(q)$. Then $\bfO_2(\Pa_m[G])=U(0)\times\bfO_2(\Pa_m[A])$.
If $0\notin I$, then $H$ is contained in $A$ (up to conjugate in $G$), and so $A=H(A\cap K)$ is an exact factorization of the almost simple group $A$. However, by~\cite[Theorem~3]{BL2021}, there is no such factorization of $A$. Therefore, $0\in I$.
If $I$ contains some nonzero $i$, then since $|U(i)|\geqslant q^{m/2}$, we have $|H|_2\geqslant|U(0)||U(i)|\geqslant q^{3m/2}$, which contradicts $|H|=|G|/|K|$ as $K\geqslant\Omega_{2m}^-(q)$. Hence $I=\{0\}$, and so $H=U(0){:}S$ with $S\leqslant\GaL_1(q^m)$. In particular, $H$ is contained in a field-extension subgroup $M$ of $G$ over $\bbF_{q^m}$ such that $M\cap L=\Sp_2(q^m){:}m$. Let $B$ be a maximal subgroup of $G$ containing $K$ such that $B\cap L=\mathrm{O}_{2m}^-(q)$. Since the intersection of $M^{(\infty)}=\Sp_2(q^m)$ and $B\cap L$ is $\mathrm{O}_2^-(q^m)$, the intersection of $M^{(\infty)}$ and $B^{(\infty)}=\Omega_{2m}^-(q)$ is either $\Omega_2^-(q^m)$ or $\mathrm{O}_2^-(q^m)$.

Suppose for a contradiction that $m$ is even. Then for each $g\in\mathrm{O}_2^-(q^m)$, since the dimension of the fixed $\bbF_q$-space of $g$ is a multiple of $m$ (and hence is even), we have $g\in\Omega_{2m}^-(q)=B^{(\infty)}$. This implies that $M^{(\infty)}\cap B^{(\infty)}=\mathrm{O}_2^-(q^m)=\mathrm{D}_{2(q^m+1)}$. Then since $U(0)=q^m$ is a Sylow $2$-subgroup of $M^{(\infty)}$, it follows that $U(0)\cap B^{(\infty)}=2$. However, this is impossible as $U(0)\cap B^{(\infty)}\leqslant H\cap K=1$.

Thus we conclude that $m$ is odd. Now Theorem~\ref{thm:Xia-4} states that $S$ is transitive on $\bbF_q^m\setminus\{0\}$. To complete the proof, notice that $q^m|S|=|H|=|G|/|K|=q^m(q^m-1)|G/\Sp_{2m}(q)|/|\mathcal{O}|$ as $G=HK$ is an exact factorization. Then we obtain $|S|/(q^m-1)=|G/\Sp_{2m}(q)|/|\mathcal{O}|$, that is, the stabilizer of $S$ on $\bbF_q^m\setminus\{0\}$ has order $|G/\Sp_{2m}(q)|/|\mathcal{O}|$.
\end{proof}

\section{Applications}\label{sec:app}

In this section, we apply the classification results established so far to describe quasiprimitive permutation groups with a solvable transitive subgroup. We first prove Theorem~\ref{thm:QP}.

\begin{proof}[Proof of Theorem~$\ref{thm:QP}$]
Clearly, if $G$ has type HA or AS, then part~\eqref{thm:QPa} or~\eqref{thm:QPb} of Theorem~\ref{thm:QP} holds. Following the argument of~\cite{LPS2000} one concludes that $G$ is not of type TW and that, if $G$ has type HS or SD, then the following statements hold:
\begin{enumerate}[\rm(i)]
\item $\Soc(G)=L^2$ for some nonabelian simple group $L$;
\item there exists an almost simple group $G_0$ with socle $L$ and solvable subgroups $N_1$ and $N_2$ such that $G_0=N_1N_2$;
\item the projections $B_1$ and $B_2$ of $H\cap L^2$ into the two direct factors $L$ of $L^2$ are contained in $N_1$ and $N_2$ respectively.
\end{enumerate}
In this case, we see from~(i) that $G$ is primitive, and a combination of~(ii) and~(iii) with~\cite[Proposition~4.1]{LX2022} leads to part~\eqref{thm:QPc} of Theorem~\ref{thm:QP}.

Assume that $G$ has type HC or CD. Then $G\leqslant G_0\wr\Sy_k$ in product action for some quasiprimitive group $G_0$ of type HS or SD such that $\Soc(G)=\Soc(G_0)^k$. Let $B=G_0^k=G_1\times\cdots\times G_k$ be the base group of $G_0\wr\Sy_k$, where $G_1,\ldots,G_k$ are the $k$ direct factors of $G_0^k$, and let $\overline{H}=HB/B$ be the induced permutation group of $H$ on $G_1,\ldots,G_k$ by conjugation.
To show that part~\eqref{thm:QPd} of Theorem~\ref{thm:QP} holds in this case, it suffices to prove that, for each $i\in\{1,\ldots,k\}$, the projection $H_i$ of $H\cap B$ to $G_i$ is a transitive subgroup of $G_i$. Without loss of generality, assume that $\{G_1,\ldots,G_\ell\}$ is the orbit of $\overline{H}$ containing $G_i$. Then $H_1\times\cdots\times H_\ell$ is a transitive subgroup of $G_1\times\cdots\times G_\ell$ since $H$ is a transitive subgroup of $G$, and each of $H_1,\ldots,H_\ell$ has the same number of orbits, say, $n$. Suppose for a contradiction that $n\geqslant2$. Then $H_1\times\cdots\times H_\ell$ has $n^\ell$ orbits, and the induced permutation group of $\overline{H}$ on $\{G_1,\ldots,G_\ell\}$ is transitive on these $n^\ell$ orbits. As a consequence, $n^k$ divides $|\Sy_\ell|$. However, for each prime divisor $p$ of $n$, we have
\[
|\Sy_\ell|_p=(\ell!)_p<p^{\ell/(p-1)}\leqslant p^\ell\leqslant n_p^\ell,
\]
a contradiction.

Now assume that $G$ has type PA. Then there is a faithful action $\psi$ of $G$ on some $G$-invariant partition of $\Omega$ such that $G^\psi\leqslant G_0\wr\Sy_k$ in product action for some permutation group $G_0$ with $\Soc(G^\psi)=\Soc(G_0)^k$. Since $H$ is a transitive subgroup of $G$, it follows that $H^\psi$ is a transitive subgroup of $G^\psi$. Then the same argument as above leads to part~\eqref{thm:QPe} of Theorem~\ref{thm:QP}.
\end{proof}

The following is a consequence of the main result of~\cite{Lindstrom1972}.

\begin{lemma}\label{lem_union}
Let $S$ be a finite set, and let $\{F_i\mid i\in I\}$ be a finite set of nonempty subsets of $S$ such that $\bigcup_{i\in I}F_i=S$ and $\bigcup_{j\in J}F_j\neq S$ for any $J\subsetneq I$. Then $|I|\leq |S|$.
\end{lemma}

\begin{proof}
Suppose for a contradiction that $|I|>|S|$. Then by~\cite{Lindstrom1972}, there exist disjoint subsets $J_1$ and $J_2$ of $I$ such that $\bigcup_{j\in J_1}F_j=\bigcup_{j\in J_2}F_j$. It follows that $\bigcup_{j\in J_1}F_j\subseteq\bigcup_{j\in I\setminus J_1}F_j$ and so
\[
\bigcup_{i\in I\setminus J_1}F_i=\bigcup_{i\in I}F_i=S,
\]
a contradiction. Therefore, $|I|\leqslant|S|$.
\end{proof}

Let $m$ and $s$ be as in Table~\ref{TabHat}. Recall that if $I=\{i_1,\dots,i_k\}\subseteq\{1,\ldots,\lfloor(m+s-1)/2\rfloor\}$, then
\[
d(I)=
\begin{cases}
\gcd(2i_1-1,\ldots,2i_k-1,m)&\textup{if $s=2$},\\
\gcd(i_1,\ldots,i_k)&\textup{if $s=1$ and $m/2\in I$},\\
\gcd(i_1,\ldots,i_k,m)&\textup{if $s=1$ and $m/2\notin I$}.
\end{cases}
\]

\begin{lemma}\label{lem_DJ1}
Let $I\subseteq\{1,\ldots,\lfloor(m+s-1)/2\rfloor\}$, and for each $i\in I$ let
\[
C_i=
\begin{cases}
(\pi(m)\setminus\{2\})\setminus\pi(2i-1)&\text{if }s=2,\\
\pi(m)\setminus\pi(i)&\text{if }s=1.
\end{cases}
\]
Then the following statements hold:
\begin{enumerate}[{\rm(a)}]
\item if $s=2$ and $m$ is a $2$-power, then $d(I)=1\,\Leftrightarrow\,I\neq\varnothing$;
\item if $s=2$ and $m$ is not a $2$-power, then $d(I)=1\,\Leftrightarrow\,\bigcup_{i\in I}C_i=\pi(m)\setminus\{2\}$;
\item if $s=1$, then $d(I)=1\,\Leftrightarrow\,\bigcup_{i\in I}C_i=\pi(m)$.
\end{enumerate}
\end{lemma}

\begin{proof}
Statement~(a) is obvious. Under the condition of statement~(b), we have
\begin{align*}
d(I)=1&\;\Leftrightarrow\;\text{for each odd $p\in\pi(m)$, there exists $i\in I$ such that $p\nmid 2i-1$}\\
&\;\Leftrightarrow\;\text{for each $p\in\pi(m)\setminus\{2\}$, there exists $i\in I$ such that $p\in C_i$}\\
&\;\Leftrightarrow\;\bigcup_{i\in I}C_i=\pi(m)\setminus\{2\},
\end{align*}
proving statement~(b). For statement~(c), assume $s=1$ and write $I=\{i_1,\dots,i_k\}$.

\textsf{Case~1}: $m/2\in I$. In this case, $m$ is even, $d(I)=\gcd(i_1,\ldots,i_\ell)$, and
\begin{equation}\label{eq:d=1}
d(I)=1\;\Leftrightarrow\;\text{for each $p\in\pi(m/2)$, there exists $i\in I \setminus\{m/2\}$ such that $p\nmid i$}.
\end{equation}
If $m\equiv0\pmod{4}$, then $\pi(m)=\pi(m/2)$ and $C_{m/2}=\pi(m)\setminus\pi(\frac{m}{2})=\varnothing$, whence~\eqref{eq:d=1} gives
\[
d(I)=1\;\Leftrightarrow\;\bigcup_{i\in I\setminus\{m/2\}}C_i=\pi(m/2)\;\Leftrightarrow\;\bigcup_{i\in I }C_i=\pi(m).
\]
If $m\equiv2\pmod{4}$, then $C_{m/2}=\pi(m)\setminus\pi(m/2)=\{2\}$, and so~\eqref{eq:d=1} implies
\[d(I)=1\;\Leftrightarrow\;\bigcup_{i\in I\setminus\{m/2\}}(C_i\setminus\{2\})=\pi(m/2)\;\Leftrightarrow\;\bigcup_{i\in I }C_i=\pi(m).
\]

\textsf{Case~2}: $m/2\notin I$. In this case, $d(I)=\gcd(i_1,\ldots,i_\ell,m)$, and so
\begin{align*}
d(I)=1 &\;\Leftrightarrow\;\text{for each $p\in\pi(m)$, there exists $i\in I$ such that $p\nmid i$}\\
&\;\Leftrightarrow\;\bigcup_{i\in I}C_i=\pi(m).
\end{align*}
This completes the proof.
\end{proof}

Lemma~\ref{lem_DJ1} enables us to establish the next two lemmas.

\begin{lemma}\label{prop_mindI1}
Let $\mathcal{I}=\{I\mid I\textup{ is minimal subject to }I\subseteq\{1,\ldots,\lfloor(m+s-1)/2\rfloor\}\text{ and }d(I)=1\}$,
and let $M=\max\{|I|\mid I\in\mathcal{I}\}$. Then
\[
M=
\begin{cases}
1 &\text{if $s=2$ and $m$ is a power of $2$},\\
|\pi(m)\setminus\{2\}| &\text{if $s=2$ and $m$ is not a power of $2$},\\
|\pi(m)| &\text{if $s=1$}.
\end{cases}
\]
\end{lemma}

\begin{proof}
Take an arbitrary $I\in\mathcal{I}$. Apply Lemma~\ref{lem_DJ1} and adopt the notation there. If $s=2$ and $m$ is a power of $2$, then clearly $M=1$.

Next assume that $s=2$ and $m$ is not a power of $2$. Then Lemma~\ref{lem_DJ1} implies that for each $i\in I$ we have $C_i\neq \varnothing$, and for each $J\subsetneq I$ we have $\bigcup_{j\in J}C_j\neq\pi(m)\setminus\{2\}$. Hence Lemma~\ref{lem_union} asserts $|I|\leqslant |\pi(m)\setminus\{2\}|$. To prove $M=|\pi(m)\setminus\{2\}|$, we construct some $I\in\mathcal{I}$ with $|I|=|\pi(m)\setminus\{2\}|$.
Write $\pi(m)\setminus\{2\}=\{p_1,\ldots,p_t\}$ and $\theta=p_1\cdots p_t$. Consider
\[
I=\left\{\frac{1}{2}\left(\frac{\theta}{p_r}+1\right)\,\middle|\;r\in\{1,\ldots,t\}\right\}.
\]
Then $I\subseteq\{1,\ldots,\lfloor(m+s-1)/2\rfloor\}$,
\[
\bigcup_{i\in I}C_i=\bigcup_{r=1}^t\{p_r\}=\pi(m)\setminus\{2\},
\]
and it holds for each $J\subsetneq I$ that $\bigcup_{j\in J}C_j\neq\pi(m)\setminus\{2\}$.
By Lemma~\ref{lem_DJ1}, this implies $I\in\mathcal{I}$.

Now assume that $s=1$. In the same vein as above, we have $M\leqslant|\pi(m)|$ and that, if $\pi(m)=\{p_1,\ldots,p_t\}$ and $\theta=p_1\cdots p_t$, then the set $\{\theta/p_r\mid r\in\{1,\ldots,t\}\}$ belongs to $\mathcal{I}$. Thus, $M=|\pi(m)|$.
\end{proof}

\begin{lemma}\label{prop_mindI2}
Let $m\geqslant 4$ be even, and let $s=1$. Then
\[
\max\{|I|\mid I\textup{ is minimal subject to }I\subseteq\{1,\ldots,m/2\}\text{ and }d(I)=2\}=|\pi(m/2)|.
\]
\end{lemma}

\begin{proof}
By~\eqref{eq:d}, a necessary condition for $I\subseteq\{1,\ldots,m/2\}$ to satisfy $d(I)=2$ is that the numbers in $I$ are all even. In this case, let $J=\{i/2\mid i\in I\}$. Then $J\subseteq\{1,\ldots,\lfloor m/4\rfloor\}$, and we see that $I$ is minimal subject to $I\subseteq\{1,\ldots,m/2\}$ and $d(I)=2$ if and only if $J$ is minimal subject to $J\subseteq\{1,\ldots,\lfloor m/4\rfloor\}$ and $d(J)=1$. Hence Lemma~\ref{prop_mindI1} implies that the maximum size of such a minimal $I$ is $|\pi(m/2)|$.
\end{proof}

Lemmas~\ref{prop_mindI1} and~\ref{prop_mindI2} are crucial in determining solvable minimally transitive subgroups of almost simple permutation groups. In particular, they play a key role in the proof of statement~\eqref{prop_smallb} of the following result.

\begin{proposition}\label{prop_small}
Let $G$ be an almost simple permutation group on $n$ points such that $G\ngeqslant\A_n$, and let $H$ be a solvable transitive subgroup of $G$. Then the following statements hold:
\begin{enumerate}[\rm(a)]
\item $|H|<n^{\log_2n}$ for sufficiently large $n$;
\item\label{prop_smallb} if $G$ is primitive and $H$ is minimally transitive, then either $\ln|H|=O(\ln n\ln\ln n/\ln\ln\ln n)$ or $\Soc(G)=\Omega_{2m+1}(q)$ with $q$ odd.
\end{enumerate}
\end{proposition}

\begin{proof}
Let $K$ be a point stabilizer in $G$. Then $G=HK$ with $H$ solvable.
For non-classical groups $G$, one can directly verify by~\cite[Theorem~1.1]{LX2022} that $|H|$ is bounded above by a polynomial of $n$, and so both statements of the proposition are true. For classical groups, the same holds for rows~1--2, 4--5 and~9 of Table~\ref{TabLX2022} (the table lists corresponding groups in the quasisimple group $L$ such that $L/\Z(L)=\Soc(G)$). Thus we only need to deal with rows~3 and~6--8 of Table~\ref{TabLX2022}. It is well known (see for instance~\cite[\S\,I.5.3]{Tenenbaum2015}) that
\begin{equation}\label{eq:pi}
|\pi(m)|=O\left(\frac{\ln m}{\ln\ln m}\right).
\end{equation}

First assume that row~3 of Table~\ref{TabLX2022} appears. In this case, $L=\Sp_{2m}(q)$ with $q$ even,
\[
|H|\leqslant q^{m(m+1)/2}|\GaL_1(q^m)|<q^{m^2}\ \text{ and }\ n\geqslant\frac{|\Sp_{2m}(q)|}{|\GO^+_{2m}(q)|}=\frac{q^m(q^m+1)}{2}>q^m.
\]
Consequently, $m<\log_qn\leqslant\log_2n$, and $|H|<(q^m)^m<n^m<n^{\log_2n}$, which proves statement~(a). Moreover, for bounded $m$, the order $|H|$ is bounded above by a polynomial of $n$. Hence we may assume that $m$ is large enough, say, $m\geqslant4$.
To prove statement~(b), suppose that $G$ is primitive and $H$ is minimally transitive. Then we derive from Theorem~\ref{thm:Xia-4} that $|H|\leqslant|U(I)||\GaL_1(q^m)|$ with one of the following:
\begin{enumerate}[\rm(i)]
\item $I=\{0\}$;
\item $q=2$ or $4$, and $I$ is minimal subject to $I\subseteq\{1,\ldots,\lfloor m/2\rfloor\}$ and $d(I)=1$;
\item $q=2$, and $I$ is minimal subject to $I\subseteq\{1,\ldots,\lfloor m/2\rfloor\}$ and $d(I)=2$.
\end{enumerate}
In case~(i), according to~\eqref{eq:Ui}, we have $|U(I)|=q^m$. In case~(ii) or~(iii), we obtain from~\eqref{eq:UI} and Lemmas~\ref{prop_mindI1} and~\ref{prop_mindI2} that $|U(I)|\leqslant q^{m|I|}\leqslant q^{m|\pi(m)|}$. Hence it holds for all the three cases that
\[
|H|\leqslant|U(I)||\GaL_1(q^m)|\leqslant q^{m|\pi(m)|}|\GaL_1(q^m)|<q^{m(|\pi(m)|+2)}=n^{|\pi(m)|+2},
\]
which in conjunction with~\eqref{eq:pi} leads to $\ln|H|=O(\ln n\ln\ln n/\ln\ln\ln n)$.

Next assume that row~6 of Table~\ref{TabLX2022} appears. In this case, $L=\SU_{2m}(q)$,
\[
|H|\leqslant q^{m^2}|\GaL_1(q^{2m})|<q^{m^2+4m}\ \text{ and }\ n\geqslant\frac{|\SU_{2m}(q)|}{|\GU_{2m-1}(q)|}=\frac{q^{2m-1}(q^{2m}-1)}{q+1}>q^{2m},
\]
whence $|H|<n^{\log_2n}$. Suppose that $G$ is primitive and $H$ is minimally transitive. Then by Theorem~\ref{thm:Xia-2}, $|H|\leqslant|U(I)||\GaL_1(q^{2m})|$ such that $I$ is minimal subject to $I\subseteq\{1,\ldots,\lfloor(m+1)/2\rfloor\}$ and $d(I)=1$.
Thus, by~\eqref{eq:UI} and Lemma~\ref{prop_mindI1},
\[
|H|\leqslant|U(I)||\GaL_1(q^{2m})|\leqslant q^{2m|\pi(m)|}|\GaL_1(q^{2m})|<q^{2m(|\pi(m)|+4)}=n^{|\pi(m)|+4},
\]
from which we deduce $\ln|H|=O(\ln n\ln\ln n/\ln\ln\ln n)$ by~\eqref{eq:pi}.

Now let $L=\Omega_{2m+1}(q)$ with $q$ odd, as in row~7 of Table~\ref{TabLX2022}. Then we have
\[
|H|\leqslant q^{m(m+1)/2}|\GaL_1(q^m)|<q^{m^2}\ \text{ and }\ n\geqslant\frac{|\Omega_{2m+1}(q)|}{|\Omega_{2m}^-(q).2|}=\frac{q^m(q^m-1)}{2}>q^m,
\]
and so $|H|<n^{\log_2n}$.

Finally, the argument for row~8 of Table~\ref{TabLX2022} is the similar to that of row~3 by replacing the application of Theorem~\ref{thm:Xia-4} to that of Theorem~\ref{thm:Xia-3}. The proof is thus complete.
\end{proof}

\begin{remark}\label{rem:Xia-5}
In statement (b) of Proposition~\ref{prop_small}, the case $\Soc(G)=\Omega_{2m+1}(q)$ does not satisfy $\ln|H|=O(\ln n\ln\ln n/\ln\ln\ln n)$. In this case, $n=|\Omega_{2m+1}(q)|/|\Omega_{2m}^-(q).2|<q^{2m}$, and according to~\cite[Remark~2]{BL2021}, the solvable minimally transitive subgroup $H$ contains the full unipotent radical $q^{m(m-1)/2}.q^m$ and hence has order at least $n^{m/4}>n^{\log_qn/8}$.
\end{remark}

We conclude the paper by proving Corollary~\ref{thm:small}.

\begin{proof}[Proof of Corollary~$\ref{thm:small}$]
By Theorem~\ref{thm:QP}, one of cases~\eqref{thm:QPa}--\eqref{thm:QPe} there holds.
As proved in Proposition~\ref{prop_small}, there exist absolute constants $a>1$ and $A>1$ such that each pair $(G,H)$ in Theorem~\ref{thm:QP}\,\eqref{thm:QPb} satisfies
\[
|H|<
\begin{cases}
n^{\log_2n}&\text{if }n>a,\\
A&\text{if }n\leqslant a.
\end{cases}
\]
Hence the conclusion of the corollary in Theorem~\ref{thm:QP}\,\eqref{thm:QPb} follows immediately.

First assume that $G$ is primitive of type HA, as in case~\eqref{thm:QPa} of Theorem~\ref{thm:QP}. Then $n=p^d$ for some prime $p$ and positive integer $d$, and the stabilizer $H_\omega$ in $H$ of any point $\omega$ is a solvable subgroup of $\GL_d(p)$. As a consequence, $|H_\omega|_p\leqslant|\GL_d(p)|_p=p^{d(d-1)/2}$. Moreover, by~\cite[Lemma~19]{Kazarin1986}, there exists some absolute constant $\alpha$ in the interval $(11/5,9/4)$ such that a Hall $p'$-subgroup of $H_\omega$ has order $|H_\omega|/|H_\omega|_p\leqslant24^{-1/3}p^{d^\alpha}\gcd(d,p-1)$. Hence
\[
|H|=n|H_\omega|\leqslant24^{-1/3}p^{d+d^\alpha+d(d-1)/2}\gcd(d,p-1)<p^{1+d^\alpha+d(d+1)/2}=n^\frac{1+d^\alpha+d(d+1)/2}{d}.
\]
As $d=\log_pn\leqslant\log_2n$, we derive that $\ln|H|=O\big((\ln n)(\log_2n)^{\alpha-1}\big)=O((\ln n)^\alpha)$.

Next assume that Theorem~\ref{thm:QP}\,\eqref{thm:QPc} holds. When $n=|L|$ is sufficiently large, $L=\PSL_2(q)$ with $q=p^f\geqslant4$ for some prime $p$ and positive integer $f$, and
\[
|H|\leqslant2|\mathrm{Out}(L)||M_1||M_2|\leqslant2f\gcd(2,q-1)\cdot\frac{2(q+1)}{\gcd(2,q-1)}\cdot\frac{q(q-1)}{\gcd(2,q-1)}<q^5<n^2.
\]

By the above paragraph, there exist absolute constants $b>1$ and $B>4$ such that each pair $(G,H)$ in Theorem~\ref{thm:QP}\,\eqref{thm:QPc} satisfies
\[
|H|<
\begin{cases}
n^2&\text{if }n>b,\\
B&\text{if }n\leqslant b.
\end{cases}
\]
Now let $(G,H)$ be a pair in Theorem~\ref{thm:QP}\,\eqref{thm:QPd}. Then $n=n_0^k$, and $H\cap G_0^k\leqslant H_1\times\cdots\times H_k$ such that each $(G_0,H_i)$ is a pair $(G,H)$ in Theorem~\ref{thm:QP}\,\eqref{thm:QPc}.
If $n>b^k$, then $n_0>b$, and so $|H\cap G_0^k|<(n_0^2)^k=n^2$. If $n\leqslant b^k$, then $n_0\leqslant b$, and so
\[
|H\cap G_0^k|<B^k=\big(2^{\log_2B}\big)^k\leqslant\big(n_0^{\log_2B}\big)^k=n^{\log_2B}.
\]
In either case, $|H\cap G_0^k|<n^{\log_2B}$. Notice that the induced subgroup $\overline{H}$ of $\Sy_k$ on the $k$ copies of $G_0$ in the base group $G_0^k$ satisfies $|\overline{R}|\leqslant24^{(k-1)/3}$ by~\cite[Theorem~3]{Dixon1967}. This implies that
\[
|H|=|\overline{H}||H\cap G_0^k|<24^{(k-1)/3}n^{\log_2B}<2^{2k}n^{\log_2B}\leqslant(n_0)^{2k}n^{\log_2B}=n^{2+\log_2B}.
\]

Finally, assume that $(G,H)$ is in Theorem~\ref{thm:QP}\,\eqref{thm:QPe}. Let $n'$ be the number of parts in the $G$-invariant partition of $\Omega$ such that $G^\psi\leqslant\Sy_{n'}$. Then along the similar lines as in the previous paragraph, we obtain
\[
|H^\psi\cap G_0^k|<
\begin{cases}
\big(n_0^{\log_2n_0}\big)^k=\big((n')^{\log_2n'}\big)^{1/k}\leqslant\big((n')^{\log_2n'}\big)^{1/2}&\text{if }n'>a^k,\\
A^k=\big(2^{\log_2A}\big)^k\leqslant\big(n_0^{\log_2A}\big)^k=(n')^{\log_2A}&\text{if }n'\leqslant a^k,
\end{cases}
\]
and hence
\[
|H^\psi|\leqslant24^{(k-1)/3}|H^\psi\cap G_0^k|<2^{2k}|H^\psi\cap G_0^k|\leqslant(n_0)^{2k}|H^\psi\cap G_0^k|=(n')^2|H^\psi\cap G_0^k|\ll(n')^{\log_2n'}.
\]
Then since $n'\leqslant n$, it follows that $|H|=|H^\psi|\ll n^{\log_2n}$. This completes the proof.
\end{proof}

\section*{Acknowledgments}
Tao Feng acknowledges the support of NNSFC grant no.~12225110. Cai Heng Li and Binzhou Xia acknowledge the support of NNSFC grant no.~11931005. Conghui Li acknowledges the support of NNSFC grant no.~12271446 and Sichuan Science and Technology Program no.~2024NSFJQ0070. Lei Wang acknowledges the support of NNSFC grant no.~12061083. Hanlin Zou acknowledges the support of NNSFC grant no.~12461061.

\end{document}